\documentclass[10pt]{amsart}
\usepackage{geometry}                
\usepackage[foot]{amsaddr}       
\usepackage{graphicx}
\usepackage{amssymb}
\usepackage{amsmath}
\usepackage{color}
\usepackage{amsthm}
\usepackage{accents}
\usepackage{bm}
\usepackage{microtype}
\usepackage{booktabs}
\usepackage{epstopdf}
\usepackage{stmaryrd}
\usepackage{hyperref}
\usepackage{cleveref} 
\usepackage{subcaption} 
\usepackage{siunitx}
\usepackage{multirow}
\usepackage{tikz}
\usetikzlibrary{decorations.markings}
\usetikzlibrary{arrows}
\usetikzlibrary{patterns}
\newcommand{\pderivative}[2]{\frac{\partial #1}{\partial #2}}
\renewcommand{\vec}[1]{\mathbf{#1}}
\newcommand{\mat}[1]{\mathbf{#1}}
\newcommand{\jump}[1]{\left\llbracket #1 \right\rrbracket}
\newcommand{\average}[1]{\left\{\!\left\{#1\right\}\!\right\}}
\newcommand{\avg}[1]{\left\{\!\left\{#1\right\}\!\right\}}
\newcommand{\halb}{\frac{1}{2}}
\newcommand{\half}{\frac{1}{2}}
\newcommand{\fhat}{\ensuremath{\textit{\textbf{\^{f} \!}}}}
\renewcommand{\vec}[1]{\ensuremath{\boldsymbol #1}}
\newcommand{\He}{\ensuremath{\bm{\mathcal{H}}}}
\renewcommand{\H}{\ensuremath{\bm{\widehat{\mathcal{H}}}}}

\newcommand{\shouldbe}{\ensuremath{\stackrel{!}{=}}}
\newcommand{\rholn}{\ensuremath{\rho\textsuperscript{ln}}}
\newcommand{\betaln}{\ensuremath{\beta\textsuperscript{ln}}}
\newcommand{\pln}{\ensuremath{p\textsuperscript{ln}}}
\newcommand{\pavg}{\ensuremath{\average{p}}}
\newcommand{\uavg}{\overline{\lVert \vec{u} \rVert^2}}
\newcommand{\betaavg}{\ensuremath{\overline{\beta^2}}}

\newcommand{\Eline}{\ensuremath{\overline{E}}}

\newcommand{\R}{\ensuremath{\bm{\widehat{\mathcal{R}}}}}
\newcommand{\T}{\ensuremath{\bm{\widehat{\mathcal{Z}}}}}
\newcommand{\rotMat}{\ensuremath{\bm{\mathcal{S}}}}
\let\rho\varrho
\newlength{\dhatheight}
\newcommand{\doublehat}[1]{%
    \settoheight{\dhatheight}{\ensuremath{\hat{#1}}}%
    \addtolength{\dhatheight}{-0.35ex}%
    \hat{\vphantom{\rule{1pt}{\dhatheight}}%
    \smash{\hat{#1}}}}

\theoremstyle{plain}
\newtheorem{thm}{Theorem}
\newtheorem{lem}{Lemma}
\theoremstyle{remark}
\newtheorem{rem}{Remark}
\numberwithin{equation}{section}
\begin{document}
%
\title{Entropy Stable Finite Volume Approximations for Ideal Magnetohydrodynamics}
\author[Derigs]{Dominik Derigs$^1$}
\author[Gassner]{Gregor J. Gassner$^2$}
\author[Walch]{Stefanie Walch$^1$}
\address{$^1$I.\,Physikalisches Institut, Universit\"at zu K\"oln, Z\"ulpicher Stra\ss{}e~77, 50937 K\"oln, Germany}
\author[Winters]{Andrew R. Winters$^{2,*}$}
\address{$^2$Mathematisches Institut, Universit\"at zu K\"oln, Weyertal 86-90, 50931 K\"oln, Germany}
\email{$^*$awinters@math.uni-koeln.de}
\maketitle
%
%
\begin{abstract}
This article serves as a summary outlining the mathematical entropy analysis of the ideal magnetohydrodynamic (MHD) equations. We select the ideal MHD equations as they are particularly useful for mathematically modeling a wide variety of magnetized fluids. In order to be self-contained we first motivate the physical properties of a magnetic fluid and how it should behave under the laws of thermodynamics. Next, we introduce a mathematical model built from hyperbolic partial differential equations (PDEs) that translate physical laws into mathematical equations. After an overview of the continuous analysis, we thoroughly describe the derivation of a numerical approximation of the ideal MHD system that remains consistent to the continuous thermodynamic principles. The derivation of the method and the theorems contained within serve as the bulk of the review article. We demonstrate that the derived numerical approximation retains the correct entropic properties of the continuous model and show its applicability to a variety of standard numerical test cases for MHD schemes. We close with our conclusions and a brief discussion on future work in the area of entropy consistent numerical methods and the modeling of plasmas.
\end{abstract}


\section{Introduction}

The focus of this summary article is to offer a comprehensive overview of mathematical entropy analysis applied to the equations of ideal magnetohydrodynamics (MHD). As such, we will motivate each step of the mathematical analysis with physical properties, and, where applicable, with everyday experience. In particular, we will discuss the three laws of thermodynamics and call their consequences for the mathematical model in question. Our analysis is first carried out at the continuous level, followed by a rigorous examination of the transition to the discrete level to demonstrate how grid-based numerical algorithms can be designed. The discrete form of the mathematical model will allow us to capture the physical structures and thermodynamic principles of a system of partial differential equations on a computer. By collecting the relevant physical and mathematical analysis into a single work we aim to demonstrate the development of thermodynamically consistent numerical approximations.

The world is full of complex (magneto-)hydrodynamic processes one wishes to understand. From the hurricanes that can devastate communities, the weather forecast for a given day, up to the formation and evolution of the Universe. Over the last two centuries, the conjoined scientific fields of physics and applied mathematics have provided a litany of tools to reproduce the world around us. The breadth of topics covered by the joint effort of these two disciplines is far reaching, from predicting the implication of quantum effects to cosmological simulations reproducing the distribution of galaxies in our Universe.

Fluid dynamics has a wide range of applications in science and engineering, including but not limited to, the prediction of noise and drag of modern aircraft, important for future mobility; the behavior of gas clouds, important for the understanding of the interstellar medium and star formation; the build up and propagation of tsunamis in oceans, important for prediction and warning systems; the impact of the solar wind on Earth, important for an uninterrupted operation of communication satellites, and predictions of instabilities in plasma, important for potential future energy sources like nuclear fusion power plant design.

In this work, we focus on the evolution of macroscopic fluid flows. For example, flows around an airfoil or the formation of new stars out of the diffuse gas in our Galaxy. One classifies a particular fluid into two important groups depending on the level of variation in the fluid's density: \textit{Incompressible} fluids have a density that remains \textit{nearly constant} as the fluid evolves, e.g., liquid water at a particular temperature; \textit{Compressible} fluids have a density that \textit{varies} as the fluid evolves, e.g., air flow around an automobile \cite{Frank2016,Landau1959} or highly compressible gas forming the interstellar medium around us \cite{Gatto2017,Girichidis2016,Walch2014}. Due to the multitude of spatial and temporal scales involved with such fluid flows, laboratory experiments are very often impractical or even impossible to carry out. So we move away from a classical experimental setting and instead focus on mathematical models that can predict such complicated flow behaviors accurately. Interestingly, one has to invest much work into the mathematical model to ensure that the solutions generated actually remain physically relevant. Therefore, the applied mathematics we use and describe herein is built from a perspective that respects several fundamental physical principles by construction.

The remainder of this work is organized as follows: Sec.~\ref{sec:physics} outlines the background physical principles. Sec.~\ref{sec:PDEs} highlights how the physical model is represented in a mathematical context. Then, Sec.~\ref{sec:contEnt} describes the mathematical analysis of entropy on the continuous level. We present the equations of ideal MHD, their usefulness and an entropy analysis in Sec.~\ref{sec:idealMHD}. Sec.~\ref{sec:finiteVolume} describes one framework to design a numerical approximation that can approximate complicated flow configurations. The bulk of the work and main theorems are contained in Secs.~\ref{sec:ECIdealMHD}--\ref{sec:discEvaluation}, where we demonstrate how to derive a numerical method that is made to respect the thermodynamic properties of the governing equations. We present the utility of thermodynamically consistent algorithms with numerical results in Sec.~\ref{sec:numExp}. Finally, our conclusions and a brief outline of future research topics for the ideal MHD equations and mathematical entropy are given in Sec.~\ref{sec:conclusions}.

\section{A fairly brief introduction to thermodynamics}\label{sec:physics}

\textit{Thermodynamics} is a branch of physics relating the heat, temperature, or entropy of a given physical system to energy and work. 
%
%
The definitions of most physical quantities used in the three laws of thermodynamics analysis are intuitive, but the entropy is a slightly esoteric quantity as we do not use this quantity in everyday life -- we have no ``feeling'' for it. From a mathematical point of view, the entropy is defined as
\begin{equation}
S = \frac{\theta}{T},
\end{equation}
where $\theta$ is the specific inner energy of the system and $T$ is the temperature of the system. For now, we will postpone any further details on the implications of entropy and thoroughly discuss it in Section~\ref{sec:contEnt}. 

The three laws of thermodynamics are some of the most important laws in all of physics, particularly because they separate \textit{possible} physical processes, e.g., heat diffusion in a liquid, from the \textit{impossible}, e.g., perpetual motion:
\begin{description}
\item[First law of thermodynamics] Energy can neither be created nor destroyed. It can only change forms. For any physical process, the total energy of the closed system remains conserved.\vspace{0.6cm}
\item[Second law of thermodynamics] The entropy of a closed physical system not in equilibrium tends to increase over time, approaching its maximum value at equilibrium. It cannot shrink.\vspace{0.6cm}
\item[Third law of thermodynamics] It is impossible for any process, no matter how idealized, to reduce the temperature of a system to absolute zero in a finite number of steps.
\end{description}

All three laws of thermodynamics must be satisfied at the same time, as otherwise a fluid could exhibit strange, i.e., unintuitive (and obviously wrong) behavior. For example, a fluid governed solely by the first law of thermodynamics, i.e., conservation of total energy, could transfer all of its internal energy into kinetic energy. The result would be a very fast, very cold jet of air. This physical state is allowed under total energy conservation, yet has never been observed in nature. To explain this discrepancy requires the second law of thermodynamics. Under the second law, certain transfers between internal energy and kinetic energy are not possible (see also Fig.~\ref{fig:entropy}).

Although most thermodynamics calculations use only entropy differences and the zero point of the entropy scale is of no importance, we still mention the third law for purposes of completeness. Interestingly enough, the third law describes the minimum temperature (absolute zero temperature) as being unattainable in somewhat the same way as the maximums speed (speed of light). Special relativity states (and experiments have shown) that no matter how fast something is moving, it can always be accelerated, but it can never reach the speed of light. In a similar way, a system can always be cooled further down, but can never reach absolute zero.

We distinguish between reversible and irreversible thermodynamic processes.
First, if work is applied to increase the fluid's kinetic energy at constant pressure and temperature, it can be completely recovered by decelerating the flow. This is called an isentropic\footnote{An isentropic process is a physical process in which the entropy remains constant over time.} acceleration and/or deceleration.
Second, if work is done to compress or expand the fluid isentropically, then the volume and therefore the temperature and pressure of the fluid change while its kinetic energy remains constant. Through such processes, the kinetic energy and internal energy (temperature) can be independently and reversibly varied. Other \textit{reversible} state changes are a combination of these two processes. For example, the work required for an isentropic compression can be provided by an isentropic deceleration.
In addition to these reversible state changes, there are \textit{irreversible} state changes. For example, heat transfer across a finite temperature difference is irreversible as is viscous momentum transfer. Irreversibility implies that certain state changes (like the transferring of heat along a temperature gradient) cannot be reversed. The reversed processes of irreversible processes (e.g., like the transferring of heat \textit{against} a temperature gradient) are \textit{unphysical} and have never been observed in the laboratory at any macroscopic scales.

\begin{figure}[h]
	\centering
	\includegraphics{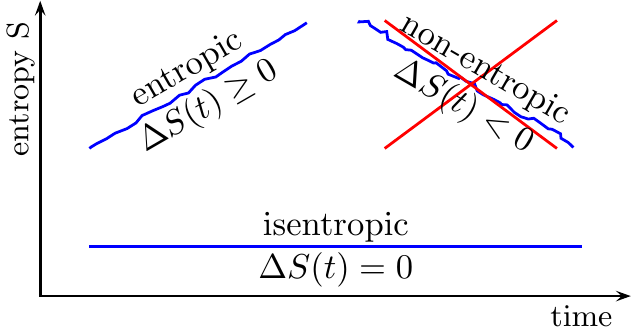}
	\caption{Possible and impossible behavior of the fluid system judged by the second law of thermodynamics. Reversible processes conserve entropy over time (``isentropic''). In contrast, irreversible processes generate entropy over time (``entropic''). ``Non-entropic'' processes, i.e., processes that reduce the entropy are forbidden by the second law of thermodynamics in closed systems. They are never observed in nature on any macroscopic scale.}
	\label{fig:entropy}
\end{figure}

Hence, the three laws of thermodynamics provide an important role in selecting the tiny subset of \textit{physically} feasible solutions from all imaginable state configurations. The first law of thermodynamics ensures that a closed system never has more or less total energy than it did at the beginning of its evolution. It brings into play the idea that energy cannot be destroyed, but can only be converted into different forms of itself (kinetic, thermal, magnetic, etc.) inside the system. The second law of thermodynamics provides a mechanism to distinguish further between possible state changes and the impossible ones. The fundamental concept of the entropy of a system is introduced to characterize such processes. For reversible processes the rate of change of the total system entropy with respect to time is zero, i.e., entropy is conserved for reversible state configurations. For irreversible processes, the entropy increases; fluid dynamics where the total system entropy shrinks in time (i.e.~reversed processes of irreversible processes) are never observed. Finally, the third law of thermodynamics provides a limitation for the temperature - it can approach zero but never reach it, regardless of the cooling mechanism and the duration of the cooling.

\noindent\textbf{Mission:} \textit{We now see the importance of the laws of thermodynamics and how entropy separates between possible and impossible realizations. So, the goal and focus of the remainder of this article are clear. We must first create a systemic approach to discuss the laws of thermodynamics in the language of mathematics on the continuous level. Then, we will mimic the continuous analysis on the discrete level such that we develop numerical algorithms that guarantee approximate solutions to the mathematical models obey the correct entropic behavior.}

\section{A fairly brief introduction to hyperbolic partial differential equations}\label{sec:PDEs}

We seek alternatives in theoretical predictions to overcome the limitations and tractability of laboratory experiments, which are often too costly, too time-consuming, too dangerous, or are even impossible to perform as they surpasses the capabilities of our technology. This is because the behavior of flows such as these are of utmost concern for modern industry, science, medicine and society.

It is possible for physicists to interpret the fundamental laws of physics, e.g., the first law of thermodynamics, into corresponding mathematical equations. Then another branch of the natural sciences, \textit{applied mathematics} is concerned with developing methods to solve the given system of equations. The combined effort of both fields enables us to make predictive statements about how physical systems evolve in time.
In particular, we are interested in mathematical models built from partial differential equations (PDEs). A PDE is an equation that involves the rate of change of unknown multivariable function(s) with respect to several continuous variables. To pose the PDE properly also requires information about the initial configuration of the physical system (\textit{initial condition}) as well as the possible prescription of how the physics should behave at any boundary (\textit{boundary conditions}), e.g., the surface of an airfoil \cite{evans2010}. 

Emmy Noether found an astounding relationship that \textit{symmetry} in a given physical law implies that there must be a conserved quantity associated with that law \cite{kosmann2011,noether1918}. We note that symmetry in this context refers to a transformation, e.g, rotational or translational, which \textbf{does not} alter the behavior of the physical system. A simple example is if we require that the results of an experiment performed today should remain the same if the experiment were performed tomorrow, i.e., symmetry in time, we arrive at the conservation of total energy. A consequence of Noether's theorem is that we can build \textit{mathematical conservation laws} (which are governed by PDEs) and immediately relate them to a particular set of fundamental physical laws. This offers a powerful link between physics and mathematics as well as provides a multitude of practical computational tools for studying fluid mechanics.

Conservation laws that govern fluid dynamics are often members of the class of \textit{hyperbolic PDEs}. A standard prototype of a hyperbolic PDE is the advection equation
\begin{equation}\label{eq:1DAdvec}
\pderivative{u}{t} + a\pderivative{u}{x} = 0,
\end{equation}
where $a$ is a constant, $t$ is the temporal variable, and $x$ is the spatial variable. Without loss of generality we assume $a>0$. Given an initial condition for the function
\begin{equation}
u_0(x):=u(x,0) ,\quad\forall x\in\mathbb{R},
\end{equation}
we have an \textit{initial value problem} where we want to determine the values of $u(x,t)$ for positive values of time. We see that a possible solution to \eqref{eq:1DAdvec} is
\begin{equation}\label{eq:1DSol}
u(x,t) = u_0(x-at),
\end{equation}
where it is known that \eqref{eq:1DSol} is actually the unique solution of \eqref{eq:1DAdvec} \cite{evans2010}. The formula \eqref{eq:1DSol} tells us two important things. First, the solution at any time $t^*$ is the original initial condition shifted to the right (as $a>0$) by the amount $at^*$. In other words, the solution at a point $(x,t)$ depends only on the value of $\xi=x-at$. The lines in the $(x,t)$ plane on which $\xi$ is a constant are referred to as \textit{characteristics}. The parameter $a$ is the speed of propagation of the solution along each characteristic line. Thus, the solution of the advection equation \eqref{eq:1DAdvec} can be regarded as a wave that propagates with speed $a$ without changing its shape. Second, it appears that the solution of \eqref{eq:1DAdvec} only makes sense if $u$ is differentiable. However, the general solution \eqref{eq:1DSol} prescribes no differentiability constraint on the initial condition $u_0$. In general, this means that \textit{discontinuous solutions are allowed} for hyperbolic conservation laws \cite{evans2010}. An example of such discontinuous solutions are \textit{shock waves}, which are a feature of non-linear hyperbolic conservation laws, that occur when two characteristic curves cross \cite{evans2010,strauss1992,toro2009}.

Many of the complicated, compressible fluid applications discussed in the previous section are modeled by the solution of \textit{non-linear conservation laws}. However, it is precisely the possibility of discontinuous solutions of {non-linear} hyperbolic PDEs that introduces a challenge when computing their solution. To be explicit we now consider a general one-dimensional system of conservation laws of the form
\begin{equation}\label{eq:1DCons}
\pderivative{\vec{q}}{t} + \pderivative{\vec{f}(\vec{q})}{x} = \vec{0},\quad \vec{q}_0(x):=\vec{q}(x,0) ,
\end{equation}
where $\vec{q}$ is the vector of conserved variables and $\vec{f}(\vec{q})$ is the non-linear vector flux. We know that regardless of the continuity of the initial data it is possible for the solution of a non-linear problem to develop discontinuities, e.g., shocks, in finite time \cite{lax1954,sod1978}. Therefore, solutions of the system \eqref{eq:1DCons} are sought in the {weak sense}. A \textit{weak solution} (sometimes called a generalized solution) of a PDE is a function for which the derivative may not exist everywhere, but nonetheless satisfies the equation in some sense \cite{evans2010}. Unfortunately, weak solutions of a PDE are, in general, not unique and must be supplemented with extra admissibility criteria in order to single out the physically relevant solution \cite{dafermos2000,harten1983,kruzkov1970,lax2005,mock1980,Tadmor1987_2,Tadmor2003}.

We previously discussed that the laws of thermodynamics provide guidelines in separating the possible state changes from the impossible. In the mathematical sense, solutions of reversible processes are smooth and differentiable. However, discontinuous solutions arise due to irreversible processes, such as within shock waves, where entropy must increase. It is straightforward for the mathematical model to recover the first law of thermodynamics provided one of the conserved variables in $\vec{q}$ is the total energy. It is trickier for the mathematical model to accurately capture the second law of thermodynamics because the mechanism that governs irreversible processes (entropy) is typically not explicitly built into the PDE \cite{harten1983,kruzkov1970,lax2005,tadmor1984,Tadmor1987_2,Tadmor2003}. Hence, we have to fulfill an auxiliary conservation law of entropy for reversible processes and must guarantee that entropy increases for irreversible processes.

\subsection{Continuous mathematical entropy analysis for general systems}\label{sec:contEnt}

We outline some of the necessary additions to the mathematical analysis in order to incorporate the entropy inequality. The investigation of hyperbolic PDEs modeling compressible fluid dynamics that are mathematically entropy consistent has been the subject of research for approximately the past fifty years \cite{godunov1961,godunov1972,harten1983,kruzkov1970,kuznetsov1976,lax1954,lax2005,mock1980,svard2016,tadmor1984,Tadmor1987_2,Tadmor2003,tadmor2016}. Here we indicate an important difference in the sign convention between the physical and mathematical communities. In physics entropy production is increasing with time, whereas in mathematics entropy is modeled with a decreasing function of time. This difference in notation is largely due to the fact that mathematicians want an upper bound on the entropy \cite{merriam1989}.
In contrast, physics -- specifically statistical mechanics -- relate entropy to the number of microscopic configurations $\Omega$ that a thermodynamic system can have when it is in a state specified by some macroscopic variables. Since it seems natural to assume that each microscopic configuration is equally probable, the number of configurations can be denoted as
\begin{equation}
	S = k_{\mathrm{B}} \ln \Omega,
\end{equation}
where $k_\mathrm{B}$ is the Boltzmann constant, a commonly found constant that can be interpreted as the bridge between macroscopic and microscopic physics \cite{Landau1959}.

Boltzmann's constant, and therefore also the entropy, have dimensions of energy divided by temperature as the number of microstates (and hence its logarithm) is a unitless quantity. From a different point of view, physicists -- in contrast to mathematicians -- interpret entropy as a measure of disorder within a macroscopic system. In fact, macroscopic systems tend to spontaneously evolve into the state with maximum entropy, commonly called thermodynamic equilibrium.
As an example, imagine the temperature is zero everywhere in the system. Then every particle in the system is ordered (in a specific way) and the number of microstates $\Omega$ is minimal. However, if we increase the temperature, then the particles start to move because of their thermal energy, i.e., Brownian motion and the number of possible microstates increases. It is evident that any orderly arrangement would be very rare because all particles are in random thermal motion. Even if we would have the a special configuration at some time again, this would be for a brief instant and the particles would again move to some other configuration.

We consider the one-dimensional entropy analysis of the general hyperbolic conservation law \eqref{eq:1DCons}. First, we define a scalar \textit{entropy function} $S(\vec{q})$ that satisfies \cite{harten1983,merriam1989,Tadmor1987_2,Tadmor2003}
\begin{enumerate}
\item[$(i)$] The entropy function $S$ is a strongly convex function of $\vec{q}$.\vspace{0.1cm}
\item[$(ii)$] The entropy function $S$ is augmented with a corresponding \textit{entropy flux} function $F(\vec{q})$ in the $x-$direction such that
\begin{equation}\label{eq:compatCond}
\left(\pderivative{S}{\vec{q}}\right)^T\pderivative{\vec{f}}{\vec{q}} = \left(\pderivative{F}{\vec{q}}\right)^T.
\end{equation}
\end{enumerate}
These conditions are referred to as the \textit{convexity condition} and the \textit{compatibility condition}, respectively. The function $S$ and its corresponding flux $F$ form an entropy-entropy flux pair $(S,F)$. It is important to note that the entropy-entropy flux pair need not be unique, e.g., as was demonstrated by Harten \cite{harten1983} and Tadmor \cite{Tadmor1987_2}.
From the entropy function we define a new vector of \textit{entropy variables}
\begin{equation}\label{eq:entVars}
\vec{v} := \pderivative{S}{\vec{q}}.
\end{equation}
Next, we contract the hyperbolic conservation law \eqref{eq:1DCons} with the entropy variables $\vec{v}$ to obtain
\begin{equation}\label{eq:intermediate}
\vec{v}^T\left(\pderivative{\vec{q}}{t} + \pderivative{\vec{f}(\vec{q})}{x}\right) = {0}.
\end{equation}
From the compatibility condition \eqref{eq:compatCond} and the definition of the entropy variables \eqref{eq:entVars} we see that \eqref{eq:intermediate} transforms into the entropy conservation law
\begin{equation}\label{eq:continuousEntCons}
\pderivative{S}{t} + \pderivative{F}{x} = 0,
\end{equation}
for smooth solutions. If we account for non-smooth regimes as well then the entropy equality modifies to be the entropy inequality \cite{Tadmor2003}
\begin{equation}\label{eq:continuousEntInq}
\pderivative{S}{t} + \pderivative{F}{x} \leq 0,
\end{equation}
where we used the mathematical convention of a decreasing entropy.
So, we see that with a particular choice of entropy function and satisfying a compatibility condition on the entropy flux we can build the second law of thermodynamics into the system in a mathematically consistent way.

There are other important quantities related to the mathematical entropy analysis necessary to discuss the PDE either in the space of conservative variables or in the space of entropy variables. Due to the strong convexity of the entropy function there exist symmetric positive definite (s.p.d.) Jacobian matrices we use to transform back and forth between the variable spaces. The Jacobian matrix to move from conservative space to entropy space is found by computing the Hessian of the entropy function
\begin{equation}\label{eq:entJac1}
\mathbf{H}^{-1} := \frac{\partial^2 S}{\partial \vec{q}^2}=\pderivative{\vec{v}}{\vec{q}}.
\end{equation}
It is interesting to note that it was shown by Friedrichs and Lax \cite{friedrichs1971} that multiplication of the conservation law ``on the left'' by the matrix $\mathbf{H}^{-1}$ symmetrizes the system \eqref{eq:1DCons}. Because the matrix $\mathbf{H}^{-1}$ is s.p.d., we can immediately compute the other entropy Jacobian through inversion of \eqref{eq:entJac1} to obtain
\begin{equation}\label{eq:EntropyJacobian}
\mathbf{H} := \pderivative{\vec{q}}{\vec{v}}.
\end{equation}

\subsection{Example: Ideal magnetohydrodynamic equations}\label{sec:idealMHD}

We are interested in capturing as much physics as possible with the mathematical model. Therefore, we consider a set of hyperbolic PDEs that offer applications to a wide range of physical phenomena. To do so we desire the ability to mathematically describe the evolution of plasmas (electrically ionized gases) such that we can model flow configurations local to our planet as well as to the Universe that surrounds us as a high electrical conductivity is ubiquitous in astrophysical objects. One example of interesting phenomena we wish to model is the star formation process.

To extend the description of fluid dynamics to gases that respond to electric and magnetic fields we couple the inviscid Euler equations \cite{Landau1959} with a reduced set of Maxwell's equations \cite{jackson1975} to obtain the ideal magnetohydrodynamic (MHD) equations \cite{freidberg1987}. The ``ideal MHD'' approximation involves some fundamental assumptions on the fluid we want to model with it:
\begin{description}
	\item[The fluid approximation] The essential approximation in the ``ideal MHD'' model is that variations in the macroscopic quantities we observe are slow compared to the time scale of microscopic processes in the fluid we want to model. We require that a sufficient number of gas particles is available to make meaningful definitions of macroscopic properties like density, velocity, and pressure.
	\item[Instantaneous relation for the electric field] There is an instantaneous relation between current densities and the electric field (Ohm's law). The term ``instantaneous'' means that the involved processes have to take place on small scales that average out over temporal and spacial scales smaller than the ones of interest. For this to happen, we require \textit{perfect conductivity} that can also be understood as vanishing electrical resistivity. Electric conductivity is given if the constituents of the plasma are at least partly ionized, which is sufficient in most astrophysical environments.
	\item[Neutrality] The fluid is (overall) electrically neutral. This approximation is most often valid.
\end{description}

The coupling of gas dynamics and magnetic fields of a perfectly conducting fluid creates a three-dimensional hyperbolic system of conservation laws of the form
\begin{align}\label{eq:3DIDEALMHD}
\pderivative{\vec{q}}{t} + \nabla\cdot\vec{f}(\vec{q}) =
\pderivative{}{t}
\begin{bmatrix}
\rho \\[1mm] \rho\vec{u} \\[1mm] E \\[1mm] \vec{B} \end{bmatrix}
+
&\nabla\cdot
\begin{bmatrix}
\rho\vec{u} \\[1mm]
\rho(\vec{u}\otimes\vec{u}) + \left(p+\frac{1}{2}\|\vec{B}\|^2\right)\mat{I}-\vec{B}\otimes\vec{B} \\[1mm]
\vec{u}\left(\frac{\rho}{2}\|\vec{u}\|^2 + \frac{\gamma p}{\gamma - 1} + \|\vec{B}\|^2 \right) - \vec{B}(\vec{u}\cdot\vec{B}) \\[1mm]
\vec{u}\otimes\vec{B} - \vec{B}\otimes\vec{u}
\end{bmatrix} = \vec{0},
\end{align}
where $\mat{I}$ is the $3\times3$ identity matrix. The density $\rho$, momenta $\rho\vec{u}$, total energy $E$, and the magnetic fields $\vec{B}$ are the \textit{primary} conserved quantities. The variables $\vec{u}$ and $p$ are the fluid velocities and the thermal pressure, respectively. No evolution equations for the electric field are present since we know from Ohm's law that it vanishes in a perfect conductor. However, it should be noted that the electric field vanishes only in the \textit{co-moving} or \textit{fluid reference frame} we use that moves with the flow. In any other reference frame there is an electric field that would require attention.

We close the system under the assumption of an ideal gas which relates the thermal pressure to the total energy by
\begin{equation}
p = (\gamma-1)\left(E - \frac{\rho}{2}\|\vec{u}\|^2 -\frac{1}{2}\|\vec{B}\|^2 \right),
\end{equation}
where the ratio of specific heats, $\gamma$, is the adiabatic constant.

We immediately see that the first law of thermodynamics, i.e., total energy conservation, is the fifth equation in \eqref{eq:3DIDEALMHD}. However, the second law of thermodynamics, i.e., the entropy inequality, is \textit{\textbf{not}} explicitly built into the ideal MHD equations. As mentioned in Sec. \ref{sec:physics}, this can be problematic. 
It is well known (see e.g.~\cite{merriam1989}) that magnetized fluids governed solely by these conservation laws can, in fact, display very strange behavior. A prime example is an air conditioner that is able to transfer internal energy (heat) to kinetic energy (motion), resulting in a very fast, very cold jet of air. This strange air conditioner is allowed under \eqref{eq:3DIDEALMHD}, yet is never observed in practice. Thus, we see that special care must be taken to ensure that the mathematics respects the physics. We do this by explicitly adding the evolution of the thermodynamic entropy in our derivations. By this, we can incorporate the laws of thermodynamics explicitly into the numerical approximation. As this air conditioner requires state changes that are not valid in our scheme it is rendered impossible and the importance of our effort is immediately apparent.

An additional constraint for the ideal MHD equations to model the evolution of magnetized fluid dynamics we must ensure that the divergence on the magnetic field is zero, as is dictated by one of Maxwell's equations, specifically Gauss' law for magnetism,
\begin{equation}\label{eq:divFree}
\nabla\cdot\vec{B}=0.
\end{equation}
We refer to this property as the \textit{divergence-free condition} (although it is sometimes called the involution \cite{barth2006,helzel2013} or solenoidal \cite{jackson1975} condition in the literature). The geometrical meaning of \eqref{eq:divFree} is that magnetic field lines have ``no ends,'' i.e., regions of reduced field strength cannot be local since magnetic field lines are not allowed to meet any monopolar singularities, e.g., see Fig. \ref{fig:divB}, like seen with electric field lines and charges. Hence, in contrast to other macroscopic quantities like density, a change in magnetic field strength must be accommodated by changes in the field morphology on a larger scale causing additional difficulties in the numerical modeling.

\begin{figure}[h]
	\centering
	\includegraphics{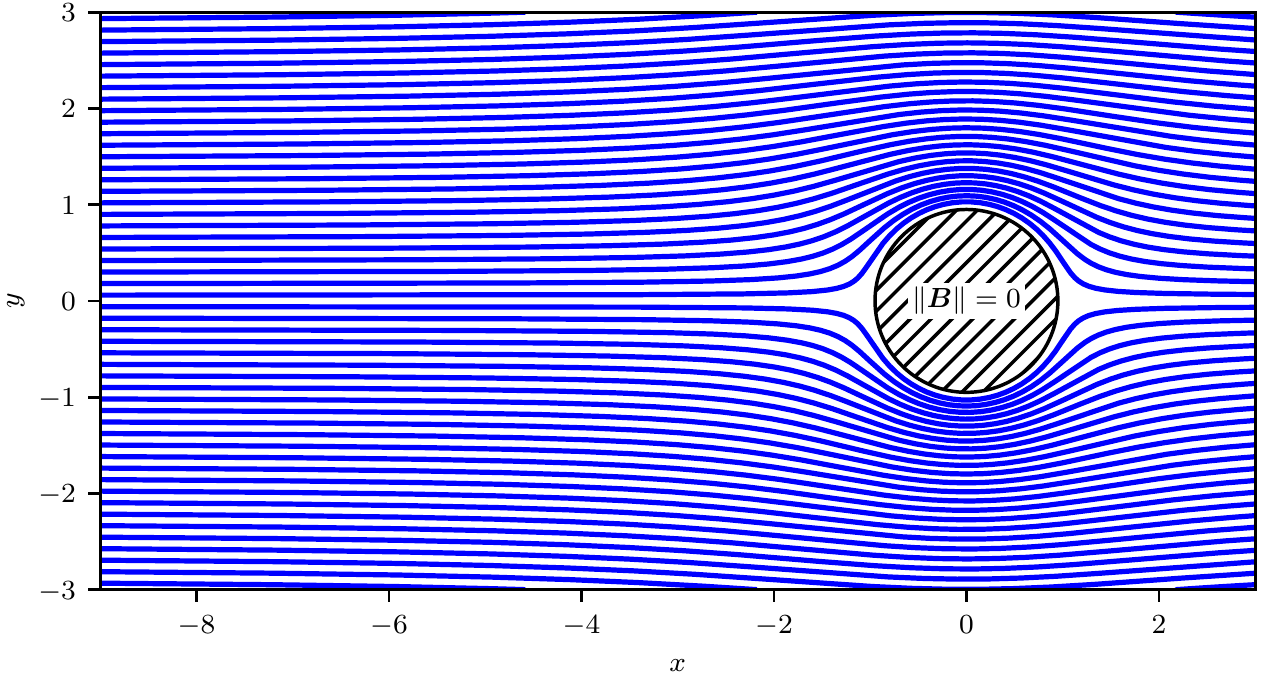}
	\caption{Magnetic field lines in free space and close to a field-free region (dashed cylinder). Due to the divergence-free condition \eqref{eq:divFree}, field lines have to wrap around the region with zero magnetic field and distort the magnetic field topology in their vicinity. Hence, regions of reduced field strength cannot only be local.}
	\label{fig:divB}
\end{figure}

On the continuous level \eqref{eq:divFree} is assumed to always be satisfied. However, even if the divergence-free constraint is satisfied with the initial conditions, it is not necessarily true that it will remain satisfied through the \textit{discrete} evolution of the equations \cite{Brackbill1980}. So, we see that the divergence-free constraint provides an important indicator to decide if flows remain physically meaningful during their numerical approximation.

Just like entropy conservation, the satisfaction of the divergence-free condition \eqref{eq:divFree} is not explicitly included in the PDE system of the ideal MHD equations \eqref{eq:3DIDEALMHD} which introduces additional complexities in the (numerical) modeling of plasmas. This is problematic as mathematical analysis requires a single equation to model physical processes rather a set of uncoupled equations. We next demonstrate that the divergence-free condition can be built into the non-linear PDE system such that the equations can mathematically model all of the observed physics of the system simultaneously.

\subsubsection{Continuous mathematical entropy and the divergence-free condition}

Surprisingly, the issue of entropy conservation and satisfaction of the divergence-free condition are inextricably linked for the ideal MHD system. To demonstrate this we consider the entropy function of the ideal MHD system to be the physical entropy density \cite{barth1999}
\begin{equation}\label{eq:entropy}
S(\rho, p) = -\frac{\rho s}{\gamma-1},
\end{equation}
where $s(\rho, p)=\ln(p)-\gamma\ln(\rho)$. The corresponding entropy fluxes are
\begin{equation}
\vec{F}(\rho, p, \vec{u}) = \vec{u}S,
\end{equation}
where $\vec{u} = (u,v,w)^T$.
Therefore, we compute the entropy variables to be
\begin{equation}\label{eq:entropyVariables}
\vec{v} = \pderivative{S}{\vec{q}} = \left[ \frac{\gamma - s}{\gamma - 1} -\frac{\rho\|\vec{u}\|^2}{2p},\;\frac{\rho u}{p},\;\frac{\rho v}{p},\;\frac{\rho w}{p},\;-\frac{\rho}{p},\;\frac{\rho B_{1}}{p},\;\frac{\rho B_{2}}{p},\;\frac{\rho B_{3}}{p}\right]^T,
\end{equation}
which contracts the ideal MHD equations into the entropy conservation law \eqref{eq:continuousEntCons} or entropy inequality \eqref{eq:continuousEntInq} depending on the smoothness of the solution $\vec{q}$. The entropy Jacobian \eqref{eq:EntropyJacobian} is
\begin{equation}\label{entropyJacobian}
\mat{H} = \begin{bmatrix}
\rho & \rho u & \rho v & \rho w & E - \frac{1}{2}\|\vec{B}\|^2 & 0 & 0 & 0 \\[0.1cm]
\rho u & \rho u^2 + p & \rho u v & \rho u w & \rho {h} u & 0 & 0 & 0 \\[0.1cm]
\rho v & \rho u v & \rho v^2 + p& \rho v w & \rho {h} v & 0 & 0 & 0 \\[0.1cm]
\rho w & \rho u w & \rho v w & \rho w^2 + p & \rho {h} w & 0 & 0 & 0 \\[0.1cm]
E - \frac{1}{2}\|\vec{B}\|^2 & \rho{h} u & \rho{h} v & \rho{h} w & \rho{h}^2 - \frac{a^2 p }{\gamma-1} + \frac{a^2\|\vec{B}\|^2}{\gamma} & \frac{pB_1}{\rho}& \frac{pB_2}{\rho}& \frac{pB_3}{\rho}\\[0.1cm]
0 & 0 & 0 & 0 & \frac{pB_1}{\rho} & \frac{p}{\rho} & 0 & 0 \\[0.1cm]
0 & 0 & 0 & 0 & \frac{pB_2}{\rho} & 0 & \frac{p}{\rho} & 0 \\[0.1cm]
0 & 0 & 0 & 0 & \frac{pB_3}{\rho} & 0 & 0 & \frac{p}{\rho}
\end{bmatrix},
\end{equation}
where
\begin{equation}
a^2 = \frac{p\gamma}{\rho},\quad E = \frac{p}{\gamma-1} + \frac{\rho}{2}\|\vec{u}\|^2 + \frac{1}{2}\|\vec{B}\|^2,\quad h = \frac{a^2}{\gamma-1} + \frac{1}{2}\|\vec{u}\|^2.
\end{equation}
Also, we compute the \textit{entropy potential} defined as
\begin{equation}\label{EntropyPotential}
\phi := \vec{v}\cdot\vec{f} - F = \rho u + \frac{\rho}{p} \left[ \frac{1}{2} u \|\vec{B}\|^2 - B_1(\vec{u}\cdot\vec{B}) \right].
\end{equation}
The usefulness of the entropy potential \eqref{EntropyPotential} will reveal itself later in Sec.~\ref{sec:ECIdealMHD} where the potential function $\phi$ plays an important role to \textit{mimic} the continuous entropy analysis at the discrete level.

Just as in Sec.~\ref{sec:contEnt} we contract the PDE system into entropy space with the goal to recover entropy conservation. So, we premultiply \eqref{eq:3DIDEALMHD} with the entropy variables and after some algebraic manipulation we arrive at the expression
\begin{equation}\label{eq:entLawMHD}
\vec{v}^T\left(\pderivative{\vec{q}}{t} + \nabla\cdot\vec{f}(\vec{q})\right) = \pderivative{S}{t} + \nabla\cdot\vec{F} - \frac{\rho(\vec{u}\cdot\vec{B})}{p}(\nabla\cdot\vec{B}) = 0.
\end{equation}
On the continuous level we know that the divergence-free constraint \eqref{eq:divFree} is always satisfied and the last term in \eqref{eq:entLawMHD} vanishes, thus recovering entropy conservation. However, as previously discussed, a numerical approximation will not necessarily satisfy the divergence-free condition. Therefore, a discrete evolution of the ideal MHD equations cannot guarantee entropy conservation without addressing errors made in the divergence-free constraint.

In the course of an entropy analysis of the ideal MHD equations \eqref{eq:3DIDEALMHD} Godunov \cite{godunov1972} observed that the divergence-free condition can be incorporated into the ideal MHD equations as a source term proportional to the divergence of the magnetic field (which, on a continuous level, is a clever way of adding zero to the system). This additional source term not only allows the equations to be put in symmetric hyperbolic form \cite{barth1999,godunov1972}, but it also restores Galilean invariance, i.e., Newton's laws of motion hold in all frames related to one another by a Galilean transformation. This source term, first investigated in the multi-dimensional numerical context by Powell et al.~\cite{Powell1999}, augments the ideal MHD system to become
\begin{align}\label{eq:Powell}
\pderivative{}{t}
\begin{bmatrix}
\rho \\[1mm] \rho\vec{u} \\[1mm] E \\[1mm] \vec{B} \end{bmatrix}
+
&\nabla\cdot
\begin{bmatrix}
\rho\vec{u} \\[1mm]
\rho(\vec{u}\otimes\vec{u}) + \left(p+\frac{1}{2}\|\vec{B}\|^2\right)\mat{I}-\vec{B}\otimes\vec{B} \\[1mm]
\vec{u}\left(\frac{\rho}{2}\|\vec{u}\|^2 + \frac{\gamma p}{\gamma - 1} + \|\vec{B}\|^2 \right) - \vec{B}(\vec{u}\cdot\vec{B}) \\[1mm]
\vec{u}\otimes\vec{B} - \vec{B}\otimes\vec{u}
\end{bmatrix} = -(\nabla\cdot\vec{B})\begin{bmatrix}
0\\
\vec{B}\\
\vec{u}\cdot\vec{B}\\
\vec{u}
\end{bmatrix}.
\end{align}
While it is now possible to recover entropy conservation from the form \eqref{eq:Powell} there is a significant drawback. The new set of equations \eqref{eq:Powell} loses conservation in all but the continuity equation. In numerics this is problematic as non-conservative approximations encounter difficulties obtaining the correct shock speeds or the physically correct weak solution, e.g., \cite{toth2000}.

Fortunately, there is an alternative source term that incorporates the divergence-free condition into the ideal MHD system and maintains the conservation of the momenta and total energy equations.
Janhunen \cite{janhunen2000} used a proper generalization of Maxwell's equations when magnetic monopoles are present and imposed electromagnetic duality invariance of the Lorentz force to derive an alternative MHD system
\begin{equation}\label{eq:JanhunenSource}
\pderivative{}{t}\begin{bmatrix} \rho \\[0.05cm] \rho\vec{u} \\[0.05cm] E \\[0.05cm] \vec{B} \end{bmatrix} + \nabla\cdot\begin{bmatrix} \rho\vec{u} \\[0.05cm]
\rho(\vec{u}\otimes\vec{u}) + \left(p+\frac{1}{2}\|\vec{B}\|^2\right)\mat{I}-\vec{B}\otimes\vec{B}  \\[0.05cm]
\vec{u}\left(\frac{\rho}{2}\|\vec{u}\|^2 + \frac{\gamma p}{\gamma - 1} + \|\vec{B}\|^2 \right) - \vec{B}(\vec{u}\cdot\vec{B}) \\[0.05cm]
\vec{u}\otimes\vec{B} - \vec{B}\otimes\vec{u}
\end{bmatrix} = -(\nabla\cdot\vec{B})\begin{bmatrix} 0\\\vec{0}\\0\\\vec{u}\end{bmatrix}.
\end{equation}
From the structure of the Janhunen source term in \eqref{eq:JanhunenSource} and the last three components of the entropy variables \eqref{eq:entropyVariables} we see that contracting the new MHD system \eqref{eq:JanhunenSource} into entropy space will cancel the extraneous terms present in \eqref{eq:entLawMHD}. Thus, the ideal MHD system with the Janhunen source term \eqref{eq:JanhunenSource} can be used to recover entropy conservation on a discrete level, as demonstrated in \cite{Winters2016}. Also, because the system remains conservative in all the hydrodynamic variables a numerical approximation will satisfy the Lax-Wendroff theorem \cite{lax1960} and retains the ability to capture the correct shock speeds.

\section{Discrete entropy analysis for the ideal MHD equations}\label{sec:discEntAnalysis}

We have discussed the continuous components of the mathematical entropy analysis for the ideal MHD equations. With the necessary physical and mathematical background complete, we are prepared to embark on the mission set out at the end of Sec.~\ref{sec:physics}. That is, we are equipped to describe a discrete numerical approximation of a system of hyperbolic PDEs that remains entropy consistent. We divide the discussion of the discrete entropy analysis into four parts: Sec.~\ref{sec:finiteVolume} provides background details on finite volume methods, a particularly useful numerical approximation for systems of hyperbolic conservation laws. Next, Sec.~\ref{sec:ECIdealMHD} presents a baseline \textit{entropy conservative} scheme for the ideal MHD equations. We have seen that physics dictates that entropy must dissipate in the presence of discontinuities. Therefore, we augment the baseline entropy conservative scheme with numerical dissipation in Sec.~\ref{sec:ESIdealMHD} to create an \textit{entropy stable} scheme. Finally, Sec.~\ref{sec:discEvaluation} provides detailed information on how to evaluate the numerical dissipation term to ensure that the approximation avoids unphysical states, e.g., negative values of the density.

For clarity, we provide a brief summary of the overall approximation in Sec.~\ref{sec:summaryDiscEnt}.

\subsection{Finite volume numerical discretization}\label{sec:finiteVolume}

Based on the current state-of-the-art of scientific knowledge, we investigate the behavior of fluids by using the conservation laws for mass, momentum, total energy, magnetic field, and entropy. However, it turns out that only a minority of very simplified hyperbolic PDEs have a known explicit form of an analytic solution. Although the knowledge of such simplified solutions are extremely valuable in helping understand how fluids behave, they can only rarely be used for practical applications in physics and engineering. Most often, the simplified equations are based on a combination of approximations and dimensional analysis; empirical input is almost always required \cite{Ferziger2013}. This approach is very successful when the system can be described by one, or at most, two parameters. However, many flows require several parameters for their specification, rendering the required experiments extremely difficult or even impossible as, often, the necessary quantities are simply not measurable with present techniques or can only be measured with insufficient accuracy.

Hence, apart from rare exceptions, it is mandatory to be able to solve the underlying equations in a different way to be able to investigate the behavior of flows accurately. A solution came with the invention of sufficiently powerful computers. Although many key numerical techniques had already been developed more than a century ago, they were of little use before the invention of the computer. The performance of computers has increased spectacularly since the 1950s. In the early days of computing, computers were able to work on a few hundred instructions per second, whereas modern computers are designed to perform in the $10^{12}$ operations per second range, with no apparent sign of this trend to slow down. At the same time, the ability to store large amounts of data has increased in a similarly rapid fashion.

It requires little to no imagination to see that computers might be a very valuable tool in trying to understand complex fluid flows. Out of this recognition the field of \textit{computational fluid dynamics} (CFD) emerged. This field solves the governing equations \textit{numerically} in order to approximate the solution of a particular problem.
The range of applications extends from small scale simulations running on personal computers within seconds to large scale simulations running for months on the largest supercomputers that exist to date. This, of course, depends on the complexity and the desired accuracy of the flows to be simulated. For example, if we are interested in the design of a new car, we would like to simulate the flow around a moving car in a wind tunnel, we would have to build a car model and blow air in a wind tunnel at it. However, to obtain accurate results, the floor of the wind tunnel would need to move with the same velocity as the air flow to avoid boundary effects in the flow pattern. At the same time, however, we would want to have the car stand perfectly still, which is of course technically difficult. However, using a computer simulation we neither have to build a car model nor is setting the boundary condition for the moving floor boundary a difficulty.

It may have become apparent that it is not practical to try to cover the whole field of CFD in a single work. Also, the field is evolving so rapidly, that we risk the discussion becoming out of date in a short while. Hence, we focus here on the specific methods designed to solve the system of equations of interest numerically.

To numerically approximate, we have to discretize the mathematical model, i.e.,~the set of partial differential equations \eqref{eq:3DIDEALMHD} themselves. Besides advanced methods like spectral methods, e.g., \cite{CHQZ:2006,Canuto:2007fj,shen2006}, the most important methods are finite difference (FD), e.g., \cite{leveque2007,strikwerda2004}, finite element (FE), e.g., \cite{brenner2007,johnson2012}, and finite volume (FV), e.g., \cite{leveque1992,leveque2002}. The introduction of numerical approximations of fluid flow introduces the notion of \textit{resolution}. Once we move to the discrete level, the approximation of a particular fluid is only known at a particular set of points. Therefore, it is intuitive to note that we need ``enough'' discrete points to capture the flow features we wish to study. The issue of resolution for a numerical method is discussed later in this section. However, we note that each of the numerical methods introduced yield the same result if the numerical resolution is sufficiently high.

Due to their simplicity of presentation and particular relevance to the solution of systems of hyperbolic conservation laws \cite{leveque1992,leveque2002} we focus mainly on the development of a FV scheme in this work. The basic idea can be broken into four steps:
\begin{enumerate}
\item[1.] Divide the domain, the simplest being a square box, into a number of Cartesian cells.\\
\item[2.] Approximate the solution $\vec{q}$ by the average of each quantity at the center of each cell.\\
\item[3.] Approximate the fluxes and any source term at each of the interfaces in between the cells.\\
\item[4.] The spatial discretization creates a systems of ordinary differential equations (ODEs) to evolve in time. We evolve the approximate solution in time using an explicit time integration technique, e.g., Adams-Bashforth \cite{quarteroni2010}, Runge-Kutta \cite{butcher1996}, or strong stability preserving Runge-Kutta \cite{Gottlieb2001}. We note it is also possible to use implicit ODE intergration techniques \cite{quarteroni2010}, but they will not be discussed further.
\end{enumerate}
We provide a sketch of the first three steps of the finite volume approximation in Fig. \ref{fig:finiteVolume}. 
\begin{figure}[h]
\captionsetup[subfigure]{labelformat=empty}
	\centering
	\begin{subfigure}[b]{1.0\textwidth}
		\centering
		\begin{tikzpicture}[scale=1.5]
  \draw[-] (-3,0) -- (4,0);
  \draw[-] (-3,-0.11) -- (-3,0.11);
  \draw[-] (4,-0.11) -- (4,0.11);
  \draw[dashed] (-2,-0.1) -- (-2,2);
  \draw[dashed] (-1,-0.1) -- (-1,2);
  \draw[dashed] (0,-0.1) -- (0,2);
  \draw[dashed] (1,-0.1) -- (1,2);
  \draw[dashed] (2,-0.1) -- (2,2);
  \draw[dashed] (3,-0.1) -- (3,2);
  \draw (-2.5,-0.35) node{$1$};
  \draw (3.5,-0.35) node{$K$};
  \draw (0.5,-0.35) node{$i$};
  \draw (-0.5,-0.35) node{$i-1$};
  \draw (1.5,-0.35) node{$i+1$};
  \draw (-1.5,-0.35) node{$\cdots$};
  \draw (2.5,-0.35) node{$\cdots$};
\end{tikzpicture}
		\caption{Divide the domain into a set of elements labeled with integers.}
	\end{subfigure}
	\\[0.5cm]
	\begin{subfigure}[b]{1.0\textwidth}
		\centering
\begin{tikzpicture}[scale=1.5]
  \draw[-] (-3,0) -- (4,0);
  \draw[-] (-3,-0.11) -- (-3,0.11);
  \draw[-] (4,-0.11) -- (4,0.11);
  \draw[dashed] (-2,-0.1) -- (-2,2);
  \draw[dashed] (-1,-0.1) -- (-1,2);
  \draw[dashed] (0,-0.1) -- (0,2);
  \draw[dashed] (1,-0.1) -- (1,2);
  \draw[dashed] (2,-0.1) -- (2,2);
  \draw[dashed] (3,-0.1) -- (3,2);
  \draw (-2.5,-0.35) node{$1$};
  \draw (3.5,-0.35) node{$K$};
  \draw (0.5,-0.35) node{$i$};
  \draw (-0.5,-0.35) node{$i-1$};
  \draw (1.5,-0.35) node{$i+1$};
  \draw (-1.5,-0.35) node{$\cdots$};
  \draw (2.5,-0.35) node{$\cdots$};
  \draw[very thick,red] (-3,1) -- (-2,1);
  \draw[very thick,red] (-2,0.5) -- (-1,0.5);
  \draw[very thick,red] (-1,0.75) -- (0,0.75);
  \draw[very thick,red] (0,1.25) -- (1,1.25);
  \draw[very thick,red] (1,1.0) -- (2,1.0);
  \draw[very thick,red] (2,0.85) -- (3,0.85);
  \draw[very thick,red] (3,0.5) -- (4,0.5);
  \draw[red,fill=red] (-2.5,1) circle (0.075cm);
  \draw[red,fill=red] (-1.5,0.5) circle (0.075cm);
  \draw[red,fill=red] (-0.5,0.75) circle (0.075cm);
  \draw[red,fill=red] (0.5,1.25) circle (0.075cm);
  \draw[red,fill=red] (1.5,1) circle (0.075cm);
  \draw[red,fill=red] (2.5,0.85) circle (0.075cm);
  \draw[red,fill=red] (3.5,0.5) circle (0.075cm);
  \draw (0.5,1.65) node{$\overline{\vec{q}}_{i}$};
  \draw (1.5,1.65) node{$\overline{\vec{q}}_{i+1}$};
  \draw (-0.5,1.65) node{$\overline{\vec{q}}_{i-1}$};
\end{tikzpicture}
		\caption{Approximate the solution with cell-centered averages.}
	\end{subfigure}
	\\[0.5cm]
	\begin{subfigure}[b]{1.0\textwidth}
		\centering
\begin{tikzpicture}[scale=1.5]
  \draw[-] (-3,0) -- (4,0);
  \draw[-] (-3,-0.11) -- (-3,0.11);
  \draw[-] (4,-0.11) -- (4,0.11);
  \draw[dashed] (-2,-0.1) -- (-2,2);
  \draw[dashed] (-1,-0.1) -- (-1,2);
  \draw[dashed] (0,-0.1) -- (0,2);
  \draw[dashed] (1,-0.1) -- (1,2);
  \draw[dashed] (2,-0.1) -- (2,2);
  \draw[dashed] (3,-0.1) -- (3,2);
  \draw (-2.5,-0.35) node{$1$};
  \draw (3.5,-0.35) node{$K$};
  \draw (0.5,-0.35) node{$i$};
  \draw (-0.5,-0.35) node{$i-1$};
  \draw (1.5,-0.35) node{$i+1$};
  \draw (-1.5,-0.35) node{$\cdots$};
  \draw (2.5,-0.35) node{$\cdots$};
  \draw[very thick,red] (-3,1) -- (-2,1);
  \draw[very thick,red] (-2,0.5) -- (-1,0.5);
  \draw[very thick,red] (-1,0.75) -- (0,0.75);
  \draw[very thick,red] (0,1.25) -- (1,1.25);
  \draw[very thick,red] (1,1.0) -- (2,1.0);
  \draw[very thick,red] (2,0.85) -- (3,0.85);
  \draw[very thick,red] (3,0.5) -- (4,0.5);
  \draw[red,fill=red] (-2.5,1) circle (0.075cm);
  \draw[red,fill=red] (-1.5,0.5) circle (0.075cm);
  \draw[red,fill=red] (-0.5,0.75) circle (0.075cm);
  \draw[red,fill=red] (0.5,1.25) circle (0.075cm);
  \draw[red,fill=red] (1.5,1) circle (0.075cm);
  \draw[red,fill=red] (2.5,0.85) circle (0.075cm);
  \draw[red,fill=red] (3.5,0.5) circle (0.075cm);
  \draw (0.5,1.65) node{$\overline{\vec{q}}_{i}$};
  \draw (1.5,1.65) node{$\overline{\vec{q}}_{i+1}$};
  \draw (-0.5,1.65) node{$\overline{\vec{q}}_{i-1}$};
  \draw (0,2.35) node{$\textit{\textbf{\^{f}}}_{i-\frac{1}{2}}$};
  \draw (1,2.35) node{$\textit{\textbf{\^{f}}}_{i+\frac{1}{2}}$};
\end{tikzpicture}
		\caption{Approximate the fluxes at cell interfaces labeled with half integer values.}
	\end{subfigure}
	\caption{One dimensional illustration of the spatial finite volume discretization.}
	\label{fig:finiteVolume}
\end{figure}
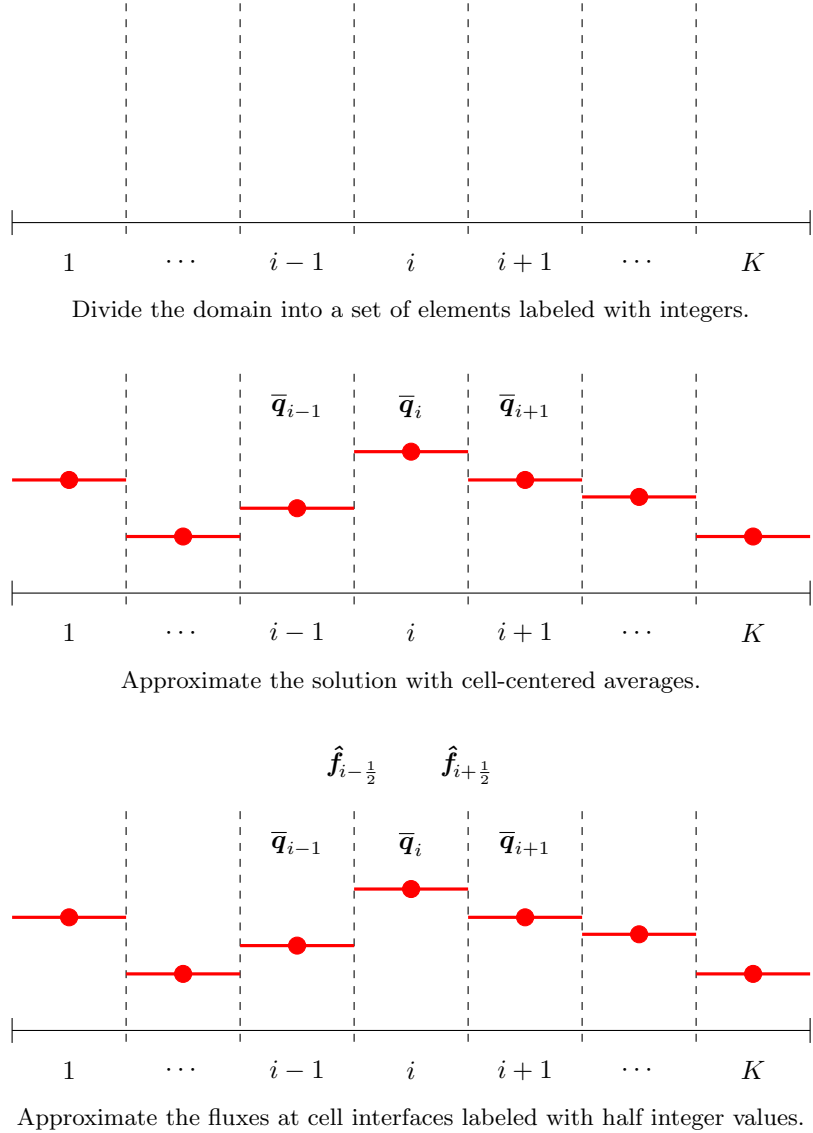
With these four steps it is possible to take a hyperbolic PDE, including a source term, from the continuous level
\begin{equation}
\pderivative{\vec{q}}{t} + \nabla\cdot\vec{f} = \vec{s},
\end{equation}
to the discrete level
\begin{equation}\label{eq:finalFVScheme}
\overline{\vec{q}}_i^{n+1} = \overline{\vec{q}}_i^n - \frac{\Delta t}{\Delta x} \left(\fhat\!_{i+\halb} - \fhat\!_{i-\halb}\right) + \vec{s}_i,
\end{equation}
where $n$ denotes the discrete level in time, $\overline{\vec{q}}_i^n$ is the average solution in a grid cell, $\fhat$ is a numerical flux function where $i\pm\halb$ indicates the value at cell edges, and $\vec{s}_i$ is a consistent discretization of the source term. Also, $\Delta t$ is the time step for the explicit time integration method and $\Delta x$ is the size of the grid cell. We present the fully discrete form of the FV method \eqref{eq:finalFVScheme} to identify the final form of the numerical scheme. Next, we provide specific details of how each of the terms in \eqref{eq:finalFVScheme} arise.

A first question that arises in the FV discretization comes from the fact that the value of the flux is multiply defined at each of the interfaces coupling the individual cells. At a given interface do we evaluate the flux from the centroid value on the left or on the right? To resolve this issue and create a unique flux at the interface, while maintaining discrete conservation of the variables $\vec{q}$, we introduce the notion of a \textit{numerical flux function} denoted as $\fhat$ \cite{leveque2002,roe1981,toro2009}. We delay the explicit form of the numerical flux to Sec.~\ref{sec:ECIdealMHD}. Presently, we provide a more in-depth background on the construction of FV schemes for hyperbolic conservation laws.

The first step in designing a numerical model for the flows of interest is to divide the continuous space around us into small portions at discrete locations in space, i.e., we \textit{discretize} space. This discretization of space can be achieved by placing a suitable numerical \textit{grid} onto our continuous geometry (see Fig.~\ref{fig:discretizationimage} for an illustration in two spatial dimensions). The spatial approximation that we select is of special importance as it will affect the quality of the result significantly. If the discretization is too coarse, we risk oversimplifying the approximated flow of interest and lose most of the details, obtaining a solution which may not be sufficient enough to gain the required conclusions. This amplifies, that, although the field of CFD and its applications are thrilling, one always has to bear in mind that numerical results are always only approximate solutions. If we, however, discretize with a resolution that is unreasonably high, we increase the computational complexity, and hence the costs of our simulation exponentially without gaining a noticeably higher degree of accuracy in the result.

\begin{figure}[h]
	\centering
	\begin{subfigure}[b]{0.45\textwidth}
		\centering
		\includegraphics{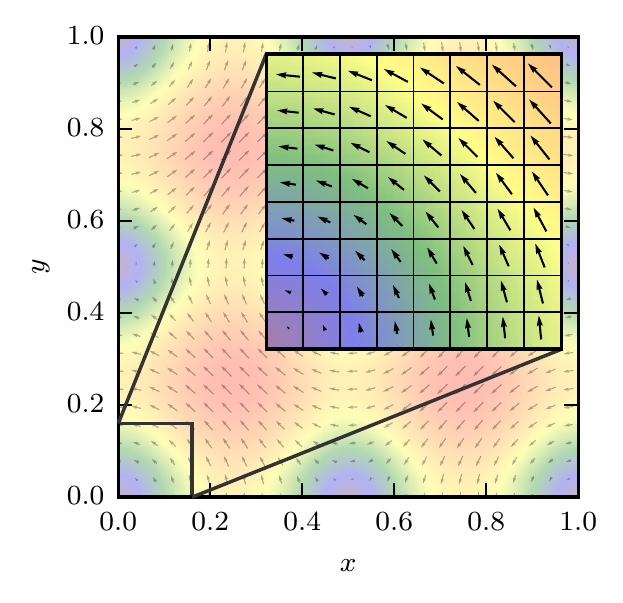}
		\caption{Fine resolution: Captures all features}
	\end{subfigure}
	\begin{subfigure}[b]{0.5\textwidth}
		\centering
		\includegraphics{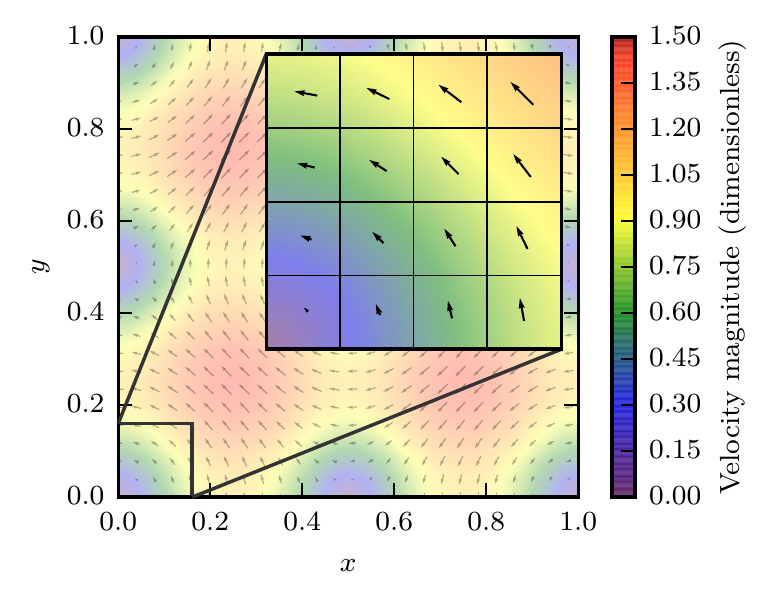}
		\caption{Coarse resolution: Loss of accuracy}
	\end{subfigure}
	\caption{A grid is put onto the fluid to be simulated. The resolution has to be fine enough in order to capture all relevant information of the flow. Too low resolution results in oversimplification of the problem and may lead to a significant loss of accuracy.}
	\label{fig:discretizationimage}
\end{figure}

We briefly compare the FD and FV methods to further motivate the choice of the FV method for the solution of hyperbolic conservation laws. FD type methods are the oldest of the mentioned methods and are accounted to the mathematician Leonhard Euler, who introduced this method in the mid-eighteenth century \cite{Ferziger2013}. The starting point is the conservation equations in differential form as given by \eqref{eq:1DCons}. Derivatives in the original equations are replaced by difference quotients, resulting in a set of algebraic equations for each nodal point on the grid (cf.~Fig.~\ref{fig:FDFV}A), in which the quantities (like density, velocity, etc.) at the current and a number of neighboring cells are used to approximate the derivatives. In principle, this method can be applied to any type of grid. On regular Cartesian grids, this method is particularly easy to understand and implement. 

\begin{figure}[h]
	\centering
	\begin{subfigure}[b]{0.3\textwidth}
		\centering
		\includegraphics{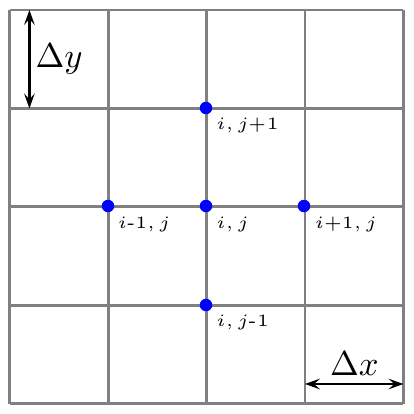}
		\caption{FD: Cell nodes}
	\end{subfigure}
	\begin{subfigure}[b]{0.3\textwidth}
		\centering
		\includegraphics{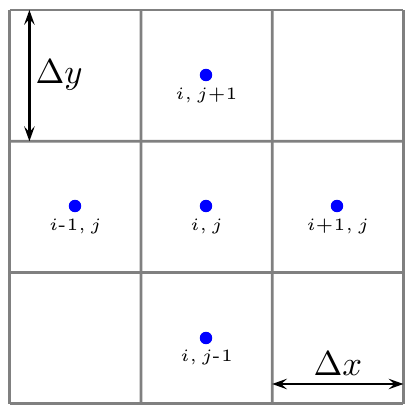}
		\caption{FV: Cell centers}
	\end{subfigure}
	\begin{subfigure}[b]{0.3\textwidth}
		\centering
		\includegraphics{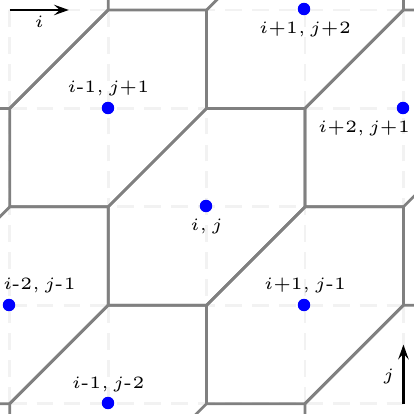}
		\caption{FV: Cell centers, polygons}
	\end{subfigure}
	\caption{Illustration of the discretization using the FD and FV method in 2D. $i$ and $j$ are the indices counting in $x$ and $y$ direction, respectively.}
	\label{fig:FDFV}
\end{figure}

The FV approach is perhaps the most popular method used to approximate hyperbolic conservation law as it is particularly easy to understand. It uses an integral form of the conservation equations as a starting point. Here, the solution is subdivided into a finite number of contiguous cell volumes. The conservation equations can then directly be applied to each volume. The discrete points are located at the centroid of each volume (cf.~Fig.~\ref{fig:FDFV}B). Interpolation is used to extrapolate the value of the volume centered values to the cell faces where the numerical fluxes are computed. Surface and volume integrals can easily be approximated using appropriate quadrature methods. The grid need not be aligned to a coordinate system. Hence, the FV method is suitable for any type of grid, as well as for discretizations of a domain that contain complex geometries. The cells can be constructed using a polyhedron of any dimension (in 2D polygons like triangles, rectangles, pentagons, etc., in 3D polyhedrons like tetrahedra, prisms, hexahedra, etc.). The only requirements are that the sum of the individual volumes must cover the whole computational domain and that the overlap is limited to the boundaries of the cells. On any given volume the flux entering through a particular boundary is identical to the flux leaving the adjacent volume. Thus, the method is conservative by construction, provided surface integrals are the same for the volumes sharing the same boundary \cite{Ferziger2013,leveque2002}. An interesting detail is that all terms that need to be approximated have a direct physical meaning, which is a major reason for its widespread use and popularity among physicists and engineers.

We seek a numerical discretization of the ideal MHD equations. Additionally, because the divergence-free condition is important to guarantee entropy conservation we include a source term. So, let us look at the general conservation law with a source term,
\begin{equation}\label{eq:generalconslaw}
\pderivative{\vec{q}}{t} + \nabla \cdot \vec{f}(\vec{q}) = \vec{s},
\end{equation}
where $\vec{q} = \vec{q}(\vec{x},t)$ is the conserved quantity, $\vec{f}(\vec{q})$ is the vector valued flux function depending on the physical variables of the solution, $\vec{s}$ is the source term, and $\vec{x}$ is the vector of the independent spatial variables. This differential conservation law can be understood as the local formulation of the integral conservation law equation which may be obtained if we take the volume integral over the total volume of a particular cell $i$, $V_i$, and apply the divergence theorem to the second term. This gives,
\begin{equation}\label{eq:integralconslaw}
	\int_{V_i} \pderivative{\vec{q}}{t} \,\mathrm{d}V + \int_{\partial V_i} \vec{f}(\vec{q}) \cdot \vec{n} \,\mathrm{d}S = \int_{V_i} \vec{s} \, \mathrm{d}V,
\end{equation}
where
$S$ is the surface of the volume, and $\vec{n}$ is the (unit) surface normal in outward direction (see Fig.~\ref{fig:Gebiet}).

\begin{figure}[h]
	\centering
	\includegraphics{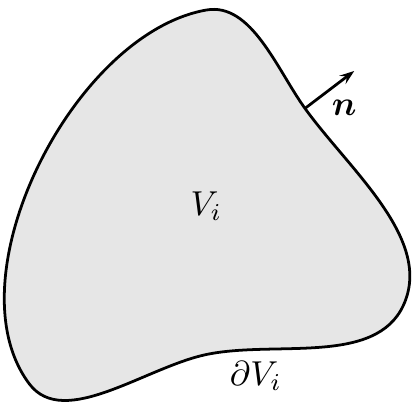}
	\caption{Illustration for the integral conservation law.}
	\label{fig:Gebiet}
\end{figure}

On integrating the first term in \eqref{eq:integralconslaw} to get the \textit{volume average} of $\vec{q}$, $\overline{\vec{q}}$, this yields
\begin{equation}\label{eq:finiteVolumeWithSource}
\frac{\mathrm{d} \overline{\vec{q}}}{\mathrm{d} t} + \frac{1}{V_i}\int_{\partial V_i} \vec{f}(\vec{q}) \cdot \vec{n} \,\mathrm{d}S = \frac{1}{V_i}\int_{V_i} \vec{s} \, \mathrm{d}V,
\end{equation}
which is a general result equivalent to \eqref{eq:generalconslaw}.
It is immediately clear that the FV scheme, which is derived from this equation, is conservative by construction as cell averages change through the edge fluxes.
By integrating over time, we obtain the evolution equation for the cell averaged quantity $\overline{\vec{q}}_i$,
\begin{equation}\label{eq:evolutionequationcellaverage}
	\overline{\vec{q}}_i^{n+1} = \overline{\vec{q}}_i^n - \frac{1}{V_i}\int_{t^n}^{t^{n+1}}\int_{\partial V_i} \vec{f}(\vec{q}) \cdot \vec{n} \,\mathrm{d}S\,\mathrm{d}t +  \frac{1}{V_i}\int_{t^n}^{t^{n+1}}\int_{V_i} \vec{s} \, \mathrm{d}V \, \mathrm{d}t.
\end{equation}
For the derivation of the FV scheme, we now assume that the boundary of grid cell $i$ with volume $V_i$ is given by $k_i$ edges $K_{i,j}$ with one certain neighboring grid cell. Further, we make the assumption that the source term contributes at each of the cells edges. We can then rewrite \eqref{eq:evolutionequationcellaverage} into the form
\begin{equation}\label{eq:evolutionequationcellaverageedges}
\overline{\vec{q}}_i^{n+1} = \overline{\vec{q}}_i^n - \frac{1}{V_i}\int_{t^n}^{t^{n+1}} \sum_{j=1}^{k_i}\int_{K_{i,j}} \vec{f}(\vec{q}) \cdot \vec{n} \,\mathrm{d}S  \, \mathrm{d}t + \int_{t^n}^{t^{n+1}} \sum_{j=1}^{k_i}\int_{K_{i,j}} \vec{s} \, \mathrm{d}S \, \mathrm{d}t.
\end{equation}

For the sake of convenience, we assume a one-dimensional geometry with equidistant discretization of space and time with uniform step sizes $\Delta x$ and $\Delta t$. The grid intervals shall be given by $I_i = \left[x_{i-\halb},x_{\halb}\right] \ \forall\ i \in \mathbb{Z}$. Note that this can be done without loss of generality.

From \eqref{eq:evolutionequationcellaverageedges} we immediately obtain
\begin{equation}
	\overline{\vec{q}}_i^{n+1} = \overline{\vec{q}}_i^n - \frac{\Delta t}{\Delta x} \left(\fhat\!_{i+\halb} - \fhat\!_{i-\halb}\right) + \vec{s}_i,
\end{equation}
where the numerical flux function, $\fhat_{i+\halb}$, is a approximation of the physical flux function $\vec{f}(\vec{q})$ of the interface $i+\halb$ over the time interval $t \in [t^n, t^{n+1}]$ and
\begin{equation}
\vec{s_i} = \halb\left(\vec{s}_{i+\halb} + \vec{s}_{i-\halb}\right),
\end{equation}
is the source term approximation found by applying the trapezoid rule to the right hand side integral in \eqref{eq:finiteVolumeWithSource}. One particular difficulty that arises is that the numerical flux has to be computed at the interfaces, whereas the FV scheme is concerned with cell-centered quantities only. Hence, we have to \textit{reconstruct} the values at the cell faces to compute the numerical fluxes.

The most simple approximation for the numerical flux is when we assume that the face values are identical to the center values, i.e.
\begin{equation}
	\fhat_{i+\halb}^n = \fhat(\vec{q}_{i}^n,\vec{q}_{i+1}^n).
\end{equation}
Since the approximate solution is constant in each grid interval, this approximation is commonly called \textit{piecewise constant reconstruction}:
\begin{equation}\label{eq:piecewise-const}
\overline{\vec{q}}_{i,\mathrm{L}}^n = \overline{\vec{q}}_{i-1}^n, \qquad\text{and}\qquad\overline{\vec{q}}_{i,\mathrm{R}}^n = \overline{\vec{q}}_i^n.
\end{equation}

\begin{figure}[h]
	\centering
	\begin{subfigure}[b]{0.3\textwidth}
		\centering
		\includegraphics[scale=1]{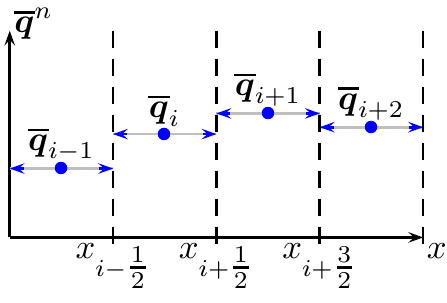}
		\caption*{\shortstack{Piecewise constant\\\eqref{eq:piecewise-const}}}
	\end{subfigure}
	\begin{subfigure}[b]{0.3\textwidth}
		\centering
		\includegraphics[scale=1]{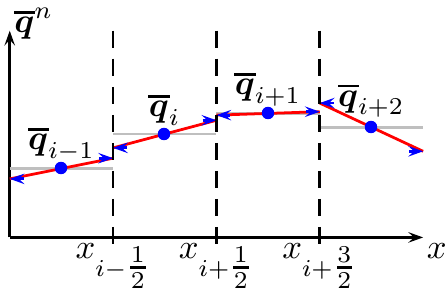}
		\caption*{\shortstack{Piecewise linear\\\eqref{eq:piecewise-linear} with \eqref{eq:alpha05}}}
	\end{subfigure}
	\begin{subfigure}[b]{0.3\textwidth}
		\centering
		\includegraphics[scale=1]{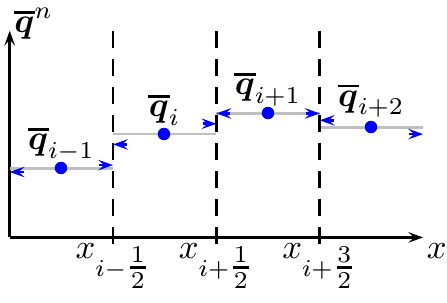}
		\caption*{\shortstack{\texttt{minmod}\\\eqref{eq:piecewise-linear} with \eqref{eq:minmod}}}
	\end{subfigure}
	\caption{Illustration of the three reconstruction schemes described here. The interface values used for the flux computation, $\overline{\vec{q}}_{i,\mathrm{R},\mathrm{L}}^n$ are indicated with blue arrows. Note that the piecewise linear reconstruction is violating the monotonicity of the solution (see $x_{i+\tfrac{3}{2}}$ interface).}
	\label{fig:Reconstruction}
\end{figure}

Unfortunately, it turns out that the piecewise constant reconstruction gives poor accuracy in smooth regions of the flow \cite{leveque2002}. Moreover, shocks tend to be heavily smeared and poorly resolved by the numerical grid. A more sophisticated and significantly more accurate reconstruction technique can be derived if we take neighboring cells into account for the approximation of the interface quantities. A particularly simple higher-order method is the \textit{piecewise linear reconstruction} \cite{vanleer1979}:
\begin{equation}\label{eq:higher-order}
	\overline{\vec{q}}^n(x) = \overline{\vec{q}}^n_i + (x-x_i) \boldsymbol{\varsigma}_i^n \quad \forall  x \in I_i,
\end{equation}
where $\boldsymbol{\varsigma}_i^n$ is a linear slope within cell $I_i = \left[x_{i-\halb},x_{i+\halb}\right]$ at time $t^n$. The interface values can now be computed using
\begin{equation}\label{eq:piecewise-linear}
	\overline{\vec{q}}_{i,\mathrm{R},\mathrm{L}}^n = \overline{\vec{q}}_i^n \pm \frac{1}{2} \Delta x \boldsymbol{\varsigma}_i^n.
\end{equation}
The interface values of cell $i$ are denoted with $\overline{\vec{q}}_{i,\mathrm{L}}$ and $\overline{\vec{q}}_{i,\mathrm{R}}$ at $x=x_{i-1/2}$, and $x=x_{i+1/2}$, restrictively.

The slope may be computed by using a generalized difference quotient with a free parameter $\alpha \in [0,1]$:
\begin{equation}
	\boldsymbol{\varsigma}_i^n = \alpha \frac{\overline{\vec{q}}_i^n - \overline{\vec{q}}_{i-1}^n}{\Delta x} + (1-\alpha) \frac{\overline{\vec{q}}_{i+1}^n - \overline{\vec{q}}_i^n}{\Delta x}.
\end{equation}
For $\alpha = 0.5$, this approximation is second-order accurate in space:
\begin{equation}\label{eq:alpha05}
	\boldsymbol{\varsigma}_i^n =\frac{\overline{\vec{q}}_{i+1}^n - \overline{\vec{q}}_{i-1}^n}{2\Delta x}.
\end{equation}
However, it has been shown that this reconstruction is not unconditionally stable for \textit{any} constant value of $\alpha$ \cite{Munz2009}. Therefore, one has to take special measures to ensure that the monotonicity of the cell averaged data is preserved. Using the ansatz given by \eqref{eq:alpha05} we see, e.g, in Fig.~\ref{fig:Reconstruction}B, that the monotonicity at the $x_{i+3/2}$ interface is violated. A commonly used \textit{slope limiter} that avoids spurious oscillations and violations of the original monotonicity is the \textit{minmod} limiter \cite{roe1986}. We define
\begin{equation}\label{eq:minmod}
	\boldsymbol{\varsigma}_i^n = \mathrm{minmod}\left( \frac{\overline{\vec{q}}_i^n - \overline{\vec{q}}_{i-1}^n}{\Delta x},\ \frac{\overline{\vec{q}}_{i+1}^n - \overline{\vec{q}}_i^n}{\Delta x}\right),
\end{equation}
where the minmod function is defined as
\begin{equation}
	\mathrm{minmod}(a,b) =
	\begin{cases}
		a & \text{if } |a| < |b| \ \text{ and } \ a \cdot b > 0,\\
		b & \text{if } |a| \ge |b| \ \text{ and } \ a \cdot b > 0,\\
		0 & \text{else.}
	\end{cases}
\end{equation}

In practical terms, this means that if the slopes on either side of the boundary have different signs, which could either indicate spurious oscillation or a (local) extremum, the value will be reconstructed using a slope of zero in order to avoid unphysical overshoots of the interpolation.
Similarly, if the slopes on either side have the same sign, the slope that is less steep is always chosen.
The general idea of the minmod algorithm is to restrict the ratio between adjacent jumps in cell averages where such jumps could indicate oscillatory behavior. Clearly, such a scheme will not develop spurious oscillations near a jump. 
Note that this, however, may smooth out extrema even when they are desired in the solution.

We note that the particular choice of the reconstruction is a very important for the accuracy and robustness of the resulting numerical FV scheme. In this work we restrict ourselves to the use of the simple minmod reconstruction \eqref{eq:minmod}, however many other reconstructions are available in the literature, e.g., \cite{barth1989,chakravarthy1983,roe1986,schmidtmann2016,sweby1984,vanleer1979}.

To complete the construction of the FV scheme we must find a suitable numerical flux function $\fhat$. In addition, we seek to construct a numerical flux function that conserves (or preserves the correct sign) of the entropy predicted by the continuous equations. The construction of such a numerical flux is the focus of the next section. As a final note on the discretization, it is important to guarantee that the numerical flux is a \textit{consistent} approximation of the physical flux. This requirement can be understood as a desire that the approximation will be exact if the value of $\vec{q}$ is identical on the left and right of the interface, i.e.,
\begin{equation}\label{eq:consistency}
	\fhat(\vec{q},\vec{q}) = {\vec{f}}(\vec{q}).
\end{equation}

We now have a numerical scheme for the system of hyperbolic conservation laws with a finite volume discretization in space and some type of explicit time integrator. The full discretization is completed once we prescribe a suitable, consistent numerical flux function $\fhat$. We apply the additional constraint that the resulting numerical flux must recover the correct entropic properties of the continuous problem. Much work has been done over the past 20 years to create entropy consistent numerical approximations for hyperbolic conservation laws (both scalar and systems) \cite{barth1999,Bouchut2007,carpenter_esdg,Chandrashekar2012,Chandrashekar2015,Chen2017,fisher2013,fisher2013_2,Fjordholm2011,Fjordholm2012_2,Fjordholm2016,gassner_skew_burgers,gassner2015,hiltebrand2014,hiltebrand2016,IsmailRoe2009,cell_entropy_dg,lefloch2000,Mishra2011,Ray2013,Ray2016,Tadmor1987_2,Tadmor2003,tadmor2016,Waagan2011,wintermeyer2017,Winters2016,winters2016_2}.

\subsection{Entropy conservative numerical flux for the ideal MHD equations}\label{sec:ECIdealMHD}

In the classical framework of Tadmor the discrete entropy conserving numerical flux is defined as an integral in phase space \cite{Tadmor1987_2,Tadmor2003}, i.e.,
\begin{equation}\label{eq:phaseSpace}
\fhat\!_{i+\halb} = \int\limits_{\xi=-\halb}^{\xi=+\halb} \vec{{f}}\left(\vec{v}_{i+\halb}(\xi)\right)\,\text{d}\xi,\quad \vec{v}_{i+\halb}(\xi)=\halb\left(\vec{v}_{i+1}+\vec{v}_i\right) + \xi\left(\vec{v}_{i+1} - \vec{v}_{i}\right).
\end{equation}
For non-linear problems, like the ideal MHD equations, the expression of the physical flux components in terms of the entropy variables are quite complicated. Thus, evaluating the integral \eqref{eq:phaseSpace} exactly is often impossible. Plus, evaluation of \eqref{eq:phaseSpace} with a sufficiently high accuracy quadrature rule to create $\fhat\!_{i+\halb}$ is computationally intensive. To avoid so much computational overhead, we will describe an affordable entropy conservative numerical flux function. Affordable entropy conservative flux functions are available for the shallow water equations \cite{Fjordholm2011,wintermeyer2017}, the compressible Euler equations \cite{Chandrashekar2012,IsmailRoe2009}, and the ideal MHD equations \cite{Chandrashekar2015,Winters2016}.

To derive an entropy conserving flux for the ideal MHD equations we will mimic the continuous entropy analysis from Sec. \ref{sec:contEnt} on the discrete level through the use of the discrete entropy analysis tools developed by Tadmor \cite{Tadmor1987_2,Tadmor2003}. We noted in Sec. \ref{sec:PDEs} that hyperbolic PDEs which model fluid mechanics are \textit{rotationally invariant}. Consider an arbitrary hyperbolic system of conservation laws
\begin{equation}\label{eq:consLawYetAgain}
\pderivative{\vec{q}}{t} + \pderivative{\vec{f}(\vec{q})}{x} + \pderivative{\vec{g(\vec{q})}}{y} + \pderivative{\vec{h(\vec{q})}}{z} = \vec{0},
\end{equation}
where we explicitly indicate that the fluxes $\vec{f}$, $\vec{g}$, and $\vec{h}$ are functions of the conserved quantities $\vec{q}$. Given a rotation matrix $\rotMat{}$, if it is true that that
\begin{equation}\label{eq:rotVarIdentity}
n_x\vec{f}+n_y\vec{g}+n_z\vec{h} = \rotMat{}^{-1}\vec{f}\left(\rotMat{}\vec{q}\right),
\end{equation}
for any non-zero vector $\vec{n} = [n_x,n_y,n_z]^T$ then the conservation law \eqref{eq:consLawYetAgain} is said to satisfy the rotational invariance property. For example, the ideal MHD equations possess the rotational invariance property, with the rotation matrix defined as
\begin{equation}\label{eq:rotMatMHD}
\rotMat{}=\begin{bmatrix}
1 & 0 & 0 & 0 & 0 & 0 & 0 & 0 \\
0 & n_x & n_y & n_z & 0 & 0 & 0 & 0 \\
0 & \ell_x & \ell_y & \ell_z & 0 & 0 & 0 & 0 \\
0 & m_x & m_y & m_z & 0 & 0 & 0 & 0 \\
0 & 0 & 0 & 0 & 1 & 0 & 0 & 0 \\
0 & 0 & 0 & 0 & 0 & n_x & n_y & n_z \\
0 & 0 & 0 & 0 & 0 & \ell_x & \ell_y & \ell_z \\
0 & 0 & 0 & 0 & 0 & m_x & m_y & m_z \\
\end{bmatrix},
\end{equation}
where $\vec{n}$, $\bm\ell$, and $\vec{m}$ are three orthogonal unit vectors. As the rotation matrix \eqref{eq:rotMatMHD} is orthogonal it is trivial to compute the inverse matrix to be 
\begin{equation}\label{eq:rotMatMHD2}
\rotMat{}^{-1} = \rotMat{}^T=\begin{bmatrix}
1 & 0 & 0 & 0 & 0 & 0 & 0 & 0 \\
0 & n_x & \ell_x & m_x & 0 & 0 & 0 & 0 \\
0 & n_y & \ell_y & m_y & 0 & 0 & 0 & 0 \\
0 & n_z & \ell_z & m_z & 0 & 0 & 0 & 0 \\
0 & 0 & 0 & 0 & 1 & 0 & 0 & 0 \\
0 & 0 & 0 & 0 & 0 & n_x & \ell_x & m_x \\
0 & 0 & 0 & 0 & 0 & n_y & \ell_y & m_y \\
0 & 0 & 0 & 0 & 0 & n_z & \ell_z & m_z \\
\end{bmatrix}.
\end{equation}
Now if we set $\vec{n} = [0,1,0]^T$, $\bm{\ell} = [1,0,0]^T$, and $\vec{m}=[0,0,1]^T$ we find that 
\begin{equation}\label{eq:gInvariance}
\vec{g} = \rotMat{}^T\vec{f}\left(\rotMat{}\vec{q}\right),\quad \rotMat{}=\rotMat{}^T =\begin{bmatrix}
1 & 0 & 0 & 0 & 0 & 0 & 0 & 0 \\
0 & 0 & 1 & 0 & 0 & 0 & 0 & 0 \\
0 & 1 & 0 & 0 & 0 & 0 & 0 & 0 \\
0 & 0 & 0 & 1 & 0 & 0 & 0 & 0 \\
0 & 0 & 0 & 0 & 1 & 0 & 0 & 0 \\
0 & 0 & 0 & 0 & 0 & 0 & 1 & 0 \\
0 & 0 & 0 & 0 & 0 & 1 & 0 & 0 \\
0 & 0 & 0 & 0 & 0 & 0 & 0 & 1 \\
\end{bmatrix}.
\end{equation}
Also, if we set $\vec{n} = [0,0,1]^T$, $\bm{\ell} = [0,1,0]^T$, and $\vec{m}=[1,0,0]^T$ we find that 
\begin{equation}\label{eq:gInvariance}
\vec{h} = \rotMat{}^T\vec{f}\left(\rotMat{}\vec{q}\right),\quad \rotMat{}=\rotMat{}^T =\begin{bmatrix}
1 & 0 & 0 & 0 & 0 & 0 & 0 & 0 \\
0 & 0 & 0 & 1 & 0 & 0 & 0 & 0 \\
0 & 0 & 1 & 0 & 0 & 0 & 0 & 0 \\
0 & 1 & 0 & 0 & 0 & 0 & 0 & 0 \\
0 & 0 & 0 & 0 & 1 & 0 & 0 & 0 \\
0 & 0 & 0 & 0 & 0 & 0 & 0 & 1 \\
0 & 0 & 0 & 0 & 0 & 0 & 1 & 0 \\
0 & 0 & 0 & 0 & 0 & 1 & 0 & 0 \\
\end{bmatrix}.
\end{equation}
That is, the explicit form of fluxes $\vec{g}$ and $\vec{h}$ do not need to be known. We recover their action from $\vec{f}$ and three orthogonal unit vectors. Thus, from the rotational invariance property we can find a flux vector projected in any direction completely from the flux vector $\vec{f}$.

Rotational invariance offers a significant advantage in the context of the FV method. That is, it is sufficient to consider each of the spatial directions to be split, i.e., we can immediately apply algorithms developed in the one-dimensional context to the multi-dimensional context, e.g., \cite{Derigs2016}. To apply the one-dimensional numerical flux into the $y$ or $z-$directions, any direction dependent quantities, i.e., the velocity and magnetic field components, are rotated such that the $x-$direction numerical flux remains valid. Quantities not listed in Table~\ref{tab:Rotation}, e.g., density, do not dependent on the directionality and consequently, do not require rotation.
\begin{table}[h]
	\centering
	\begin{tabular}{c|cccccc}
		\multicolumn{1}{c}{} & \multicolumn{6}{c}{Quantity} \\
		direction & $u$ & $v$ & $w$ & $B_1$ & $B_2$ & $B_3$ \\
		\midrule
		$x$ & $u$ & $v$ & $w$ & $B_1$ & $B_2$ & $B_3$ \\
		$y$ & $v$ & $u$ & $w$ & $B_2$ & $B_1$ & $B_3$ \\
		$z$ & $w$ & $v$ & $u$ & $B_3$ & $B_2$ & $B_1$ \\
	\end{tabular}
	\caption{Rotation of multi-dimensional quantities to compute the numerical flux using the one-dimensional formulation in $x-$direction.}
	\label{tab:Rotation}
\end{table}

This rotational invariance greatly simplifies the discrete entropy analysis of the ideal MHD equations as it is sufficient to design the numerical flux for the one-dimensional case. Note that we have to rotate three times for any update. First, we rotate a copy of the physical quantities to compute the numerical fluxes so that we can apply the fluxes as if they were to be computed in $x-$direction. Applying the rotation given in Table~\ref{tab:Rotation} twice results in the original reference frame of the computational domain. Hence, we apply the same rotation both on the computed numerical fluxes and the source term to rotate them back into the reference frame of the computational simulation. Although this procedure sounds computationally intense, it actually is not, since \emph{rotating} means only interchanging four quantities. Moreover, we gain the noteworthy advantage that we can apply the same flux and source term computation algorithms in all three directions. Not only does this save us from repeating the (algorithmically quite intense) derivation for the other directions, it also allows us to thoroughly test our implementation in comparable simple and standardized one-dimensional test cases. Once we are comfortable with the obtained performance, we can immediately apply this code to two and three-dimensional systems with a comparably negligible chance of introducing mistakes in the code that would come with the implementation of new $y$ and $z-$direction fluxes.

To begin let's assume that we have two adjacent states $(L,R)$ with cell areas $(\Delta x_L,\Delta x_R)$. We discretize the ideal MHD equations semi-discretely and examine the approximation at the $i+\tfrac{1}{2}$ interface. We suppress the interface index unless it is necessary for clarity. Also note the factor of one half from the source term discretization.
\begin{equation}\label{FVupdate}
\begin{aligned}
\Delta x_L\pderivative{\vec{q}_L}{t} &= \vec{f}_L - \fhat  + \halb\Delta x_L{\vec{s}}_{i+\tfrac{1}{2}},\\
\Delta x_R\pderivative{\vec{q}_R}{t} &= \fhat-\vec{f}_R + \halb\Delta x_R{\vec{s}}_{i+\tfrac{1}{2}}.
\end{aligned}
\end{equation}
We can interpret the update \eqref{FVupdate} as a finite volume scheme where we have left and right cell-averaged values separated by a common flux interface.

We premultiply the expressions \eqref{FVupdate} by the entropy variables to convert the system into entropy space. From the definition of the entropy variables \eqref{eq:entropyVariables} we know that $S_t = \vec{v}^T\vec{q}_t$, hence a semi-discrete entropy update is
\begin{equation}\label{EntropyUpdate}
\begin{aligned}
\Delta x_L\pderivative{S_L}{t} &= \vec{v}^T_L\left(\vec{f}_L - \fhat  +\halb\Delta x_L{\vec{s}}_{i+\tfrac{1}{2}}\right), \\
\Delta x_R\pderivative{S_R}{t} &= \vec{v}^T_R\left(\fhat-\vec{f}_R + \halb\Delta x_R{\vec{s}}_{i+\tfrac{1}{2}}\right).
\end{aligned}
\end{equation}
If we denote the jump in a state as $\jump{\cdot} = (\cdot)_R - (\cdot)_L$ and the average of a state as $\average{\cdot} = \halb((\cdot)_R + (\cdot)_L)$, then the total update will be
\begin{equation}\label{TotalUpdate}
\pderivative{}{t}(\Delta x_L S_L + \Delta x_R S_R) =  \jump{\vec{v}}^T\fhat-\jump{\vec{v}\cdot\vec{f}}  + \average{\Delta x\vec{v}}^T{\vec{s}}_{i+\tfrac{1}{2}}.
\end{equation}
We want the discrete entropy update to satisfy the discrete entropy conservation law. To achieve this, we require \cite{Fjordholm2011,Tadmor2003,Winters2016}
\begin{equation}\label{entropyConservationCondition1}
\jump{\vec{v}}^T\fhat -\jump{\vec{v}\cdot\vec{f}} + \average{\Delta x\vec{v}}^T{\vec{s}}_{i+\tfrac{1}{2}}  = -\jump{ F}.
\end{equation}
We combine the known entropy potential $\phi$ in \eqref{EntropyPotential} and the linearity of the jump operator to rewrite the entropy conservation condition \eqref{entropyConservationCondition1} as
\begin{equation}\label{entropyConservationCondition2}
\jump{\vec{v}}^T\fhat = \jump{ \rho u } +\jump{\frac{\rho u\|\vec{B}\|^2}{2p}} - \jump{\frac{\rho B_1(\vec{u}\cdot\vec{B})}{p}} - \average{\Delta x \vec{v}}^T{\vec{s}}_{i+\tfrac{1}{2}},
\end{equation}
and denote the constraint \eqref{entropyConservationCondition2} as the \textit{discrete entropy conserving condition}. This is a single condition on the numerical flux vector $\fhat$, so there are many potential solutions for the entropy conserving flux. Recall, however, that we have the additional requirement that the numerical flux must be consistent \eqref{eq:consistency}. In the following, we develop the expression for an entropy conserving flux $\fhat$ for the ideal MHD equations. We will see that the discretization of the source term ${\vec{s}_{i+\halb}}$ plays an important role to ensure that discrete entropy conservation \eqref{entropyConservationCondition2} is maintained.

With the definition of the entropy variables and the formulation of the discrete entropy conserving condition \eqref{entropyConservationCondition2} we are ready to derive a computationally affordable entropy conserving numerical flux. We have previously defined the arithmetic mean. To derive an entropy conserving flux we also require the logarithmic mean
\begin{equation}\label{logMean}
(\cdot)^{\ln} := \frac{(\cdot)_L - (\cdot)_R}{\ln((\cdot)_L) - \ln((\cdot)_R)}.
\end{equation}
A numerically stable procedure to compute the logarithmic mean is described by Ismail and Roe \cite[App. B]{IsmailRoe2009}.

Just as Chandrashekar \cite{Chandrashekar2012} did for the compressible Euler equations, we develop a baseline numerical flux function that is both kinetic energy preserving \cite{jameson2008} and entropy conserving (KEPEC) for the ideal MHD equations. We explicitly derive the flux in the $x-$direction, but generalization to higher dimensions is straightforward through symmetry arguments \cite{Derigs2016}. To do so we introduce notation for the inverse of the temperature
\begin{equation}\label{tempInverse}
\beta = \frac{1}{RT} = \frac{\rho}{2 p},
\end{equation}
therefore the entropy variables \eqref{eq:entropyVariables} are rewritten to be
\begin{equation}\label{entVarsEKEP}
\vec{v} = \left[ \frac{\gamma - s}{\gamma - 1} - \beta\|\vec{u}\|^2,\;2\beta u,\;2\beta v,\;2\beta w,\;-2\beta,\;2\beta B_1,\;2\beta B_2,\;2\beta B_3\right]^T.
\end{equation}
\begin{thm}[Kinetic Energy Preserving and Entropy Conserving (KEPEC) Numerical Flux]~\newline
If we define the logarithmic mean $(\cdot)^{\ln}$ \eqref{logMean}, the arithmetic mean $\average{\cdot}$, and discretize the source term in the finite volume method to contribute to each element as
\begin{equation}\label{SourceTermDiscyEKEP}
\vec{s}_{i} = \frac{1}{2}\left(\vec{s}_{i+\frac{1}{2}} + \vec{s}_{i-\frac{1}{2}} \right)= -\frac{1}{2}\left(
\jump{B_1}_{i+\tfrac{1}{2}}\begin{bmatrix}
0\\
0\\
0\\
0\\
0\\
\average{u}\frac{\average{\beta}\average{B_1}}{\average{\Delta x \beta B_1}}\\[0.15cm]
\average{v}\frac{\average{\beta}\average{B_2}}{\average{\Delta x \beta B_2}}\\[0.15cm]
\average{w}\frac{\average{\beta}\average{B_3}}{\average{\Delta x \beta B_3}}
\end{bmatrix}_{i+\tfrac{1}{2}}
+
\jump{B_1}_{i-\tfrac{1}{2}}\begin{bmatrix}
0\\
0\\
0\\
0\\
0\\
\average{u}\frac{\average{\beta}\average{B_1}}{\average{\Delta x \beta B_1}}\\[0.15cm]
\average{v}\frac{\average{\beta}\average{B_2}}{\average{\Delta x \beta B_2}}\\[0.15cm]
\average{w}\frac{\average{\beta}\average{B_3}}{\average{\Delta x \beta B_3}}
\end{bmatrix}_{i-\tfrac{1}{2}}
\right),
\end{equation}
then we determine a discrete entropy and kinetic energy conservative flux to be
\begin{equation}\label{Eq:entropyconservative-yEKEP}
\fhat^\mathrm{KEPEC} = \begin{bmatrix}
\rho^{\ln}\average{u} \\[0.15cm]
\rho^{\ln}\average{u}^2 + \frac{\average{\rho}}{2\average{\beta}} +\frac{1}{2}\left(\average{B_1^2}+\average{B_2^2}+\average{B_3^2}\right) - \average{B_1^2} \\[0.15cm]
\rho^{\ln}\average{u}\average{v} - \average{B_1B_2} \\[0.15cm]
\rho^{\ln}\average{u}\average{w} - \average{B_1B_3} \\[0.15cm]
\hat{f}_5\\[0.15cm]
0 \\[0.15cm]
\average{ u}\average{B_2} - \average{ v}\average{B_1} \\[0.15cm]
\average{ u}\average{B_3} - \average{ w}\average{B_1}
\end{bmatrix},
\end{equation}
where
\begin{equation}
\begin{aligned}
\hat{f}_5 &= \frac{\average{u}}{2}\left(\frac{\rho^{\ln}}{\beta^{\ln}(\gamma-1)} + \frac{\average{\rho}}{\average{\beta}} \right)+ \frac{\rho^{\ln}\average{u}}{2}\left(2\left(\average{u}^2+\average{v}^2+\average{w}^2\right)-\left(\average{u^2}+\average{v^2}+\average{w^2}\right)\right)\\[0.15cm]
&\qquad +\frac{\average{u}}{2}\left(\average{B_1^2}+\average{B_2^2}+\average{B_3^2}+2\left[\average{B_2}^2+\average{B_3}^2\right]\right) - \avg{u}\avg{B_1^2} - \avg{v}\avg{B_1B_2}\\[0.15cm]
&\qquad -\avg{w}\avg{B_1B_3} - \avg{v}\avg{B_1}\avg{B_2} - \avg{w}\avg{B_1}\avg{B_3} + \avg{uB_1^2}\\[0.15cm]
&\qquad +\avg{vB_1B_2}+\avg{wB_1B_3}-\halb\left(\avg{uB_1^2}+\avg{uB_2^2}+\avg{uB_3^2}\right).
\end{aligned}
\end{equation}
\end{thm}
\begin{proof}
To derive an affordable entropy conservative flux for the one-dimensional ideal MHD equations we first expand the discrete entropy conserving condition \eqref{entropyConservationCondition2} componentwise to find
\begin{equation}\label{componentJumpCondition}
\begin{aligned}
-\hat{f}_1\left(\frac{\jump{s}}{\gamma-1} + \jump{\beta\|\vec{u}\|^2}\right) + \hat{f}_2\jump{2\beta u}+ \hat{f}_3\jump{2\beta v}&+ \hat{f}_4\jump{2\beta w}- \hat{f}_5\jump{2\beta}+ \hat{f}_6\jump{2\beta B_1}+ \hat{f}_7\jump{2\beta B_2}+ \hat{f}_8\jump{2\beta B_3} \\ &=
\jump{\rho u} +\jump{\beta u\|\vec{B}\|^2} - \jump{2\beta B_1(\vec{u}\cdot\vec{B})} - \average{\Delta x \vec{v}}^T{\vec{s}}_{i+\tfrac{1}{2}}.
\end{aligned}
\end{equation}
To determine the unknown components of $\fhat$ we must expand each jump term in \eqref{componentJumpCondition} into linear jump components. This will provide us with a system of eight equations from which we determine $\fhat$. To do so we utilize two identities of the linear jump operator
\begin{equation}\label{eq:jumpProperties}
\jump{ab} = \average{a}\jump{b} + \average{b}\jump{a},\quad\jump{a^2} = 2\average{a}\jump{a}.
\end{equation}
We first note that the jump in the physical entropy $s$ written in terms of logarithmic mean \eqref{logMean} of $\rho$ and $\beta$ is
\begin{equation}
\jump{s} = \frac{\jump{\rho}}{\rho^{\ln}} + \frac{\jump{\beta}}{\beta^{\ln}(\gamma-1)}.
\end{equation}
Thus the left hand side of the expression \eqref{componentJumpCondition} expands to become
\begin{equation} \label{eq:LHS}
\begin{aligned}
& \hat{f}_{1}\left(\frac{\jump{\rho}}{\rholn} + \frac{\jump{\beta}}{\betaln (\gamma-1)} - \left(\average{u^2} + \average{v^2} + \average{w^2}\right) \jump{\beta} - 2\average{\beta}\Big(\average{u}\jump{u} + \average{v}\jump{v} + \average{w}\jump{w}\Big)\right)  \\
&+\, \hat{f}_2 \left( 2 \average{\beta}\jump{u} + 2 \average{u}\jump{\beta} \right) + \hat{f}_3 \left( 2 \average{\beta}\jump{v} + 2 \average{v}\jump{\beta} \right) + \hat{f}_4 \left( 2 \average{\beta}\jump{w} + 2 \average{w}\jump{\beta} \right) + \hat{f}_5 \left( - 2 \jump{\beta}\right)  \\
&+\, \hat{f}_6 \left(2 \average{\beta}\jump{B_1} + 2 \average{B_1}\jump{\beta} \right) +  \hat{f}_7 \left( 2 \average{\beta}\jump{B_2} + 2 \average{B_2}\jump{\beta} \right) + \hat{f}_8 \left( 2 \average{\beta}\jump{B_3} + 2 \average{B_3}\jump{\beta} \right).
\end{aligned}
\end{equation}
Next we expand the right hand side of \eqref{componentJumpCondition} into a combination of linear jumps:
\begin{equation}
\begin{aligned}\label{RHSCondition}
\jump{\rho u} &= \average{\rho}\jump{u} + \average{u}\jump{\rho}, \\[0.15cm]
\jump{\beta u \|\vec{B}\|^2} &= \jump{\beta} \big( \average{u B_1^2} + \average{u B_2^2} + \average{u B_3^2} \big) + \jump{u} \big( \average{\beta} \average{B_1^2} + \average{\beta} \average{B_2^2} + \average{\beta} \average{B_3^2} \big) \\
\quad&\quad+ \jump{B_1} \big( 2\average{\beta} \average{u} \average{B_1} \big) + \jump{B_2} \big( 2\average{\beta} \average{u} \average{B_2} \big) + \jump{B_3} \big( 2\average{\beta} \average{u} \average{B_3} \big), \\[0.1cm]
\jump{\beta B_1 (\vec{u}\cdot\vec{B})} &=\jump{\beta}\left( \average{u B_1 B_1} + \average{v B_1 B_2} + \average{w B_1 B_3} \right)   \\
\quad&\quad+ \jump{u} \left( \average{\beta} \average{B_1 B_1} \right) + \jump{v} \left( \average{\beta} \average{B_1 B_2} \right) + \jump{w} \left( \average{\beta} \average{B_1 B_3} \right) \\
\quad&\quad+ \jump{B_1} \average{\beta} \left(2\average{B_1} \average{u} + \average{B_2} \average{v} + \average{B_3} \average{w} \right)  \\
\quad&\quad+ \jump{B_2} \left( \average{\beta} \average{B_1} \average{v} \right) + \jump{B_3} \left( \average{\beta} \average{B_1} \average{w} \right),
\end{aligned}
\end{equation}
Finally, we expand the source term contribution on the right hand side of \eqref{entropyConservationCondition2}
\begin{equation}\label{sourceTermRemainder2}
-\average{\Delta x \vec{v}}^T\vec{s} = -\jump{B_1} \Big( \average{\Delta x (2\beta B_1)}s_6 + \average{\Delta x (2\beta B_2)}s_7 + \average{\Delta x (2\beta B_3)}s_8 \Big),
\end{equation}
where we leave the individual components of the source term general as a consistent approximation will reveal itself in the later analysis. Note that we used $\pderivative{B_1}{x} \approx \frac{B_{1,\text{R}} - B_{1,\text{L}}}{\Delta x} = \frac{\jump{B_1}}{\Delta x}$ in the discrete case.

Every term in the discrete entropy conservation condition \eqref{componentJumpCondition} is now rewritten into linear jump components. Though algebraically laborious this provides us with a set of eight equations for which we can determine the yet unknown components in the entropy conserving numerical flux. Next we gather the like terms of each jump component. Once we have grouped all the like terms for each linear jump it will become clear how to discretize the Janhunen source term in order to guarantee consistency.  Gathering terms from \eqref{eq:LHS}, \eqref{RHSCondition}, and \eqref{sourceTermRemainder2} we determine the system of eight equations to be
\begin{equation}\label{eq:jumpEquations}
\begin{aligned}
\jump{\rho} &: \frac{\hat{f}_1}{\rholn} = \average{u}, \\[0.15cm]
\jump{u} &: -2 \hat{f}_1 \average{\beta} \average{u} + 2 \hat{f}_2 \average{\beta}= \average{\rho} + \average{\beta}\big(\average{B_1^2} + \average{B_2^2} + \average{B_3^2}\big)  - 2 \average{\beta}\average{B_1^2},\\[0.15cm]
\jump{v} &: -2 \hat{f}_1 \average{\beta} \average{v}  + 2 \hat{f}_3 \average{\beta} = - 2 \average{\beta} \average{B_1 B_2},\\[0.15cm]
\jump{w} &: -2 \hat{f}_1 \average{\beta} \average{w} + 2 \hat{f}_4 \average{\beta} = - 2 \average{\beta} \average{B_1 B_3},\\[0.15cm]
\jump{B_1} &: 2 \hat{f}_6 \average{\beta}  = 2 \average{\beta} \average{B_1} \average{u}  - 2\average{\beta} \big( 2 \average{u} \average{B_1} + \average{v} \average{B_2} + \average{w} \average{B_3} \big) \\
&\qquad\qquad\qquad  -  2\average{\Delta x \beta B_1} s_6 - 2\average{\Delta x \beta B_2} s_7 - 2\average{\Delta x \beta B_3} s_8,\\[0.15cm]
\jump{B_2} &: 2 \hat{f}_7 \average{\beta} = 2 \average{\beta} \average{B_2} \average{u} - 2 \average{\beta} \average{B_1} \average{v},  \\[0.15cm]
\jump{B_3} &: 2 \hat{f}_8 \average{\beta} = 2 \average{\beta} \average{B_3} \average{u}  - 2 \average{\beta} \average{B_1} \average{w},  \\[0.15cm]
\jump{\beta} &: \frac{\hat{f}_1}{\betaln (\gamma-1)} - \hat{f}_1 \left(\average{u^2} + \average{v^2} + \average{w^2}\right) + 2 \hat{f}_2 \average{u}+ 2 \hat{f}_3 \average{v} + 2 \hat{f}_4 \average{w} \\
&\qquad  -2 \hat{f}_5 + 2 \hat{f}_6 \average{B_1} + 2 \hat{f}_7 \average{B_2}+ 2 \hat{f}_8 \average{B_3} \\
&\quad = \big(\average{u B_1^2}+\average{u B_2^2}+\average{u B_3^2}\big) -2 \big(\average{u B_1^2}+\average{v B_1 B_2}+\average{w B_1 B_3}\big).
\end{aligned}
\end{equation}

With the collection of equations \eqref{eq:jumpEquations} we find a rather alarming result. We know that the sixth component of the physical flux for the ideal MHD system is zero, i.e., $f_6 = 0$. However, we have found in our entropy conservation condition that the sixth component of the numerical flux must satisfy
\begin{equation}\label{eq:badEquation}
\begin{aligned}
\hat{f}_6 \average{\beta}  &= \average{\beta} \average{B_1} \average{u}  - \average{\beta} \big( 2 \average{u} \average{B_1} + \average{v} \average{B_2} + \average{w} \average{B_3} \big) \\
&\qquad\qquad\qquad  -  \average{\Delta x \beta B_1} s_6 - \average{\Delta x \beta B_2} s_7 - \average{\Delta x \beta B_3} s_8.
\end{aligned}
\end{equation}
In general, we cannot guarantee that the right hand side of \eqref{eq:badEquation} will vanish. In one spatial dimension the argument could be made that $\jump{B_1} = 0$ (as it is a constant \cite{freidberg1987}) and there is, in fact, no issue. However, this assumption is too restrictive to discuss higher dimensional extensions of the entropy conservative flux formulae. The assumption that $B_1\ne constant$ revealed extra terms which otherwise would have been hidden from the analysis.

To remove this inconsistency introduced by the $\jump{B_1}$ equation we discretize the source term to cancel the problematic terms \eqref{eq:badEquation}.
We compare the structure of the extra terms in \eqref{eq:badEquation} and the degrees of freedom $s_6$, $s_7$, and $s_8$ to determine a consistent discretization to cancel the extraneous terms in the $\jump{B_1}$ equation in \eqref{eq:badEquation}:
\begin{equation}
s_6 = - \average{u}\frac{\average{\beta} \average{B_1}}{\average{\Delta x \beta B_1}},\quad s_7 = -\average{v}\frac{\average{\beta} \average{B_2}}{\average{\Delta x \beta B_2}}, \quad s_8 = -\average{w}\frac{\average{\beta} \average{B_3}}{\average{\Delta x \beta B_3}}. \label{eq:janhunen}
\end{equation}
The source term at each interface of the cell has an identical structure. We collect the total source term discretization in cell $i$ for clarity:
 \begin{equation}\label{SourceInI}
 \vec{s}_i = \frac{1}{2}\left(\vec{s}_{i+\tfrac{1}{2}} + \vec{s}_{i-\tfrac{1}{2}}\right),
 \end{equation}
where
\begin{equation}\label{SourceTermDiscinProof2}
\begin{aligned}
\vec{s}_{i+\tfrac{1}{2}} &=-\jump{B_1}_{i+\tfrac{1}{2}}\left[
0,\;
0,\;
0,\;
0,\;
0,\;
\average{u}\frac{\average{\beta}\average{B_1}}{\average{\Delta x \beta B_1}},\;
\average{v}\frac{\average{\beta}\average{B_2}}{\average{\Delta x \beta B_2}},\;
\average{w}\frac{\average{\beta}\average{B_3}}{\average{\Delta x\beta B_3}}
\right]_{i+\tfrac{1}{2}}^T,\\
\vec{s}_{i-\tfrac{1}{2}} &=-\jump{B_1}_{i-\tfrac{1}{2}}\left[
0,\;
0,\;
0,\;
0,\;
0,\;
\average{u}\frac{\average{\beta}\average{B_1}}{\average{\Delta x \beta B_1}},\;
\average{v}\frac{\average{\beta}\average{B_2}}{\average{\Delta x \beta B_2}},\;
\average{w}\frac{\average{\beta}\average{B_3}}{\average{\Delta x\beta B_3}}
\right]_{i-\tfrac{1}{2}}^T.
\end{aligned}
\end{equation}
It is straightforward to check the consistency of the source term discretization \eqref{SourceInI}.

We substitute the source term discretization \eqref{eq:janhunen} into the entropy constraint \eqref{eq:jumpEquations} and find
that the source term components exactly cancel the extraneous terms in the $\jump{B_1}$ equation \eqref{eq:badEquation}. Thus, we recover a consistent term for the sixth numerical flux component and it is now true that
\begin{equation}
\hat{f}_6 = 0.
\end{equation}
Finally, we are able solve the remaining seven equations \eqref{eq:jumpEquations} and determine the components of the numerical flux to be
\begin{equation}\label{eq:newECFlux}
\begin{aligned}
\hat{f}_1 &= \rholn \average{u}, \\[0.15cm]
\hat{f}_2 &= \rholn \average{u}^2 + \frac{\average{\rho}}{2\average{\beta}} + \frac{1}{2} \Big(\average{B^2_1} + \average{B^2_2} + \average{B^2_3}\Big) - \average{B_1^2},\\[0.15cm]
\hat{f}_3&= \rholn\average{u} \average{v} - \average{B_1 B_2},\\[0.15cm]
\hat{f}_4 &= \rholn\average{u} \average{w} - \average{B_1 B_3},\\[0.15cm]
\hat{f}_5 &= \frac{\average{u}}{2}\left(\frac{\rho^{\ln}}{\beta^{\ln}(\gamma-1)} + \frac{\average{\rho}}{\average{\beta}} \right)+ \frac{\rho^{\ln}\average{u}}{2}\left(2\left(\average{u}^2+\average{v}^2+\average{w}^2\right)-\left(\average{u^2}+\average{v^2}+\average{w^2}\right)\right)\\[0.15cm]
&\qquad +\frac{\average{u}}{2}\left(\average{B_1^2}+\average{B_2^2}+\average{B_3^2}+2\left[\average{B_2}^2+\average{B_3}^2\right]\right) - \avg{u}\avg{B_1^2} - \avg{v}\avg{B_1B_2}\\[0.15cm]
&\qquad -\avg{w}\avg{B_1B_3} - \avg{v}\avg{B_1}\avg{B_2} - \avg{w}\avg{B_1}\avg{B_3} + \avg{uB_1^2}\\[0.15cm]
&\qquad +\avg{vB_1B_2}+\avg{wB_1B_3}-\halb\left(\avg{uB_1^2}+\avg{uB_2^2}+\avg{uB_3^2}\right),\\[0.15cm]
\hat{f}_6 &= 0,\\[0.15cm]
\hat{f}_7 &= \average{u}\average{B_2} - \average{v}\average{B_1},\\[0.15cm]
\hat{f}_8 &= \average{u}\average{B_3} - \average{w}\average{B_1}.
\end{aligned}
\end{equation}

The newly derived numerical flux \eqref{eq:newECFlux} conserves the discrete entropy by construction. Next, we verify that the numerical flux $\fhat^\mathrm{KEPEC}$ is consistent to the physical flux. It will make the demonstration of consistency more straightforward if we simplify the fifth component of the physical flux to be
\begin{equation}
\begin{aligned}
f_5 = {u}\left(\frac{\rho}{2}\|\vec{u}\|^2 + \frac{\gamma p}{\gamma - 1} + \|\vec{B}\|^2 \right) - \vec{B}(\vec{u}\cdot\vec{B}) &= \frac{\rho u}{2}\|\vec{u}\|^2 + \frac{\gamma u p}{\gamma-1} + u\|\vec{B}\|^2 - B_1(uB_1 + vB_2 + wB_3), \\
&= \frac{\rho u}{2}\|\vec{u}\|^2 + \frac{\gamma u p}{\gamma -1} + uB_2^2 + uB_3^2 - vB_1B_2 - wB_1B_3.
\end{aligned}
\end{equation}
Now, if we assume that the left and right states are identical in the numerical flux \eqref{eq:newECFlux}, we find that
\begin{equation} \label{eq;consistency}
\begin{aligned}
\hat{f}_1 &\rightarrow \rho u &= f_1,\\[0.15cm]
\hat{f}_2 &\rightarrow p + \rho u^2 + \frac{1}{2}\|\vec{B}\|^2 - B_1^2 &= f_2,\\[0.15cm]
\hat{f}_3 &\rightarrow \rho uv - B_1 B_2 &= f_3,\\[0.15cm]
\hat{f}_4 &\rightarrow \rho uw - B_1B_3 &= f_4,\\[0.15cm]
\hat{f}_5 &\rightarrow  \frac{\rho u}{2}\|\vec{u}\|^2 +\frac{\gamma u p}{\gamma-1}+ uB_2^2 + uB_3^2 - vB_1B_2 - wB_1B_3 &= f_5,\\[0.1cm]
\hat{f}_6 &\rightarrow 0 &= f_6,\\[0.15cm]
\hat{f}_7 &\rightarrow uB_2 - vB_1 &= f_7 ,\\[0.15cm]
\hat{f}_8 &\rightarrow uB_3 - wB_1 &= f_8.
\end{aligned}
\end{equation}
Thus, we have shown that the numerical flux given by \eqref{eq:newECFlux} is consistent and entropy conservative.
\end{proof}

\subsection{Entropy stable numerical flux for the ideal MHD equations}\label{sec:ESIdealMHD}

The KEPEC numerical flux presented so far conserves the entropy in the semi-discrete approximation. However, the solution of the ideal MHD equations (or any hyperbolic PDE for that matter) may develop discontinuities, e.g., shock waves or contact discontinuities, in finite time even for smooth initial data. We know in the presence of discontinuities that the conservation law for the entropy function \eqref{eq:continuousEntCons} must be replaced by the entropy inequality \eqref{eq:continuousEntInq}, e.g., \cite{Tadmor1987_2}. Thus, we seek to add numerical dissipation to the KEPEC flux so that the entropy is guaranteed to be dissipated (or conserved for smooth, well-resolved solutions), thereby ensuring that a discrete version of the entropy inequality holds. A typical way to add dissipation in an FV approximation is via the definition of the numerical flux function.

To create a kinetic energy preserving and entropy stable (KEPES) numerical flux function we use the entropy conserving flux \eqref{Eq:entropyconservative-yEKEP} as a baseline and subtract a general form of numerical dissipation, e.g.,
\begin{equation}\label{dissipation}
\fhat^\mathrm{KEPES} = \fhat^\mathrm{KEPEC} - \frac{1}{2}\mat{D}\jump{\vec{q}},
\end{equation}
where $\mat{D}$ is a dissipation matrix. The form of dissipation \eqref{dissipation} is motivated by the fact that we want less dissipation in regions of the flow that are smooth (or well-resolved) and more dissipation in regions containing discontinuities. Of utmost concern for entropy stability of the approximation is to formulate the dissipation term \eqref{dissipation} such that it is guaranteed to cause a negative contribution in the discrete entropy equation \eqref{TotalUpdate}.

We select the dissipation matrix to be the absolute value of the flux Jacobian for the ideal MHD 8-wave formulation \cite{godunov1972,powell1994,Powell1999} in the $x-$direction:
\begin{equation}\label{fluxJacobianConservativeVars}
\mat{D} := \bigg|\frac{\partial\vec{f}}{\partial\vec{q}} + \mat{P}\bigg| = |\mat{A} + \mat{P}|,
\end{equation}
where $\mat{A}$ is the flux Jacobian for the homogeneous ideal MHD equations and $\mat{P}$ is the Powell source term written in matrix form, i.e.,
\begin{equation}\label{PowellMatrix}
\mat{P}\frac{\partial\vec{q}}{\partial x} = \begin{bmatrix}
0 & 0 & 0 & 0 & 0 & 0 & 0 & 0 \\
0 & 0 & 0 & 0 & 0 & B_1 & 0 & 0 \\
0 & 0 & 0 & 0 & 0 & B_2 & 0 & 0 \\
0 & 0 & 0 & 0 & 0 & B_3 & 0 & 0 \\
0 & 0 & 0 & 0 & 0 & \vec{u}\cdot\vec{B} & 0 & 0 \\
0 & 0 & 0 & 0 & 0 & u & 0 & 0 \\
0 & 0 & 0 & 0 & 0 & v & 0 & 0 \\
0 & 0 & 0 & 0 & 0 & w & 0 & 0
\end{bmatrix}
\pderivative{}{x}\begin{bmatrix}
\rho \\
\rho u\\
\rho v \\
\rho w\\
E \\
B_1\\
B_2\\
B_3
\end{bmatrix} = \frac{\partial B_1}{\partial x}\begin{bmatrix}
0 \\
B_1 \\
B_2\\
B_3\\
\vec{u}\cdot\vec{B} \\
u\\
v\\
w
\end{bmatrix}.
\end{equation}
The fact that we force the positivity of $\mat{D} = |{\mat{A}}+\mat{P}|$ does not guarantee that the dissipative term \eqref{dissipation}
is negative \cite{barth1999}. Thus, in the remaining sections we will motivate a reformulation of the dissipation term \eqref{dissipation} in order to restore negativity.

It is important to distinguish that we use the Janhunen source term \eqref{eq:JanhunenSource} to derive an entropy conservative numerical flux function in Sec.~\ref{sec:ECIdealMHD}. However, to design an entropy stable approximation we require that the eigendecomposition of the flux Jacobian matrix can be related to the entropy Jacobian \eqref{entropyJacobian}. This particular scaling, first examined by Merriam \cite{merriam1989} and explored more thoroughly by Barth \cite{barth1999}, requires the PDE system to be symmetrizable. Previous analysis of the ideal MHD equations \cite{barth1999,godunov1972} has demonstrated that the Powell source term is necessary to restore a symmetric ideal MHD system. We reiterate that the altered flux Jacobian is used only to derive the dissipation term. Just as Lax-Friedrichs \cite{leveque1992} differs from Roe \cite{roe1981} in the construction of a dissipation term, we use the Powell source term only to build our dissipation term. Thus, no inconsistency with the previous entropy conserving flux derivations is introduced.

\subsubsection{Eigenstructure of the dissipation matrix}
The background discussion of the eigenstructure of the augmented flux Jacobian matrix ${\mat{D}}$ \eqref{fluxJacobianConservativeVars} is algebraically intense and follows the steps:
\begin{enumerate}
\item[1.] We compute the eigendecomposition for the symmetrizable MHD system written in the primitive variables.\\
\item[2.] We use previous results from Roe and Balsara \cite{roe1996} and rescale the eigenvectors to remove degeneracies.\\
\item[3.] We recover the, now stabilized, eigendecomposition for the matrix ${\mat{D}}$.
\end{enumerate}
For brevity we provide minimal details of the derivation and present the final result for the eigendecomposition of $\mat{D}$. However, complete details can be found in \cite{Winters2016}.

To discuss the eigenstructure of the matrix ${\mat{D}}$ it is easiest to work with primitive variables, which we denote $\boldsymbol{\omega}$, and convert back to conservative variables, denoted by $\vec{q}$, when necessary. We first write the ideal MHD system modified by the Powell source term in terms of the conservative variables
\begin{equation}\label{MHDPowellSystem}
\frac{\partial\vec{q}}{\partial t} + {\mat{D}}\frac{\partial \vec{q}}{\partial x} = \vec{0}.
\end{equation}
We are free to move between primitive and conservative variables in the system \eqref{MHDPowellSystem} with the matrix
\begin{equation}\label{MMatrix}
\mat{M} = \frac{\partial \vec{q}}{\partial\boldsymbol{\omega}} =
\begin{bmatrix}
1 & 0 & 0 & 0 & 0 & 0 & 0 & 0 \\
u & \rho & 0 & 0 & 0 &0&  0 & 0 \\
v & 0 & \rho & 0 & 0 &0& 0 & 0 \\
w & 0 & 0 & \rho & 0 & 0 &0& 0 \\
\frac{\|u\|^2}{2} & \rho u & \rho v & \rho w & \frac{1}{\gamma - 1} &B_1& B_2& B_3 \\
0 & 0 & 0 & 0 & 0 & 1 & 0&0 \\
0 & 0 & 0 & 0 & 0 & 0 &1 &0 \\
0 & 0 & 0 & 0 & 0 & 0 & 0&1
\end{bmatrix},
\end{equation}
and from conservative to primitive variables with $\mat{M}^{-1}$. Then we can rewrite the system \eqref{MHDPowellSystem} in terms of the vector of primitive variables $\vec{\omega}$
\begin{equation}\label{MHDPowellPrim}
\frac{\partial\boldsymbol{\omega}}{\partial t} + \mat{B}\frac{\partial\boldsymbol{\omega}}{\partial x} = \vec{0},
\end{equation}
where
\begin{equation}\label{BtoA}
{\mat{B}} = \mat{M}^{-1}\mat{D}\mat{M}.
\end{equation}

To describe the eigenstructure of the symmetrizable ideal MHD system flux Jacobian in conservative variables ${\mat{D}}$ we first investigate the eigendecompostion of ${\mat{B}}$, the flux Jacobian in primitive variables:
\begin{equation}\label{Bmatrix}
{\mat{B}} =
\begin{bmatrix}
u & \rho & 0 & 0 & 0 &0& 0 & 0 \\
0 & u & 0 & 0 & \frac{1}{\rho}&0&\frac{B_2}{\rho} & \frac{B_3}{\rho}  \\
0 & 0 & u & 0 & 0&0& -\frac{B_1}{\rho} & 0  \\
0 & 0 & 0 & u & 0 &0&0& -\frac{B_1}{\rho}  \\
0 & \gamma p & 0 & 0 & u &0& 0 & 0 \\
0 & 0 & 0 & 0 & 0&u& 0 & 0 \\
0 & B_2 & -B_1 & 0 &0& 0 & u & 0 \\
0 & B_3 & 0 & -B_1 & 0&0 & 0& u
\end{bmatrix}.
\end{equation}
From \eqref{BtoA} we see that we can convert the resulting eigendecomposition to conservative variables with straightforward matrix multiplication and the identity
\begin{equation}
\mat{D} = \mat{M}\mat{B}\mat{M}^{-1}.
\end{equation}
So, once we compute the eigendecomposition
\begin{equation}
{\mat{B}} = \mat{R}{\boldsymbol{\Lambda}}\mat{R}^{-1},
\end{equation}
we can recover the eigendecomposition of the flux Jacobian matrix in conservative variables as
\begin{equation}
{\mat{D}} = \mat{M}\mat{B}\mat{M}^{-1} = \mat{M}\mat{R}{\boldsymbol{\Lambda}}\mat{R}^{-1}\mat{M}^{-1}=\widehat{\mat{R}}{\boldsymbol{\Lambda}}\widehat{\mat{R}}^{-1},\quad \widehat{\mat{R}} = \mat{M}\mat{R}.
\end{equation}

Forgoing a large amount of algebra, we present the eigendecomposition of the matrix ${\mat{D}}$. The 8-wave formulation supports eight traveling wave solutions
\begin{itemize}
\item two fast magnetoacoustic waves ($\pm f$),
\item two slow magnetoacoustic waves ($\pm s$),
\item two Alfv\'{e}n waves ($\pm a$),
\item an entropy wave ($ E$),
\item a divergence wave ($D$).
\end{itemize}
with eigenvalues
\begin{equation}\label{eigenvalues}
\lambda_{\pm f} = u \pm c_f,\quad \lambda_{\pm s} =u \pm c_s,\quad \lambda_{\pm a} = u+c_a,\quad \lambda_{E} =u,\quad \lambda_{D} =u,
\end{equation}
where $c_f$, $c_s$ are the fast and slow magnetoacoustic wave speeds and $c_a$ is the Alfv\'{e}n wave speed. The double eigenvalue $u$ represent the entropy wave and divergence wave. The divergence wave is a direct result of including the Powell source. This is because the Powell source term turns the divergence wave into an advected scalar which is directly reflected in the eigenstructure. The values for the characteristic wave speeds may be written as
\begin{equation}\label{characteristicSpeeds}
c_a^2 = b_1^2,\qquad c_{f,s}^2 = \frac{1}{2}(a^2+b^2)\pm\frac{1}{2}\sqrt{(a^2+b^2)^2 - 4a^2b_1^2},
\end{equation}
with the conventional notation
\begin{equation}
\vec{b} = \frac{\vec{B}}{\sqrt{\rho}},\quad b^2=b_1^2+b_2^2+b_3^2,\quad b_{\perp}^2 = b_2^2+b_3^2,\quad a^2 = \frac{p\gamma}{\rho}.
\end{equation}
In \eqref{characteristicSpeeds} the plus sign is for the fast speed $c_f$ and minus sign is the slow speed $c_s$. It is known that the right eigenvectors of $\mat{D}$ may exhibit several forms of degeneracy that are carefully described by Roe and Balsara \cite{roe1996}. We follow the same rescaling procedure of Roe and Balsara for the fast/slow magnetoacoustic eigenvectors. The algebra is simplified greatly if we introduce the parameters
\begin{equation}\label{rescaleParams}
\alpha_f^2 = \frac{a^2 - c_s^2}{c_f^2 - c_s^2},\quad\alpha_s^2 = \frac{c_f^2 - a^2}{c_f^2 - c_s^2}.
\end{equation}
The parameters \eqref{rescaleParams} have several useful properties:
\begin{equation}\label{rescaleIdent}
\alpha_f^2+\alpha_s^2 = 1,\quad \alpha_f^2c_f^2 + \alpha_s^2c_s^2 = a^2,\quad \alpha_f\alpha_s = \frac{a^2b_{\perp}}{c_f^2-c_s^2}.
\end{equation}
The parameters $\alpha_{f,s}$ measure how closely the fast/slow waves approximate the behavior of acoustic waves. In the rescaling process we utilize several identities that arise from the quartic equation for the magnetoacoustic wave speeds $\pm c_{f,s}$
\begin{equation}
c^4 -(a^2+b^2)c^2 + a^2b_1^2 = 0,
\end{equation}
which are
\begin{equation}\label{identities2}
\begin{aligned}
c_fc_s &= a|b_1|, \\[0.1cm] c_f^2 + c_s^2 &= a^2 + b^2, \\[0.1cm] c_{f,s}^4-a^2b_1^2 &= c_{f,s}^2\left(c_{f,s}^2 - c_{s,f}^2\right),\\[0.1cm] \left(c_{f,s}^2 - a^2\right)\left(c_{f,s}^2 - b_1^2\right) &= c_{f,s}^2b_{\perp}^2.
\end{aligned}
\end{equation}
Applying the identities \eqref{rescaleIdent} and \eqref{identities2} it is possible to rewrite the eigenvectors for the fast/slow waves in a more stable form in terms of the parameters $\alpha_{f,s}$  \cite{roe1996}.

The matrix of right eigenvectors is given by
\begin{equation}\label{rightEV}
\widehat{\mat{R}} = \left[\,\hat{\vec{r}}_{+{ f}} \,|\, \hat{\vec{r}}_{+{ a}} \,|\, \hat{\vec{r}}_{+{ s}} \,|\, \hat{\vec{r}}_{ E} \,|\, \hat{\vec{r}}_{ D} \,|\, \hat{\vec{r}}_{-{ s}} \,|\, \hat{\vec{r}}_{-{ a}} \,|\, \hat{\vec{r}}_{-{ f}} \, \right],
\end{equation}
with the eigenvectors $\hat{\vec{r}}$, and corresponding eigenvalues $\lambda$ \cite{barth1999,Derigs2016,roe1996}
\begin{itemize}
\item[] \underline{Entropy and Divergence Waves}: $\lambda_{ E,D} = u$
\begin{equation}\label{entropyAS1}
\hat{\vec{r}}_{ E} = \begin{bmatrix} 1 \\ u \\v \\w \\ \frac{\|\vec{u}\|^2}{2} \\[0.05cm]0 \\0 \\0 \end{bmatrix},\quad\hat{\vec{r}}_{ D} = \begin{bmatrix} 0 \\ 0 \\0 \\0 \\B_1 \\[0.05cm]1 \\0 \\0 \end{bmatrix},
\end{equation}
\item[] \underline{Alfv\'{e}n Waves}: $\lambda_{\pm a} = u\pm b_1$
\begin{equation}\label{AlfvenAS1}
\hat{\vec{r}}_{\pm  a} = \begin{bmatrix}
0 \\
0 \\
\pm \rho^{\frac{3}{2}}\,\beta_3 \\
\mp \rho^{\frac{3}{2}}\,\beta_2 \\
\mp \rho^{\frac{3}{2}}(\beta_2 w - \beta_3 v) \\
0 \\
-\rho \beta_3 \\
\rho \beta_2
\end{bmatrix},
\end{equation}
\item[] \underline{Magnetoacoustic Waves}: $\lambda_{ \pm f,\pm s} = u\pm c_{ f,s}$
\begin{equation}\label{MHDAS1}
\hat{\vec{r}}_{\pm  f} = \begin{bmatrix}
\alpha_{ f}\rho \\[0.1cm]
\alpha_{ f}\rho(u \pm c_{ f}) \\[0.1cm]
\rho\left(\alpha_{ f} v \mp \alpha_{ s} c_{ s} \beta_2 \sigma(b_1) \right) \\[0.1cm]
\rho\left(\alpha_{ f} w \mp \alpha_{ s} c_{ s} \beta_3 \sigma(b_1) \right) \\[0.1cm]
\Psi_{\pm f} \\[0.1cm]
0 \\[0.1cm]
\alpha_{ s} a \beta_2 \sqrt{\rho} \\[0.1cm]
\alpha_{ s} a \beta_3 \sqrt{\rho}
\end{bmatrix},
\qquad
\hat{\vec{r}}_{\pm s} = \begin{bmatrix}
\alpha_{ s}\rho \\[0.1cm]
\alpha_{ s}\rho\left(u \pm c_{ s}\right) \\[0.1cm]
\rho\left(\alpha_{ s} v \pm \alpha_{ f} c_{ f} \beta_2 \sigma(b_1)\right) \\[0.1cm]
\rho\left(\alpha_{ s} w \pm \alpha_{ f} c_{ f} \beta_3 \sigma(b_1)\right) \\[0.1cm]
\Psi_{\pm s} \\[0.1cm]
0 \\[0.1cm]
-\alpha_{ f} a \beta_2 \sqrt{\rho} \\[0.1cm]
-\alpha_{ f} a \beta_3 \sqrt{\rho}
\end{bmatrix},
\end{equation}
\end{itemize}
where we introduced several convenience variables
\begin{equation}\label{eq:alotofequations1}
\begin{aligned}
\Psi_{\pm{s}} &= \frac{\alpha_{s} \rho \lVert\vec{u}\rVert^2}{2} - a \alpha_{f} \rho b_\perp + \frac{\alpha_{s} \rho a^2}{\gamma-1} \pm \alpha_{s} c_{s} \rho u \pm \alpha_{f} c_{f} \rho \sigma(b_1) (v \beta_2 + w \beta_3), \\
\Psi_{\pm{f}} &= \frac{\alpha_{f} \rho \lVert\vec{u}\rVert^2}{2} + a \alpha_{s} \rho b_\perp + \frac{\alpha_{f} \rho a^2}{\gamma-1} \pm \alpha_{f} c_{f} \rho u \mp \alpha_{s} c_{s} \rho \sigma(b_1) (v \beta_2 + w \beta_3),  \\
c_{a}^2& = b_1^2, \quad c_{f,s}^2 = \frac{1}{2}\left((a^2+b^2) \pm \sqrt{(a^2+b^2)^2 - 4a^2 b_1^2}\right), \quad a^2 = \gamma \, \frac{p}{\rho}, \\
b^2 &= b_1^2 + b_2^2 + b_3^2, \quad b_\perp^2 = b_2^2 + b_3^2, \quad \vec{b} = \frac{\vec{B}}{\sqrt{\rho}}, \quad \beta_{1,2,3} = \frac{b_{1,2,3}}{b_\perp},\\
\alpha_{f}^2 &= \frac{a^2 - c_{s}^2}{c_{f}^2 - c_{s}^2}, \quad \alpha_{s}^2 = \frac{c_{f}^2 -a^2}{c_{f}^2 - c_{s}^2},\quad
\sigma(\omega) = \begin{cases}
+1 &\mbox{if } \omega \ge 0, \\
-1 &\text{otherwise}
\end{cases}.
\end{aligned}
\end{equation}

\subsubsection{Entropy scaled right eigenvectors}\label{Sec:EntropySclaed}

We have a symmetrizable matrix ${\mat{D}}$ with a complete set of eigenvalues and right eigenvectors. We next utilize a previous result from Barth \cite{barth1999} which provides a systematic approach to restructure a general eigenvalue problem to a symmetric eigenvalue problem. To do so we rescale the right eigenvectors of an eigendecomposition with respect to a right symmetrizer matrix in the following way:
\begin{lem}[Eigenvector Scaling]~\newline
Let $\mat{A}\in\mathbb{R}^{n\times n}$ be an arbitrary diagonalizable matrix and $S$ the set of all right symmetrizers:
\begin{equation}
S=\left\{\mat{B}\in\mathbb{R}^{n\times n}\,\big|\;\mat{B}\;is\; s.p.d,\;\;\mat{AB} = (\mat{AB})^T\right\}.
\end{equation}
Further, let $\mat{R}\in\mathbb{R}^{n\times n}$ denote the right eigenvector matrix which diagonalizes $\mat{A}$, i.e., $\mat{A}=\mat{R}\boldsymbol\Lambda\mat{R}^{-1}$, with $r$ distinct eigenvalues. Then for each $\mat{B}\in S$ there exists a symmetric block diagonal matrix $\mat{T}$ that block scales columns of $\mat{R}$, $\widetilde{\mat{R}} = \mat{RT}$, such that
\begin{equation}
\mat{B}=\widetilde{\mat{R}}\widetilde{\mat{R}}^T,\; \mat{A}=\widetilde{\mat{R}}\boldsymbol\Lambda\widetilde{\mat{R}}^{-1},
\end{equation}
which implies
\begin{equation}
\mat{AB}=\widetilde{\mat{R}}\boldsymbol\Lambda\widetilde{\mat{R}}^{T}.
\end{equation}
\end{lem}

\begin{proof} The proof of the eigenvector scaling lemma is given in \cite{barth1999}.\end{proof}

\begin{thm}[Kinectic Energy Preserving and Entropy Stable (KEPES) Numerical Flux]~\newline
	If we apply the diagonal scaling matrix
\begin{equation}\label{scalingMatrix}
\mat{T} = \text{diag}\left(\frac{1}{\sqrt{2\rho\gamma}},\,\sqrt{\frac{p}{2\rho^3}},\,\frac{1}{\sqrt{2\rho\gamma}},\,\sqrt{\frac{\rho(\gamma-1)}{\gamma}},\,\sqrt{\frac{p}{\rho}},\,\frac{1}{\sqrt{2\rho\gamma}},\,\sqrt{\frac{p}{2\rho^3}},\,\frac{1}{\sqrt{2\rho\gamma}}\right),
\end{equation}
to the matrix of right eigenvectors $\widehat{\mat{R}}$ \eqref{rightEV}, then we obtain the identity \cite{barth1999,merriam1989}
\begin{equation}\label{MerriamIdentity}
\mat{H} = \widetilde{\mat{R}}\widetilde{\mat{R}}^T = \left(\widehat{\mat{R}}\mat{T}\right) \left(\widehat{\mat{R}}\mat{T}\right)^T = \widehat{\mat{R}}\mat{Z}\widehat{\mat{R}}^T,
\end{equation}
that relates the right eigenvectors of ${\mat{D}}$ to the entropy Jacobian matrix \eqref{entropyJacobian}. For convenience, we introduce the diagonal scaling matrix $\mat{Z}=\mat{T}\,^2$ in \eqref{MerriamIdentity}. We then have the guaranteed entropy stable flux interface contribution
\begin{equation}\label{minimalDiss}
\fhat^\mathrm{KEPES} = \fhat^\mathrm{KEPEC} - \halb\widehat{\mat{R}}|{\boldsymbol\Lambda}|\mat{Z}\widehat{\mat{R}}^T\jump{\vec{v}}.
\end{equation}
\end{thm}

\begin{proof}
We define the dissipation term in the numerical flux \eqref{dissipation} to be
\begin{equation}\label{dissTerm1}
-\frac{1}{2}\mat{D}\jump{\vec{q}} = -\frac{1}{2} \widehat{\mat{R}}|{\boldsymbol\Lambda}|\widehat{\mat{R}}^{-1}\jump{\vec{q}},
\end{equation}
where the eigendecomposition of ${\mat{D}}$ is given by \eqref{eigenvalues} and \eqref{rightEV}. We define entropy stability to mean the approximation guarantees that the entropy within the system is a decreasing function, satisfying the following inequality
\begin{equation}
\pderivative{S}{t} + \pderivative{F}{x} - \vec{v}^T\vec{s} \leq 0.
\end{equation}
From the previously computed discrete entropy update \eqref{TotalUpdate} and the condition \eqref{entropyConservationCondition1} we find the total entropy update (now including the dissipation term \eqref{dissTerm1}) to be
\begin{equation}\label{TotalUpdate2}
\begin{aligned}
\pderivative{}{t}(\Delta x_L S_L + \Delta x_R S_R) &=  \jump{\vec{v}}^T\fhat^\mathrm{KEPES} -\jump{\vec{v}\cdot\vec{f}}  + \average{\Delta x\vec{v}}^T{\vec{s}}_{i+\tfrac{1}{2}}, \\
\pderivative{}{t}(\Delta x_L S_L + \Delta x_R S_R) &=  \jump{\vec{v}}^T\fhat^\mathrm{KEPEC} - \frac{1}{2}\jump{\vec{v}}^T \widehat{\mat{R}}|{\boldsymbol\Lambda}|\widehat{\mat{R}}^{-1}\jump{\vec{q}}-\jump{\vec{v}\cdot\vec{f}} + \average{\Delta x\vec{v}}{^T}{\vec{s}}_{i+\tfrac{1}{2}}, \\
\pderivative{}{t}(\Delta x_L S_L + \Delta x_R S_R) &=-\jump{F}- \frac{1}{2}\jump{\vec{v}}^T \widehat{\mat{R}}|{\boldsymbol\Lambda}|\widehat{\mat{R}}^{-1}\jump{\vec{q}}, \\
\pderivative{}{t}(\Delta x_L S_L + \Delta x_R S_R) + \jump{F} &= - \frac{1}{2}\jump{\vec{v}}^T \widehat{\mat{R}}|{\boldsymbol\Lambda}|\widehat{\mat{R}}^{-1}\jump{\vec{q}},
\end{aligned}
\end{equation}
due to the design condition \eqref{entropyConservationCondition1} on the kinetic energy preserving and entropy conserving flux, $\fhat^\mathrm{KEPEC}$. To ensure entropy stability, we must guarantee that the right hand side in \eqref{TotalUpdate2} is non-positive. Unfortunately, it was shown by Barth \cite{barth1999} that the term
\begin{equation}\label{RHSEntropy}
- \frac{1}{2}\jump{\vec{v}}^T \widehat{\mat{R}}|{\boldsymbol\Lambda}|\widehat{\mat{R}}^{-1}\jump{\vec{q}},
\end{equation}
may become positive in the presence of very strong shocks. However, we know from entropy symmetrization theory, e.g \cite{barth1999,merriam1989}, that the entropy Jacobian $\mat{H}$, given by \eqref{entropyJacobian}, is a right symmetrizer for the flux Jacobian that incorporates the Powell source term ${\mat{D}}$. Therefore, with the proper scaling matrix $\mat{T}$ we acquire the identity
\begin{equation}\label{MerriamIdentityinproof}
\mat{H} = \widetilde{\mat{R}}\widetilde{\mat{R}}^T = \left(\widehat{\mat{R}}\mat{T}\right) \left(\widehat{\mat{R}}\mat{T}\right)^T = \widehat{\mat{R}}\mat{Z}\widehat{\mat{R}}^T.
\end{equation}
The rescaling of the right eigenvectors of ${\mat{D}}$ to satisfy the identity \eqref{MerriamIdentityinproof} is sufficient to guarantee the negativity of \eqref{RHSEntropy}. We see from \eqref{RHSEntropy} and \eqref{MerriamIdentityinproof}
\begin{equation}\label{signSatified2}
\begin{aligned}
-\frac{1}{2}\jump{\vec{v}}^T \widehat{\mat{R}}|{\boldsymbol\Lambda}|\widehat{\mat{R}}^{-1}\jump{\vec{q}} &\simeq - \frac{1}{2}\jump{\vec{v}}^T \widehat{\mat{R}}|{\boldsymbol\Lambda}|\widehat{\mat{R}}^{-1}\pderivative{\vec{q}}{\vec{v}}\jump{\vec{v}},  \\
&= - \frac{1}{2}\jump{\vec{v}}^T \widehat{\mat{R}}|{\boldsymbol\Lambda}|\widehat{\mat{R}}^{-1}\mat{H}\jump{\vec{v}}, \\
&=- \frac{1}{2}\jump{\vec{v}}^T \widehat{\mat{R}}|{\boldsymbol\Lambda}|\widehat{\mat{R}}^{-1}\left(\widehat{\mat{R}}\mat{Z}\widehat{\mat{R}}^T\right)\jump{\vec{v}}, \\
&=- \frac{1}{2}\jump{\vec{v}}^T \widehat{\mat{R}}|{\boldsymbol\Lambda}|\mat{Z}\widehat{\mat{R}}^T\jump{\vec{v}}.
\end{aligned}
\end{equation}
So with the appropriate diagonal scaling matrix $\mat{Z}$ we have shown that \eqref{signSatified2} is guaranteed negative because the product is a quadratic form scaled by a negative. We use the right eigenvectors from \eqref{rightEV}, the constraint \eqref{MerriamIdentityinproof}, and are able to determine the diagonal scaling matrix to be
\begin{equation}\label{scalingMatrixinProof}
\mat{Z} = \text{diag}\left(\frac{1}{{2\rho\gamma}},\,{\frac{p}{2\rho^3}},\,\frac{1}{{2\rho\gamma}},\,{\frac{\rho(\gamma-1)}{\gamma}},\,{\frac{p}{\rho}},\,\frac{1}{{2\rho\gamma}},\,{\frac{p}{2\rho^3}},\,\frac{1}{{2\rho\gamma}}\right),
\end{equation}
from which the the diagonal scaling matrix $\mat{T}$ \eqref{scalingMatrix} follows directly.
\end{proof}

\subsection{Discrete evaluation of the dissipation term}\label{sec:discEvaluation}


We now know the structure of the dissipation term at the interface between two cells. However, in the discrete setting, there is still the question of where to evaluate the dissipation operator. We know the specific averaging for the baseline entropy conserving flux \eqref{Eq:entropyconservative-yEKEP}. However, there is an open question of how to evaluate the dissipation term in \eqref{minimalDiss} discretely at some mean state. Much care is taken in the baseline flux $\fhat^\mathrm{KEPEC}$ by using very specific averages to guarantee discrete entropy conservation, e.g., \cite{Chandrashekar2015,Tadmor2003,Winters2016}. We will see that an equal amount of care must be taken for the dissipation term to guarantee that the numerical flux remains applicable to a wide variety of flow configurations \cite{Winters2017}.

Specifically, this section provides detailed derivations on the discrete forms of the matrices $\widehat{\mat{R}}$, $\mat{Z}$, and $\boldsymbol{\Lambda}$ needed to evaluate the entropy stable numerical flux function \eqref{minimalDiss}.

\subsubsection{Evaluation of the entropy Jacobian $\mat{H}$}

Somewhat counterintuitively, the first step to evaluate the dissipation term in the entropy stable flux \eqref{minimalDiss} is to determine a discrete evaluation of the entropy Jacobian matrix $\mat{H}$. This is due to the critical first step of \eqref{signSatified2} where we took
\begin{equation}\label{eq:simeq}
\jump{\vec{q}} \simeq \mat{H}\jump{\vec{v}}.
\end{equation}
To avoid unphysical dissipation we want to build the average state of $\mat{H}$ in such a way that equality holds in \eqref{eq:simeq} whenever possible \cite{Derigs2016_2}.

\begin{lem}[Failure of equality]\label{lem:notEqualLemma}~\newline
It is not possible to construct a consistent discrete matrix $\He{}$ such that
\begin{equation}\label{eq:equalityWeWant}
\jump{\vec{q}} = \He{}\jump{\vec{v}},
\end{equation}
holds for all eight components of the ideal MHD system.
\end{lem}
\begin{proof}
We seek a proof by contradiction. Thus, we assume that averages in the entries of $\mat{H}$ exist in such a way that the relation $\jump{\vec{q}} = \He{} \jump{\vec{v}}$ holds discretely. We derive the entries of the matrix $\He{}$ through a step-by-step process computing the solution of 64 linear equations:
\begin{equation}\label{eq:tobesolved}
\jump{\vec{q}} = \jump{\begin{bmatrix}\rho\\\rho u\\\rho v\\\rho w\\E\\B_1\\B_2\\B_3\end{bmatrix}} = \begin{bmatrix}
\He_{1,1} & \He_{1,2} & \dots & \dots & \He_{1,7} & \He_{1,8} \\
\He_{2,1} & \He_{2,2} & \dots & \dots & \He_{2,7} & \He_{2,8} \\
\vdots  & \vdots & \ddots & \ddots & \vdots & \vdots \\
\vdots  & \vdots & \ddots & \ddots & \vdots & \vdots \\
\He_{7,1} & \He_{7,2} & \dots & \dots & \He_{7,7} & \He_{7,8} \\
\He_{8,1} & \He_{8,2} & \dots & \dots & \He_{8,7} & \He_{8,8} \\
\end{bmatrix}
\jump{\begin{bmatrix}\frac{\gamma - s}{\gamma - 1}-\beta \lVert\vec{u}\rVert^2\\2\beta  u\\2\beta  v\\2\beta  w\\-2\beta \\2\beta B_1\\2\beta B_2\\2\beta B_3\end{bmatrix}}
 = \He{}\jump{\vec{v}}.
\end{equation}
The procedure is to multiply each row of $\He{}$ with the expanded jump in the entropy variables. By examining each equation individually, we are able to determine all unknown entries of the discrete matrix.

The derivation of the first row of $\He{}$ is straightforward and therefore serves as an excellent example for the derivation technique. First, we use properties of the linear jump operator \eqref{eq:jumpProperties} to expand the jump in the conservative
\begin{equation}\label{eq:jumpConsVars}
\jump{\vec{q}} = \resizebox{.885\hsize}{!}{$ %
	\jump{\begin{bmatrix}\rho \\ \rho u \\ \rho v \\ \rho w \\ E \\ B_1 \\ B_2 \\ B_3 \\\end{bmatrix}}
	=
	\begin{bmatrix}
	\jump{\rho} \\ \avg{\rho} \jump{u} + \avg{u} \jump{\rho} \\ \avg{\rho} \jump{v} + \avg{v} \jump{\rho}\\ \avg{\rho} \jump{w} + \avg{w} \jump{\rho} \\ \left(\frac{\avg{\beta^{-1}}}{2(\gamma-1)} + \frac{1}{2}\overline{\avg{\vec{u}^2}}\right)\jump{\rho}+\avg{\rho}\left(\avg{u}\jump{u} + \avg{v}\jump{v} + \avg{w}\jump{w}\right) - \frac{\avg{\rho}}{2\betaavg(\gamma-1)}\jump{\beta} + \sum\limits_{i=1}^{3}\avg{B_i}\jump{B_i} \\ \jump{B_1} \\ \jump{B_2} \\ \jump{B_3} \\
	\end{bmatrix},
$}
\end{equation}
as well as the entropy variables
\begin{equation}\label{eq:jumpEntVars}
\jump{\vec{v}} = \resizebox{.885\hsize}{!}{$ %
	\jump{\begin{bmatrix}\frac{\gamma - s}{\gamma - 1}-\beta \lVert\vec{u}\rVert^2\\2\beta  u\\2\beta  v\\2\beta  w\\-2\beta \\2\beta B_1\\2\beta B_2\\2\beta B_3\end{bmatrix}}
	=
	\begin{bmatrix}
	\frac{\jump{\rho}}{\rholn}+\frac{\jump{\beta}}{\betaln(\gamma-1)}-\Big(\avg{u^2}+\avg{v^2}+\avg{w^2}\Big)\jump{\beta}-2\avg{\beta}\Big( \avg{u}\jump{u} + \avg{v}\jump{v} + \avg{w}\jump{w} \Big) \\
	2 \avg{\beta}\jump{u} + 2 \avg{u}\jump{\beta} \\
	2 \avg{\beta}\jump{v} + 2 \avg{v}\jump{\beta} \\
	2 \avg{\beta}\jump{w} + 2 \avg{w}\jump{\beta} \\
	-2 \jump{\beta} \\
	2 \avg{\beta}\jump{B_1} + 2 \avg{B_1}\jump{\beta} \\
	2 \avg{\beta}\jump{B_2} + 2 \avg{B_2}\jump{\beta} \\
	2 \avg{\beta}\jump{B_3} + 2 \avg{B_3}\jump{\beta} \\
	\end{bmatrix},
	$}
\end{equation}
with
\begin{equation}
	{\betaavg = 2 \avg{\beta}^2 - \avg{\beta^2},\ \text{and}\ \overline{\avg{\vec{u}^2}} = \avg{u^2} + \avg{v^2} + \avg{w^2}.}
\end{equation}
According to \eqref{eq:tobesolved}, the entries of the first row of $\He{}$ can be obtained by solving
{\small
\begin{gather}
\jump{\rho} = \He_{1,1}\left(\frac{\jump{\rho}}{\rholn} + \frac{\jump{\beta}}{\betaln (\gamma-1)} - \left(\avg{u^2} + \avg{v^2} + \avg{w^2}\right) \jump{\beta} - 2\avg{\beta}\Big(\avg{u}\jump{u} + \avg{v}\jump{v} + \avg{w}\jump{w}\Big)\right) \notag\\
+\, \He_{1,2}\left( 2 \avg{\beta}\jump{u} + 2 \avg{u}\jump{\beta} \right) + \He_{1,3}\left( 2 \avg{\beta}\jump{v} + 2 \avg{v}\jump{\beta} \right) + \He_{1,4}\left( 2 \avg{\beta}\jump{w} + 2 \avg{w}\jump{\beta} \right) + \He_{1,5}\left( - 2 \jump{\beta}\right) \\
+\, \He_{1,6}\left( 2 \avg{\beta}\jump{B_1} + 2 \avg{B_1}\jump{\beta} \right) +  \He_{1,7}\left( 2 \avg{\beta}\jump{B_2} + 2 \avg{B_2}\jump{\beta} \right) +  \He_{1,8}\left( 2 \avg{\beta}\jump{B_3} + 2 \avg{B_3}\jump{\beta} \right). \notag\label{eq:firstrow}
\end{gather}}
From this equation, we directly obtain the entries of the first row of the discretized entropy Jacobian,
{\small
\begin{equation}\label{eq:firstrowH}
\He_{1} = \begin{bmatrix}\rholn & \rholn\avg{u} & \rholn\avg{v} & \rholn\avg{w} & \Eline & 0 & 0 & 0 \end{bmatrix},
\end{equation}}
where we introduced additional notation for compactness
{\small
\begin{equation}
	\pln := \frac{\rholn}{2 \betaln},\quad \Eline := \frac{\pln}{\gamma-1} + \frac{1}{2} \rholn \uavg, \quad \mbox{and} \quad \uavg := 2\left(\avg{u}^2 + \avg{v}^2 + \avg{w}^2\right)-\left(\avg{u^2} + \avg{v^2} + \avg{w^2}\right).
\end{equation}
}
The same procedure is repeated over each row to identify the remaining entries.

Unfortunately, we find that such a forthright solution of \eqref{eq:tobesolved} leads to an asymmetric discrete matrix $\He{}$. This asymmetry indicates that the scheme will not be entropy stable nor will it satisfy the necessary eigenvector scaling realtionship \eqref{MerriamIdentity}. Therefore, it is not possible to derive a discrete symmetric matrix such that the equality $\jump{\vec{q}} = \He{} \jump{\vec{v}}$ holds exactly for all components of $\vec{q}$.
\end{proof}

\begin{thm}[Discrete entropy Jacobian $\H{}$]~\newline
If we take the discrete dissipation matrix to be of the form
\begin{equation}\label{eq:H}
\H{} = \resizebox{0.95\hsize}{!}{$ %
\begin{bmatrix}
\rholn & \rholn\avg{u} & \rholn\avg{v} & \rholn\avg{w} & \Eline & 0 & 0 & 0 \\
\rholn\avg{u} & \rholn\avg{u}^2 + \pavg & \rholn\avg{u}\avg{v} & \rholn\avg{u}\avg{w} & \left(\Eline + \pavg \right) \avg{u} & 0 & 0 & 0 \\
\rholn\avg{v} & \rholn\avg{v}\avg{u} & \rholn\avg{v}^2 + \pavg & \rholn\avg{v}\avg{w} & \left(\Eline + \pavg \right) \avg{v} & 0 & 0 & 0 \\
\rholn\avg{w} & \rholn\avg{w}\avg{u} & \rholn\avg{w}\avg{v} & \rholn\avg{w}^2 + \pavg & \left(\Eline + \pavg \right) \avg{w} & 0 & 0 & 0 \\
\Eline & \left(\Eline + \pavg \right) \avg{u} & \left(\Eline + \pavg \right) \avg{v} & \left(\Eline + \pavg \right) \avg{w} & \H_{5,5} & \tau \avg{B_1} & \tau \avg{B_2} & \tau \avg{B_3} \\
0 & 0 & 0 & 0 & \tau \avg{B_1} & \tau & 0 & 0 \\
0 & 0 & 0 & 0 & \tau \avg{B_2} & 0 & \tau & 0 \\
0 & 0 & 0 & 0 & \tau \avg{B_3} & 0 & 0 & \tau \\
\end{bmatrix},
$}
\end{equation}
with
\begin{equation*}\label{eq:alotofequations}
\resizebox{0.98\hsize}{!}{$
\H_{5,5} = \frac{1}{\rholn}\Big(\frac{(\pln)^2}{\gamma-1} + {\Eline^2}\Big) + \pavg \Big(\!\avg{u}^2 + \avg{v}^2 + \avg{w}^2\!\Big) + \tau \sum_{i=1}^{3} \Big(\!\avg{B_i}^2\!\Big),\quad
\pavg := \frac{\avg{\rho}}{2\avg{\beta}},\quad \mbox{and}\quad \tau := \frac{\avg{p}}{\avg{\rho}},$}
\end{equation*}
then we obtain exact equality in \eqref{eq:equalityWeWant} for each component except the total energy, i.e.,
\begin{equation}\label{eq:almostEqual}
(\jump{\vec{q}})_i = (\H{}\jump{\vec{v}})_i \;\text{  for  }\; i = \{1,2,3,4,6,7,8\}\; \text{  and  }\; (\jump{\vec{q}})_5 \simeq (\H{}\jump{\vec{v}})_5.
\end{equation}
\end{thm}
\begin{proof}
Due to result in Lemma \ref{lem:notEqualLemma} we know that equality cannot hold for each component in \eqref{eq:equalityWeWant}. However, if special care is taken during the expansion of the total energy term, a matrix $\H{}$ that obeys the required property can be found. It guarantees equality in all but the jump in the total energy term where the equality reduces to an asymptotic one. The modified jump in total energy reads
\begin{equation}
	\overline{\jump{E}} = \resizebox{.8\hsize}{!}{$ %
		\left(\frac{1}{2(\gamma-1)\betaln} + \frac{1}{2}\uavg\right)\jump{\rho} - \frac{\rholn}{2(\gamma-1)}\frac{\jump{\beta}}{(\betaln)^2} + \avg{\rho}\left(\avg{u}\jump{u} + \avg{v}\jump{v} + \avg{w}\jump{w}\right) + \sum\limits_{i=1}^{3}\big( \avg{B_i}\jump{B_i}\big)
	$} \simeq \jump{E}.
\end{equation}
We replace $\jump{E}$ with $\overline{\jump{E}}$ in \eqref{eq:jumpConsVars}. Now, we apply the previously described tactic from the proof of Lemma \ref{lem:notEqualLemma} to solve \eqref{eq:tobesolved}. Forgoing the computational details we arrive at the discrete matrix
\begin{equation}\label{eq:HinProof}
\H{} = \resizebox{0.95\hsize}{!}{$ %
\begin{bmatrix}
\rholn & \rholn\avg{u} & \rholn\avg{v} & \rholn\avg{w} & \Eline & 0 & 0 & 0 \\
\rholn\avg{u} & \rholn\avg{u}^2 + \pavg & \rholn\avg{u}\avg{v} & \rholn\avg{u}\avg{w} & \left(\Eline + \pavg \right) \avg{u} & 0 & 0 & 0 \\
\rholn\avg{v} & \rholn\avg{v}\avg{u} & \rholn\avg{v}^2 + \pavg & \rholn\avg{v}\avg{w} & \left(\Eline + \pavg \right) \avg{v} & 0 & 0 & 0 \\
\rholn\avg{w} & \rholn\avg{w}\avg{u} & \rholn\avg{w}\avg{v} & \rholn\avg{w}^2 + \pavg & \left(\Eline + \pavg \right) \avg{w} & 0 & 0 & 0 \\
\Eline & \left(\Eline + \pavg \right) \avg{u} & \left(\Eline + \pavg \right) \avg{v} & \left(\Eline + \pavg \right) \avg{w} & \H_{5,5} & \tau \avg{B_1} & \tau \avg{B_2} & \tau \avg{B_3} \\
0 & 0 & 0 & 0 & \tau \avg{B_1} & \tau & 0 & 0 \\
0 & 0 & 0 & 0 & \tau \avg{B_2} & 0 & \tau & 0 \\
0 & 0 & 0 & 0 & \tau \avg{B_3} & 0 & 0 & \tau \\
\end{bmatrix},
$}
\end{equation}
with
\begin{equation*}
\resizebox{0.98\hsize}{!}{$
\H_{5,5} = \frac{1}{\rholn}\Big(\frac{(\pln)^2}{\gamma-1} + {\Eline^2}\Big) + \pavg \Big(\!\avg{u}^2 + \avg{v}^2 + \avg{w}^2\!\Big) + \tau \sum_{i=1}^{3} \Big(\!\avg{B_i}^2\!\Big),\quad
\pavg := \frac{\avg{\rho}}{2\avg{\beta}},\quad \mbox{and}\quad \tau := \frac{\avg{p}}{\avg{\rho}}.$}
\end{equation*}
Clearly, the discrete entropy Jacobian matrices \eqref{eq:H} is symmetric. Furthermore, it has been shown using Sylvester's criterion that the discrete matrix $\H{}$ is  symmetric positive definite (s.p.d.) \cite{Derigs2016_2}.
\end{proof}

\subsubsection{Discrete right eigenvector, $\widehat{\mat{R}}$, and scaling, $\mat{Z}$, matrices}
With the knowledge of the discrete symmetric matrix $\H{}$ the next goal is to determine the discrete evaluation of the matrix of right eigenvectors $ \widehat{\mat{R}}$ and the diagonal scaling matrix $\mat{Z}$ such that the identity
\begin{equation}
\mat{H} = \widehat{\mat{R}}\mat{Z}\widehat{\mat{R}}^T,
\end{equation}
will hold discretely.

\begin{thm}[Discrete right eigenvector $\R{}$ and scaling $\T{}$ matrices]~\newline
If we take the discrete right eigenvector matrix \eqref{eq:rightEigen} and discrete scaling matrix \eqref{eq:scalingMatDiscrete}, then we have the discrete entropy scaled eigenvector relation
\begin{equation}\label{eq:importantCondition}
\H{} = \R{}\T{}\R{}^T.
\end{equation}
Due to their complexity we present the specific form of $\R{}$ and $\T{}$ in three parts. First, we give the specific averages of the convenience variables \eqref{eq:alotofequations1}
\begin{equation}\label{eq:alotofequations_discrete}
\begin{aligned}
\hat{\Psi}_{\pm{s}} &:= \frac{\hat{\alpha}_{s} \rho^{\ln} \overline{\lVert\vec{u}\rVert^2}}{2} - a^{\beta} \hat{\alpha}_{f} \rho^{\ln} \bar{b}_\perp + \frac{\hat{\alpha}_{s} \rho^{\ln} (a^{\ln})^2}{\gamma-1} \pm \hat{\alpha}_{s} \hat{c}_{s} \rho^{\ln} \avg{u} \pm \hat{\alpha}_{f} \hat{c}_{f} \rho^{\ln} \sigma(\bar{b}_1) (\avg{v} \bar{\beta}_2 + \avg{w} \bar{\beta}_3), \\
\hat{\Psi}_{\pm{f}} &:= \frac{\hat{\alpha}_{f} \rho^{\ln} \overline{\lVert\vec{u}\rVert^2}}{2} + a^{\beta} \hat{\alpha}_{s} \rho^{\ln} \bar{b}_\perp + \frac{\hat{\alpha}_{f} \rho^{\ln} (a^{\ln})^2}{\gamma-1} \pm \hat{\alpha}_{f} \hat{c}_{f} \rho^{\ln} \avg{u} \mp \hat{\alpha}_{s} \hat{c}_{s} \rho^{\ln} \sigma(\bar{b}_1) (\avg{v} \bar{\beta}_2 + \avg{w} \bar{\beta}_3),  \\
\hat{c}_{a}^2& := \bar{b}_1^2 = \frac{\avg{B_1}^2}{\rho^{\ln}}, \quad \hat{c}_{f,s}^2 := \frac{1}{2}\left((\bar{a}^2+\bar{b}^2) \pm \sqrt{(\bar{a}^2+\bar{b}^2)^2 - 4\bar{a}^2 \bar{b}_1^2}\right), \\
\avg{p} &= \frac{\avg{\rho}}{2\avg{\beta}},\quad \bar{a}^2 := \gamma \frac{\avg{p}}{\rho^{\ln}} \quad (a^{\ln})^2 := \gamma\frac{p^{\ln}}{\rho^{\ln}},\quad (a^{\beta})^2 := \gamma\frac{1}{2\avg{\beta}},\\
\bar{b}^2 &= \bar{b}_1^2 + \bar{b}_2^2 + \bar{b}_3^2, \quad \bar{b}_\perp^2 = \bar{b}_2^2 + \bar{b}_3^2, \quad \bar{\beta}_{1,2,3} = \frac{\bar{b}_{1,2,3}}{\bar{b}_\perp}, \quad \bar{b}^2_{1,2,3} = \frac{\avg{B_{1,2,3}}}{\sqrt{\rho^{\ln}}},\\
\hat{\alpha}_{f}^2 &= \frac{\bar{a}^2 - \hat{c}_{s}^2}{\hat{c}_{f}^2 - \hat{c}_{s}^2}, \quad \hat{\alpha}_{s}^2 = \frac{\hat{c}_{f}^2 -\bar{a}^2}{\hat{c}_{f}^2 - \hat{c}_{s}^2},\quad
\sigma(\omega) = \begin{cases}
+1 &\mbox{if } \omega \ge 0, \\
-1 &\text{otherwise}
\end{cases}.
\end{aligned}
\end{equation}
Next, we give the discrete right eigenvector matrix
\begin{equation}\label{eq:rightEigen}
\R{} = \left[\,\hat{\vec{r}}_{ +{f}} \,|\, \hat{\vec{r}}_{ +{a}} \,|\, \hat{\vec{r}}_{ +{s}} \,|\, \hat{\vec{r}}_{ E} \,|\, \hat{\vec{r}}_{ D} \,|\, \hat{\vec{r}}_{ -{s}} \,|\, \hat{\vec{r}}_{ -{a}} \,|\, \hat{\vec{r}}_{ -{f}} \, \right],
\end{equation}
built with the dicrete eigenvectors $\hat{\vec{r}}$
\begin{itemize}
\item[] \underline{Entropy and Divergence Waves}:
\begin{equation}
\hat{\vec{r}}_{ E}  = \begin{bmatrix} 1 \\ \avg{u} \\ \avg{v} \\ \avg{w} \\[0.1cm] \frac{1}{2} \overline{\|\vec{u}\|^2}\\[0.05cm]0 \\0 \\0 \end{bmatrix},\quad\hat{\vec{r}}_{ D}  = \begin{bmatrix} 0 \\ 0 \\0 \\0 \\ \avg{B_1} \\[0.05cm]1 \\0 \\0 \end{bmatrix},
\end{equation}
\item[] \underline{Alfv\'{e}n Waves}:
\begin{equation}
\hat{\vec{r}}_{ \pm a} = \begin{bmatrix}
0 \\
0 \\
\pm \rho^{\ln}\sqrt{\avg{\rho}}\,\bar{\beta}_3 \\[0.1cm]
\mp \rho^{\ln}\sqrt{\avg{\rho}}\,\bar{\beta}_2 \\[0.1cm]
\mp \rho^{\ln}\sqrt{\avg{\rho}}(\bar{\beta}_2 \avg{w} - \bar{\beta}_3 \avg{v})\\[0.1cm]
0 \\
-\rho^{\ln} \bar{\beta}_3 \\[0.1cm]
\rho^{\ln} \bar{\beta}_2
\end{bmatrix},
\end{equation}
\item[] \underline{Magnetoacoustic Waves}:
\begin{equation}
\hat{\vec{r}}_{ \pm f} = \begin{bmatrix}
\hat{\alpha}_{ f}\rho^{\ln} \\[0.1cm]
\hat{\alpha}_{ f}\rho^{\ln}(\avg{u} \pm \hat{c}_{f}) \\[0.1cm]
\rho^{\ln}\left(\hat{\alpha}_{ f} \avg{v} \mp \hat{\alpha}_{ s} \hat{c}_{ s} \bar{\beta}_2 \sigma(\bar{b}_1) \right) \\[0.1cm]
\rho^{\ln}\left(\hat{\alpha}_{ f} \avg{w} \mp \hat{\alpha}_{ s} \hat{c}_{ s} \bar{\beta}_3 \sigma(\bar{b}_1) \right) \\[0.1cm]
\hat{\Psi}_{\pm f} \\[0.1cm]
0 \\[0.1cm]
\hat{\alpha}_{ s} a^{\beta} \bar{\beta}_2 \sqrt{\rho^{\ln}} \\[0.1cm]
\hat{\alpha}_{ s} a^{\beta} \bar{\beta}_3 \sqrt{\rho^{\ln}}
\end{bmatrix},
\qquad
\hat{\vec{r}}_{ \pm s} = \begin{bmatrix}
\hat{\alpha}_{ s}\rho^{\ln} \\[0.1cm]
\hat{\alpha}_{ s}\rho^{\ln}\left(\avg{u} \pm \hat{c}_{ s}\right) \\[0.1cm]
\rho^{\ln}\left(\hat{\alpha}_{ s} \avg{v} \pm \hat{\alpha}_{ f} \hat{c}_{ f} \bar{\beta}_2 \sigma(\bar{b}_1)\right) \\[0.1cm]
\rho^{\ln}\left(\hat{\alpha}_{ s} \avg{w} \pm \hat{\alpha}_{ f} \hat{c}_{ f} \bar{\beta}_3 \sigma(\bar{b}_1)\right) \\[0.1cm]
\hat{\Psi}_{\pm s} \\[0.1cm]
0 \\[0.1cm]
-\hat{\alpha}_{ f} a^{\beta} \bar{\beta}_2 \sqrt{\rho^{\ln}} \\[0.1cm]
-\hat{\alpha}_{ f} a^{\beta} \bar{\beta}_3 \sqrt{\rho^{\ln}}
\end{bmatrix},
\end{equation}
\end{itemize}
and finally the discrete diagonal scaling matrix
\begin{equation}\label{eq:scalingMatDiscrete}
\T{} = \text{diag}\left(\frac{1}{2\gamma\rho^{\ln}},\frac{1}{4\avg{\beta}(\rho^{\ln})^2},\frac{1}{2\gamma\rho^{\ln}},\frac{\rho^{\ln}(\gamma-1)}{\gamma},\frac{1}{2\avg{\beta}},\frac{1}{2\gamma\rho^{\ln}},\frac{1}{4\avg{\beta}(\rho^{\ln})^2},\frac{1}{2\gamma\rho^{\ln}}\right).
\end{equation}
\end{thm}
\begin{proof}
We have already computed the discrete entropy Jacobian $\H{}$ given in \eqref{eq:H}. The strategy to obtain the discrete evaluation of the right eigenvector and scaling matrices is similar to that used to create $\H{}$. That is, we seek average states in the right eigenvectors \eqref{rightEV} and scaling matrix \eqref{scalingMatrixinProof} such that we have the discrete entropy scaled eigenvector relationship
\begin{equation}
\H{} = \R{}\T{}\R{}^T.
\end{equation}
Once this relationship is satisfied discretely it will define unique averaging procedures for three of the terms in the dissipation operator \eqref{minimalDiss} while retaining the almost equal property \eqref{eq:almostEqual}.

It is straightforward to relate the entries of the matrix $\H{}$ and determine the 64 individual components of the matrices $\R{}$ and $\T{}$. We will explicitly demonstrate two computations to outline the general technique and justify the somewhat unconventional averaging strategies employed in the final form. After this brief outline of the derivation, we will arrive at the complete, discrete evaluation for the right eigenvector and scaling matrices.

We begin by computing the first entry of the first row of the system, which must satisfy
\begin{equation}\label{eq:firstEquation}
\H{}_{1,1} = \rholn \shouldbe \frac{1}{2\hat{\rho}\gamma}\left(2(\hat{\alpha}_{ f}^2 + \hat{\alpha}_{ s}^2)\hat{\rho}^2\right) + \frac{\hat{\rho}(\gamma-1)}{\gamma} = (\R{}\T{}\R{}^T)_{1,1}.
\end{equation}
The $\alpha_{f,s}$ variables satisfy many useful identities \cite{roe1996} that we must recover discretely. Namely,
\begin{equation}\label{eq:alphas}
\hat{\alpha}_{ f}^2 + \hat{\alpha}_{ s}^2 = 1,\qquad \hat{\alpha}_{ f}^2\hat{c}_{ f}^2 + \hat{\alpha}_{ s}^2\hat{c}_{ s}^2 = \hat{a}^2.
\end{equation}
Thus, there is still some freedom in the underlying averaging of the $\hat{\alpha}_{f,s}$ and the wave speeds $\hat{c}_{f,s}$ as long as the identities \eqref{eq:alphas} hold. We use the first identity in \eqref{eq:alphas} and choose $\hat{\rho} = \rholn$ in \eqref{eq:firstEquation} to guarantee that $\H{}_{1,1}= (\R{}\T{}\R{}^T)_{1,1}$. Therefore, we have determined the first row of the discrete eigenvector matrix, denoted by $\R{}_{1,:}$, as well as five entries of the scaling matrix $\T{}$ to be
\begin{equation}\label{eq:firstRow}
\begin{aligned}
&\R_{1,:} = \left[\hat{\alpha}_{ f}\rholn \;\;\; 0 \;\;\; \hat{\alpha}_{ s}\rholn \;\;\; 1 \;\;\; 0 \;\;\; \hat{\alpha}_{ s}\rholn \;\;\; 0 \;\;\; \hat{\alpha}_{ f}\rholn\right],\\[0.1cm]
&\T_{1,1} = \T_{3,3} = \T_{6,6} = \T_{8,8} = \frac{1}{2\rholn\gamma},\quad \T_{4,4} = \frac{\rholn(\gamma-1)}{\gamma}.
\end{aligned}
\end{equation}

The second example is the second entry of the second row of the system given by
\begin{equation}\label{eq:secondEquation}
\H{}_{2,2} = \rholn\avg{u}^2 + \avg{p} \shouldbe \frac{\hat{\rho}}{\gamma}\left(\hat{u}^2 + \hat{\alpha}_{ f}^2\hat{c}_{ f}^2 + \hat{\alpha}_{ s}^2\hat{c}_{ s}^2\right)+\frac{\rholn\hat{u}^2(\gamma-1)}{\gamma}=(\R{}\T{}\R{}^T)_{2,2}.
\end{equation}
We now select the particular average for the sound speed in the second identity of \eqref{eq:alphas} to be
\begin{equation}\label{eq:avgSoundSpeed}
\hat{\alpha}_{ f}^2\hat{c}_{ f}^2 + \hat{\alpha}_{ s}^2\hat{c}_{ s}^2 = \bar{a}^2 = \gamma\frac{\avg{p}}{\rholn},
\end{equation}
and also take $\hat{\rho}=\rholn$ and $\hat{u}=\avg{u}$ to guarantee equality in \eqref{eq:secondEquation}. So, we have determined some important terms necessary for the second row of the eigenvector matrix $\R_{2,:}$
\begin{equation}\label{eq:secondRow}
\begin{aligned}
&\R_{2,:} = \left[\hat{\alpha}_{ f}\rholn(\avg{u}+\hat{c}_{ f}) \;\;\; 0 \;\;\; \hat{\alpha}_{ s}\rholn(\avg{u}+\hat{c}_{ s}) \;\;\; \avg{u} \;\;\; 0 \;\;\; \hat{\alpha}_{ s}\rholn(\avg{u}-\hat{c}_{ s}) \;\;\; 0 \;\;\; \hat{\alpha}_{ f}\rholn(\avg{u}-\hat{c}_{ f})\right],\\[0.1cm]
&\quad\quad\hat{\alpha}_{f}^2 = \frac{\bar{a}^2 - \hat{c}_{s}^2}{\hat{c}_{f}^2 - \hat{c}_{s}^2}, \quad \hat{\alpha}_{s}^2 = \frac{\hat{c}_{f}^2 -\bar{a}^2}{\hat{c}_{f}^2 - \hat{c}_{s}^2},
\end{aligned}
\end{equation}
where the averages of the two wave speeds $\hat{c}_{f,s}$ are still arbitrary. We apply this same process to the remaining unknown values from the condition \eqref{eq:importantCondition} and, after a considerable amount of algebraic manipulation, determine the unique averaging procedure for the discrete eigenvector and scaling matrices, given in \eqref{eq:alotofequations_discrete}--\eqref{eq:scalingMatDiscrete}. We note that the derivations as well as the constraint \eqref{eq:importantCondition} have been verified using the symbolic algebra program Maxima \cite{maxima}.
\end{proof}

\begin{rem}
The newly derived discrete dissipation term remains valid for the Euler computations when all magnetic field components are zero. Details of the average right eigenvector and diagonal scaling matrices for the compressible Euler equations can be found in \cite[App. A]{Winters2017}.
\end{rem}

\subsubsection{Evaluation of the diagonal eigenvalue matrix $\boldsymbol{\Lambda}$}
The final term which requires an explicit discrete form for the KEPES flux \eqref{minimalDiss} is the diagonal matrix of eigenvalues $\boldsymbol{\Lambda}$.
\begin{thm}[Discrete diagonal eigenvalue matrix $\widehat{\boldsymbol{\Lambda}}$]\label{thm:eig}~\newline
If we consider the augmented ideal MHD system in terms of primitive variables \eqref{MHDPowellPrim}
\begin{equation}
\pderivative{\vec{\omega}}{t} + \mat{B}\pderivative{\vec{\omega}}{x} = 0,
\end{equation}
then we determine the discrete evaluation of the eigenvalues to be
\begin{equation}\label{eq:easyidealMHDeigenvaluesDiscrete}
\hat{\vec{\lambda}} :=
	\begin{bmatrix}
	\hat{\lambda}_{-f}\\[0.1cm]
	\hat{\lambda}_{-s}\\[0.1cm]
	\hat{\lambda}_{+s}\\[0.1cm]
	\hat{\lambda}_{+f}\\[0.1cm]
	\hat{\lambda}_{-a}\\[0.1cm]
	\hat{\lambda}_{+a}\\[0.1cm]
	\hat{\lambda}_D\\[0.1cm]
	\hat{\lambda}_E\\[0.1cm]
	\end{bmatrix} =
	\begin{bmatrix}
	\avg{u}-\doublehat{c}_f\vphantom{\hat{\lambda}_{-f}}\\[0.1cm]
	\avg{u}-\doublehat{c}_s\vphantom{\hat{\lambda}_{-f}}\\[0.1cm]
	\avg{u}+\doublehat{c}_s\vphantom{\hat{\lambda}_{-f}}\\[0.1cm]
	\avg{u}+\doublehat{c}_f\vphantom{\hat{\lambda}_{-f}}\\[0.1cm]
	\avg{u}-\doublehat{c}_a\vphantom{\hat{\lambda}_{-f}}\\[0.1cm]
	\avg{u}+\doublehat{c}_a\vphantom{\hat{\lambda}_{-f}}\\[0.1cm]
	\avg{u}\vphantom{\hat{\lambda}_{-f}}\\[0.1cm]
	\avg{u}\vphantom{\hat{\lambda}_{-f}}\\[0.1cm]
	\end{bmatrix}\quad
	\begin{matrix}
	\mbox{left going fast magnetoacoustic wave,}\vphantom{\hat{\lambda}_{-f}}\\[0.1cm]
	\mbox{left going slow magnetoacoustic wave,}\vphantom{\hat{\lambda}_{-f}}\\[0.1cm]
	\mbox{right going slow magnetoacoustic wave,}\vphantom{\hat{\lambda}_{-f}}\\[0.1cm]
	\mbox{right going fast magnetoacoustic wave,}\vphantom{\hat{\lambda}_{-f}}\\[0.1cm]
	\mbox{left going Alfv\'en wave,}\vphantom{\hat{\lambda}_{-f}}\\[0.1cm]
	\mbox{right going Alfv\'en wave,}\vphantom{\hat{\lambda}_{-f}}\\[0.1cm]
	\mbox{divergence wave,}\vphantom{\hat{\lambda}_{-f}}\\[0.1cm]
	\mbox{entropy wave,}\vphantom{\hat{\lambda}_{-f}}\\[0.1cm]
	\end{matrix}
\end{equation}
where we introduce the discrete wave speeds
\begin{equation}\label{eq:app:characteristicSpeedsDiscrete}
\doublehat{c}_a := \sqrt{\hat{b}_1^2}, \qquad
\doublehat{c}_{f,s} := \frac{1}{2}\left( \sqrt{\hat{a}^2 + \hat{b}^2 + 2\sqrt{\hat{a}^2 \hat{b}_1^2}} \pm \sqrt{\hat{a}^2 + \hat{b}^2 - 2\sqrt{\hat{a}^2 \hat{b}_1^2}}\right),
\end{equation}
with the special discrete averages
\begin{equation}
\hat{\vec{b}}^2 := \avg{\vec{B}} \cdot \avg{\frac{\vec{B}}{\rho}},\quad
\hat{a}^2 := \gamma \avg{p}\avg{\rho^{-1}},\quad
\hat{b}^2:=\hat{b}_1^2+\hat{b}_2^2+\hat{b}_3^2.
\end{equation}
Therefore, the discrete evaluation of the diagonal matrix of eigenvalues is
\begin{equation}\label{eq:discreteEigenvaluesMat}
\boldsymbol{\widehat{\Lambda}} = \text{diag}(\hat{\lambda}_{+f},\hat{\lambda}_{+a},\hat{\lambda}_{+s},\hat{\lambda}_{E},\hat{\lambda}_{D},\hat{\lambda}_{-s},\hat{\lambda}_{-a},\hat{\lambda}_{-f}).
\end{equation}
\end{thm}
\begin{proof}
We begin with the augmented ideal MHD equations written in primitive variables
\begin{equation}
\pderivative{\vec{\omega}}{t} + \mat{B}\pderivative{\vec{\omega}}{x} = 0,
\end{equation}
Note that the flux Jacobian $\mat{D}$ in conservative variables and the flux Jacobian in primitive variables $\mat{B}$ are similar. That is, they have the same eigenvalues, but different eigenvectors. Therefore, we can use the much simpler matrix $\mat{B}$ \eqref{Bmatrix} to determine the discrete evaluation of the eigenvalues.

We make the discrete ansatz, where we discretize the update of the primitive variables in the spatial dimension \eqref{MHDPowellPrim} to be
\begin{equation}\label{eq:discreteEVs}
\begin{aligned}
	-\pderivative{\vec{\omega}}{t} &= \mat{B}\pderivative{\vec{\omega}}{x},  \\
	&= \int_{L}^{R} \left( \mat{B}\pderivative{\vec{\omega}}{x} \,\right) \text{d}x, \\
	&\approx \frac{\Delta x}{2} \sum_{k=\{L,R\}}\left(\mat{B} \pderivative{\vec{\omega}}{x}\right)_k, \\
	&= \frac{\Delta x}{2} \left(\mat{B}_L \frac{\vec{\omega}_R - \vec{\omega}_L}{\Delta x} + \mat{B}_R \frac{\vec{\omega}_R - \vec{\omega}_L}{\Delta x}\right),\\
	&= \frac{1}{2} (\mat{B}_L + \mat{B}_R) (\vec{\omega}_R - \vec{\omega}_L),\\
	&= \avg{\mat{B}} \jump{\vec{\omega}},
\end{aligned}
\end{equation}
where we used the trapezoidal rule to approximate the integral on the right hand side of \eqref{eq:discreteEVs} using the left and right states. We immediately see that the discretized version of the coefficient matrix is the continuous coefficient matrix, $\mat{B}$, arithmetically averaged in each entry, $\mathcal{B} := \avg{\mat{B}} = \frac{1}{2} \left(\mat{B}_L + \mat{B}_R\right)$ is given by
\begin{equation}
\mathcal{B} =
\begin{bmatrix}
\avg{u} & \avg{\rho} & 0 & 0 & 0 & 0 & 0 & 0 \\[0.2cm]
0 & \avg{u} & 0 & 0 & \avg{\rho^{-1}} & -\avg{\!\frac{B_1}{\rho}\!} & \avg{\!\frac{B_2}{\rho}\!} & \avg{\!\frac{B_3}{\rho}\!} \\[0.2cm]
0 & 0 & \avg{u} & 0 & 0 & -\avg{\!\frac{B_2}{\rho}\!} & -\avg{\!\frac{B_1}{\rho}\!} & 0 \\[0.2cm]
0 & 0 & 0 & \avg{u} & 0 & -\avg{\!\frac{B_3}{\rho}\!} & 0 & -\avg{\!\frac{B_1}{\rho}\!} \\[0.2cm]
0 &\gamma \avg{p} & 0 & 0 & \avg{u} & 0 & 0 & 0 \\[0.2cm]
0 & 0 & 0 & 0 & 0 & \avg{u} & 0 & 0 \\[0.2cm]
0 & \avg{B_2} & -\avg{B_1} & 0 & 0 & 0 & \avg{u} & 0 \\[0.2cm]
0 & \avg{B_3} & 0 & -\avg{B_1} & 0 & 0 & 0 & \avg{u} \\
\end{bmatrix}.
\end{equation}

We compute the discrete eigenvalues of $\mathcal{B}$ to be
\begin{equation}\label{eq:supercomplicateddiscreteidealGLMMHDeigenvaluespressure}
\hat{\vec{\lambda}} = \resizebox{0.92\textwidth}{!}{$
	\begin{bmatrix}
	\avg{u}-\frac{\sqrt{2\,\sqrt{\avg{B_1}\,
				\avg{\frac{B_1}{\rho}}\,\avg{p}\,\avg{\rho^{-1}}\,\gamma}+\avg{p}
			\,\avg{\rho^{-1}}\,\gamma+\avg{B_3}\,\avg{\frac{B_3}{\rho}}+
			\avg{B_2}\,\avg{\frac{B_2}{\rho}}+\avg{B_1}\,\avg{\frac{B_1}{\rho}}}
		+\sqrt{-2\,\sqrt{\avg{B_1}\,\avg{\frac{B_1}{\rho}}\,\avg{p}\,
				\avg{\rho^{-1}}\,\gamma}+\avg{p}\,\avg{\rho^{-1}}\,\gamma+
			\avg{B_3}\,\avg{\frac{B_3}{\rho}}+\avg{B_2}\,\avg{\frac{B_2}{\rho}}+
			\avg{B_1}\,\avg{\frac{B_1}{\rho}}}}{2}\\[0.2cm]
\avg{u}-\frac{\sqrt{2\,\sqrt{\avg{B_1}\,
			\avg{\frac{B_1}{\rho}}\,\avg{p}\,\avg{\rho^{-1}}\,\gamma}+\avg{p}
		\,\avg{\rho^{-1}}\,\gamma+\avg{B_3}\,\avg{\frac{B_3}{\rho}}+
		\avg{B_2}\,\avg{\frac{B_2}{\rho}}+\avg{B_1}\,\avg{\frac{B_1}{\rho}}}
	-\sqrt{-2\,\sqrt{\avg{B_1}\,\avg{\frac{B_1}{\rho}}\,\avg{p}\,
			\avg{\rho^{-1}}\,\gamma}+\avg{p}\,\avg{\rho^{-1}}\,\gamma+
		\avg{B_3}\,\avg{\frac{B_3}{\rho}}+\avg{B_2}\,\avg{\frac{B_2}{\rho}}+
		\avg{B_1}\,\avg{\frac{B_1}{\rho}}}}{2}\\[0.2cm]
\avg{u}+\frac{\sqrt{2\,\sqrt{\avg{B_1}\,
				\avg{\frac{B_1}{\rho}}\,\avg{p}\,\avg{\rho^{-1}}\,\gamma}+\avg{p}
			\,\avg{\rho^{-1}}\,\gamma+\avg{B_3}\,\avg{\frac{B_3}{\rho}}+
			\avg{B_2}\,\avg{\frac{B_2}{\rho}}+\avg{B_1}\,\avg{\frac{B_1}{\rho}}}
		-\sqrt{-2\,\sqrt{\avg{B_1}\,\avg{\frac{B_1}{\rho}}\,\avg{p}\,
				\avg{\rho^{-1}}\,\gamma}+\avg{p}\,\avg{\rho^{-1}}\,\gamma+
			\avg{B_3}\,\avg{\frac{B_3}{\rho}}+\avg{B_2}\,\avg{\frac{B_2}{\rho}}+
			\avg{B_1}\,\avg{\frac{B_1}{\rho}}}}{2}\\[0.2cm]
	\avg{u}+\frac{\sqrt{2\,\sqrt{\avg{B_1}\,
				\avg{\frac{B_1}{\rho}}\,\avg{p}\,\avg{\rho^{-1}}\,\gamma}+\avg{p}
			\,\avg{\rho^{-1}}\,\gamma+\avg{B_3}\,\avg{\frac{B_3}{\rho}}+
			\avg{B_2}\,\avg{\frac{B_2}{\rho}}+\avg{B_1}\,\avg{\frac{B_1}{\rho}}}
		+\sqrt{-2\,\sqrt{\avg{B_1}\,\avg{\frac{B_1}{\rho}}\,\avg{p}\,
				\avg{\rho^{-1}}\,\gamma}+\avg{p}\,\avg{\rho^{-1}}\,\gamma+
			\avg{B_3}\,\avg{\frac{B_3}{\rho}}+\avg{B_2}\,\avg{\frac{B_2}{\rho}}+
			\avg{B_1}\,\avg{\frac{B_1}{\rho}}}}{2}\\[0.2cm]
	\avg{u}-\sqrt{\avg{B_1}\,\avg{\frac{B_1}{\rho}}}\\[0.2cm]
	\avg{u}+\sqrt{\avg{B_1}\,\avg{\frac{B_1}{\rho}}}\\[0.2cm]
	\avg{u}\\[0.2cm]
	\avg{u}\\[0.2cm]
	\end{bmatrix}.$}
\end{equation}
After many algebraic manipulations and introducing the new average wave speeds \eqref{eq:app:characteristicSpeedsDiscrete}, we find the greatly simplified form of the discrete eigenvalues
given in \eqref{eq:easyidealMHDeigenvaluesDiscrete}. The discrete diagonal matrix $\hat{\boldsymbol{\Lambda}}$ follows immediately.
\end{proof}

\subsection{Summary of the discrete entropy analysis}\label{sec:summaryDiscEnt}

The derivations and proofs in this section were quite technical. Therefore, we found it appropriate to collect the main results in a short summary for easier reference. From the derivations performed we have constructed a fully discrete and unique evaluation for the KEPES numerical flux function \eqref{minimalDiss} given by
\begin{equation}\label{eq:entropystableflux_final}
	\fhat^\mathrm{KEPES} = \fhat^\mathrm{KEPEC} -\half \R{}|\boldsymbol{\widehat{\Lambda}}|\T{}\R{}^T\jump{\vec{v}},
\end{equation}
where the KEPEC flux is given by \eqref{Eq:entropyconservative-yEKEP}, the discrete right eigenvector matrix $\R{}$ is found in \eqref{eq:rightEigen}, the discrete scaling matrix $\T{}$ is \eqref{eq:scalingMatDiscrete}, the discrete eigenvalues matrix $\boldsymbol{\widehat{\Lambda}}$ \eqref{eq:discreteEigenvaluesMat}, and the jump in the entropy variables $\jump{\vec{v}}$ \eqref{eq:jumpEntVars}.

So, we can now describe the complete FV update for each cell in the numerical approximation to be of the form (in one spatial dimension)
\begin{equation}
	\overline{\vec{q}}_i^{n+1} = \overline{\vec{q}}_i^n - \frac{\Delta t}{\Delta x} \left(\fhat_{i+\halb}^\mathrm{KEPES} - \fhat_{i-\halb}^\mathrm{KEPES}\right) + \vec{s}_i,
\end{equation}
where $\fhat^\mathrm{KEPES}$ is given by \eqref{eq:entropystableflux_final} and the source term discretization $\vec{s}_i$ \eqref{SourceTermDiscyEKEP}.

The numerical discretization and proofs provided throughout this section creates an approximation for the ideal MHD equations that is conservative in the hydrodynamic variables and discretely obeys the correct entropic behavior predicted by the continuous hyperbolic PDEs, a fact we demonstrate in Sec. \ref{sec:algProps}.

We close our derivation by looking back at the three laws of thermodynamics defined at the beginning of Sec.~\ref{sec:physics} and scrutinize how they are incorporated into the numerical approximation we derived:
\begin{enumerate}
	\item[] \textbf{First law of thermodynamics: Energy can neither be created nor destroyed.}\par
	We fulfill this law by having the total energy of the system be a conserved quantity that can only change due to fluxes. The corresponding entry in the source term is zero. Hence, mathematically, there is no way that energy could be either created or destroyed in a closed box simulation.\\[0.1cm]
	\item[] \textbf{Second law of thermodynamics: Entropy is either conserved or increases over time. It cannot shrink.}\par
	We have made sure that our model guarantees the correct sign in the entropy evolution. Entropy can only be conserved (KEPEC scheme) or ``generated'' (KEPES scheme, i.e.\ entropy conserving scheme + dissipation term). However, entropy can never be ``destroyed,'' as the dissipation scheme is of quadratic form as shown in \eqref{signSatified2}.\\[0.1cm]
	\item[] \textbf{Third law of thermodynamics: It is impossible for any process to cool a system to absolute zero in a finite number of steps.}\par
	From the ideal gas law, we have $T \propto p$. If we assume $\rho = \mathrm{const.}$, we can re-write \eqref{eq:entropy} to isolate the dependence of the entropy on the temperature,
	\begin{equation*}
		S(\vec{q}) \propto \ln(T),
	\end{equation*}
	where we used the physical sign convention of an increasing entropy function. It can easily be seen that for vanishing temperature, the entropy diverges:
	\begin{equation*}
		\lim_{T \rightarrow 0} S(\vec{q}) \propto \lim_{T \rightarrow 0} \ln(T) = - \infty.
	\end{equation*}
	As our scheme ensures that the entropy evolution is well-behaved, it is clear that $T=0$ cannot be reached in any finite amount of time.
\end{enumerate}

Thus, the \textbf{Mission} we gave at the end of Sec. \ref{sec:physics} is complete!

\section{Numerical results}\label{sec:numExp}

We apply the described entropy stable scheme to several numerical examples. Section \ref{sec:software} provides a brief discussion of the implementation and its place in a state-of-the-art numerical framework. Next, Sec. \ref{sec:algProps} verifies the primary results of this work, e.g., that the baseline FV scheme is entropy conservative (from Sec. \ref{sec:ECIdealMHD}) and improved robustness of the discrete evaluation for the dissipation operator (from Sec. \ref{sec:discEvaluation}). Finally, we apply the entropy stable numerical solver to explore its utility for several standard test problems for the ideal MHD equations in Sec. \ref{sec:stdProblems}.

\subsection{A fairly brief discussion of software}\label{sec:software}

For completeness, we provide a few details regarding the implementation of the KEPES FV numerical approximation. This is partially motivated due to the fact that this aspect of CFD can feel overwhelming because there are so many choices. Thus, we note again, that the discussion is limited to a very small sector of the computational mathematics community.

We implement the three-dimensional FV solver in the FORTRAN programming language. As mentioned earlier, we are most interested in the numerical scheme to reach a broad user base and maintain the ability to model a large range of physical phenomena. Therefore, we choose to implement the entropy stable framework into the multi-scale multi-physics simulation code \texttt{FLASH} \cite{FLASH2009,FLASH2000}. A chief motivation of this choice is because \texttt{FLASH} is publicly available (\texttt{http://flash.uchicago.edu}) and has a wide international user base. Complete details on the implementation of the entropy stable ideal MHD solver is provided by the authors in \cite{Derigs2016}.

We know that complex flow phenomena, particularly in multiple spatial dimensions, exhibit a large range of spatial and temporal scales. Increasing the accuracy in time might be straightforward, e.g., selecting a higher order explicit time integration scheme. The increase in spatial accuracy has more features. We outlined how to increase spatial accuracy through the use of sophisticated reconstruction techniques in Sec.~\ref{sec:finiteVolume}. However, accuracy can also be increased by adding more cells into the computational grid. Obviously, it is best to do this grid refinement locally and automatically near the flow complexities which may develop, e.g., vortical structures, and de-refine the grid where the flow is less complex, to reduce the overall computational cost. The simulation code we are using, \texttt{FLASH}, is an \textit{adaptive mesh refinement} (AMR) code that uses the \texttt{PARAMESH} library \cite{PARAMESH2006}. Here the grid is organized in a block structured, oct-tree adaptive grid, supporting local refinement and de-refinement as can be seen, e.g.,~in Fig.~\ref{fig:Windtunnel} \cite{Derigs2016}.

A final aspect of the numerical approximation is how to determine the time step for an explicit time integration scheme such that the method is stable and convergent \cite{leveque2007}. The selection of the time step for hyperbolic systems is equation dependent as it involves the fastest wave speed. In order for the results of the numerical approximation to remain relevant we select the time step using the \textit{Courant, Friedrichs, and Lewy (CFL) condition} \cite{courant1928,courant1967,lax1967}. One version of the general CFL condition for the three spatial dimension approximation on Cartesian meshes is given by \cite{Derigs2016} 
\begin{equation}\label{eq:CFL_timestep_3D}
\Delta t \le \mathtt{CFL} \cdot \min\bigg[\frac{\Delta x}{\lambda_{\mathrm{max}}^x}, \frac{\Delta y}{\lambda_{\mathrm{max}}^y}, \frac{\Delta z}{\lambda_{\mathrm{max}}^z}\bigg].
\end{equation}
where $\lambda^d_\mathrm{max}$ is the largest wave speed at the current time traveling in the $d = \{x,y,z\}$ direction and $\mathtt{CFL}\in(0,1]$ is an adjustable coefficient. If $\lambda_\mathrm{max}$ is known exactly, then the choice $\mathtt{CFL} = 1.0$ may be adequate for stability \cite[pg.~222]{toro2009}. However, $\lambda_\mathrm{max}$ is usually computed in some approximate way. For example, the discrete eigenvalues of the ideal MHD system discussed in Thm. \ref{thm:eig}. Thus, a more conservative choice for the CFL coefficient is typically used in practice. For example, in the numerical results discussed in Secs. \ref{sec:algProps} and \ref{sec:stdProblems} we select  $\mathtt{CFL} = 0.8$ unless stated otherwise. 

The version of \texttt{FLASH} on hand is 4.4 as of 31\textsuperscript{st} October, 2016.

\subsection{Numerical validation of algorithmic properties}\label{sec:algProps}

We first verify the properties of the numerical algorithm discussed in Sec. \ref{sec:discEntAnalysis}. This includes the verification of the order of the finite volume method in Secs.~\ref{sec:spatialConvergence}. The main focus of this work is on mathematical entropy analysis. We next verify that the approximation satisfies entropy conservation on the discrete level in Sec.~\ref{sec:entropyConvergence}. Finally, we examine the robustness of the approximation and justify the complicated averaging used in the construction of the dissipation operator from Sec.~\ref{scn:WindTunnel}. 

\subsubsection{Verification of spatial convergence}\label{sec:spatialConvergence}

The smooth Alfv\'en wave test \cite{toth2000} is used to compare the accuracy of MHD schemes for smooth flows. The initial circularly polarized Alfv\'en wave propagates across a periodic domain \cite{Derigs2016}.
The Alfv\'en wave speed in $x$-direction is $|v_A|=B_1/\sqrt{\rho} = 1$ and thus, the wave is expected to return to its initial state at each time $t \in \mathbb{N}$. The initial conditions listed in Table~\ref{tab:smoothAlfven} ensure that the magnetic pressure is constant.

\begin{table}[h]
	\centering
	\begin{minipage}[t]{0.51\textwidth}
		\begin{tabular}[t]{|l|l|}\hline
			Density	$\rho$	&	1.0	\\ \hline
			Pressure $p$		&	0.1\\\hline
			Velocity $\vec{u}$	& $\left(0 , 0.1 \sin(2\pi x) , 0.1 \cos(2\pi x)\right)^\intercal$\\\hline
			Mag.~field $\vec{B}$& $\left(1 , 0 , 0\right)^\intercal + \vec{u}$\\\hline
		\end{tabular}
	\end{minipage}
	\hspace*{1mm}
	\begin{minipage}[t]{0.35\textwidth}
		\begin{tabular}[t]{|l|l|}\hline
			Domain size &$x_\mathrm{min} = 0.0$ \\
			&$x_\mathrm{max} = 1.0$ \\ \hline
			Boundary conditions & periodic\\\hline
			Simulation end time & $t_\mathrm{max} = 1.0$ \\\hline
			Adiabatic index & $\gamma = 5/3$ \\\hline
		\end{tabular}
	\end{minipage}
	\caption{Initial conditions and runtime parameters: Smooth Alfv\'en wave test \cite{Derigs2016}.}
	\label{tab:smoothAlfven}
\end{table}

For sufficiently smooth fields, i.e., in cases where discontinuous features are absent, the reconstruction technique used is designed to achieve more than first-order accuracy. To test the accuracy of our scheme, we run several simulations with varying resolutions and compute the $L_1$ and $L_2$ errors for the quantity $B_\perp = B_2$. We fix the time step to a very small value ($\Delta t = 1 \times 10^{-5}$) in order to diminish the influence of the time integration scheme on this test. The obtained errors are plotted as a function of the number of grid points in logarithmic scale in Fig.~\ref{fig:Alfven1DErrors} and are listed in Table~\ref{tab:Alfven1D}.

\begin{figure}[h]
	\centering
	\includegraphics[scale=1]{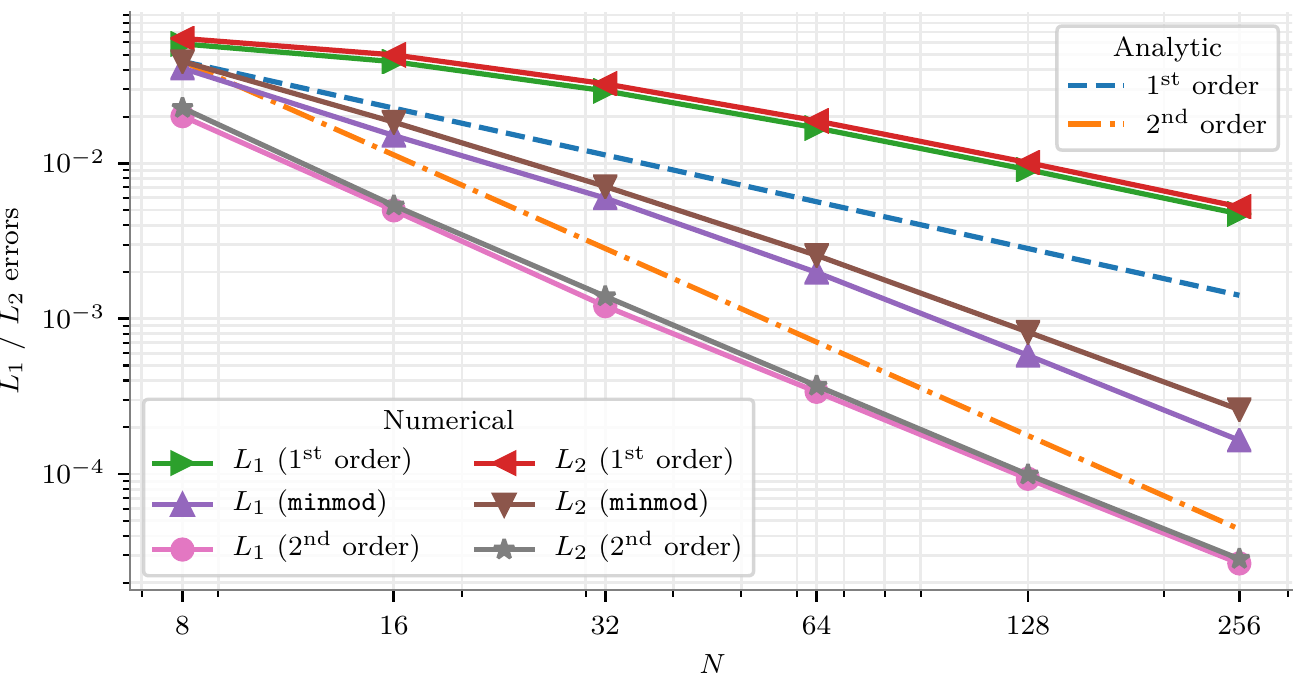}%
	\caption{Spatial accuracy test: $L_1$ and $L_2$ errors measured with the smooth Alfv\'en wave test (cf.~Table \ref{tab:Alfven1D}).}
	\label{fig:Alfven1DErrors}
\end{figure}

\begin{table}[h]
	\small
	\centering
	\sisetup{table-format=1.1e1,table-column-width=0mm}
	\begin{tabular}[t]{lc|SSSSSS}
		\toprule
		&N & {8} & {16} & {32} & {64} & {128} & {256}\\
		\midrule
		\vspace*{-1em}
		\parbox[t]{0mm}{\multirow{4}{*}{\rotatebox[origin=c]{90}{\scriptsize 1\textsuperscript{st} order}}}&&&&&&&\\
		&$L_1$ error & 5.9e-2 & 4.5e-2 & 2.9e-2 & 1.7e-2 & 9.1e-3 & 4.7e-3 \\
		&$L_2$ error & 6.4e-2 & 5.9e-2 & 3.2e-2 & 1.9e-2 & 1.0e-2 & 5.2e-3 \\
		&$L_1$ EOC & N/A & {0.4} & {0.6} & {0.8} & {0.9} & {1.0} \\
		&$L_2$ EOC & N/A & {0.4} & {0.6} & {0.8} & {0.9} & {1.0} \\
		\midrule
		\vspace*{-1em}
		\parbox[t]{0mm}{\multirow{4}{*}{\rotatebox[origin=c]{90}{\scriptsize \texttt{minmod}}}}&&&&&&&\\
		&$L_1$ error & 4.1e-02 & 1.5e-2 & 5.9e-3 & 1.9e-3 & 5.8e-4 & 1.6e-4 \\
		&$L_2$ error & 4.5e-02 & 1.8e-2 & 7.1e-3 & 2.6e-3 & 8.2e-4 & 2.6e-4 \\
		&$L_1$ EOC & N/A & {1.4} & {1.3} & {1.6} & {1.8} & {1.8} \\
		&$L_2$ EOC & N/A & {1.3} & {1.4} & {1.5} & {1.6} & {1.7} \\
		\midrule
		\vspace*{-1em}
		\parbox[t]{0mm}{\multirow{4}{*}{\rotatebox[origin=c]{90}{\scriptsize 2\textsuperscript{nd} order}}}&&&&&&&\\
		&$L_1$ error & 2.0e-2 & 4.9e-3 & 1.2e-3 & 3.3e-4 & 9.3e-5 & 2.7e-5 \\
		&$L_2$ error & 2.3e-2 & 5.3e-3 & 1.4e-3 & 3.7e-4 & 9.9e-5 & 2.8e-5 \\
		&$L_1$ EOC & N/A & {2.0} & {2.0} & {1.8} & {1.9} & {1.8} \\
		&$L_2$ EOC & N/A & {2.0} & {2.0} & {1.9} & {1.9} & {1.8} \\
		\bottomrule
	\end{tabular}
	\caption{Spatial accuracy test: Computed errors and experimental order of convergence (EOC) for $B_2$ after one oscillation of the Alfv\'en wave ($t=1.0$).}
	\label{tab:Alfven1D}
\end{table}

We run tests with three different choices for the interface quantities, $\overline{\vec{q}}^n$:
\begin{description}
	\item[1\textsuperscript{st} order] Piecewise constant interface values \eqref{eq:piecewise-const}.
	\item[2\textsuperscript{nd} order] Piecewise linear interface values \eqref{eq:piecewise-linear} with \eqref{eq:alpha05}.
	\item[\texttt{minmod}]  limited interface values \eqref{eq:piecewise-linear} with \eqref{eq:minmod}.
\end{description}
As can be seen, the standard finite volume scheme is at most first-order accurate in space. With second order reconstruction, close to second-order accuracy is achieved already at a coarse resolution. Given a sufficient numerical resolution ($N \gtrapprox 64$), the \texttt{minmod} limited interface values converge to higher accuracy. We do not expect pure second-order accuracy to be achieved using \texttt{minmod} reconstruction in this test, as the slope limiter will always interpret the extrema of the setup Alfv\'en wave as discontinuities and drop to first-order accuracy at these points. However, this effect becomes less significant the higher the numerical resolution. Note that we used 2\textsuperscript{nd} order reconstructed interface values in this test as we knew that the solution will always stay smooth. However, we remind the reader that it cannot be used in the general case as discontinuous solutions may develop even from smooth initial data.

\subsubsection{Multi-dimensional entropy conservation test (2D)}\label{sec:entropyConvergence}


The mathematical entropy conservation obtained in Sec. \ref{sec:ECIdealMHD} is in the semi-discrete sense. That is, the discrete entropy is conserved up to the errors introduced by the temporal approximation. Hence, we can use the error in the conservation of the total entropy with respect to the chosen time step size as a measurement for the temporal discretization error. We use three time integration schemes of increasing order for this test case: 
\begin{description}
	\item[1\textsuperscript{st} order] explicit Euler.
	\item[2\textsuperscript{nd} order] strong stability preserving Runge-Kutta \cite{Gottlieb2001}.
	\item[3\textsuperscript{rd} order] strong stability preserving Runge-Kutta (SSP RK).
\end{description}
We note that SSP RK time integration schemes are weighted convex combinations of explicit Euler steps.

On the basis of all these considerations, we measure the temporal accuracy and independently confirm the order of each time integration schemelisted above. In order to do so, we disable the entropy stabilization, i.e., the resulting scheme is entropy conserving. As our test of choice, we run the two-dimensional version of the Brio and Wu magnetohydrodynamical shock tube problem \cite{brio1988} with varying fixed time step lengths $\Delta t$.
This test includes discontinuities, a magnetic field and is performed in multiple dimensions and starts from discontinuous initial conditions. Hence, it utilizes the full set of features of the entropy aware scheme we derived in this work.
We keep the previously used periodic boundary conditions to eliminate any possible influence from the boundaries of the domain and to ensure that we observe a closed system.
We construct the two-dimensional initial conditions by rotating the one-dimensional conditions (see Table~\ref{tab:BrioWu}) at a $45{}^\circ$ angle. The fluid is initially at rest on either side of the interface.

Note that the entropy conserving scheme cannot describe systems with discontinuities as it cannot add the physically needed dissipation.
We limit the end time step to $t_\mathrm{end} = 0.001$ when the oscillations have not grown too large as to cause numerical instabilities. Our sole intention is to show that even under high stresses our scheme is still capable of conversing the thermodynamic entropy correctly. We point out that the obtained solution is physically not correct by intention.

\begin{table}[h]
	\centering
	\begin{minipage}[t]{0.32\textwidth}
		\begin{tabular}[t]{l|cc}
			&	{$x < x_\mathrm{shock}$}	& {$x \ge x_\mathrm{shock}$}\\
			\midrule
			$\rho$		& 1	& 0.125	\\
			$p$			& 1	& 0.1 \\
			$\vec{u}$	& $\vec{0}$	& $\vec{0}$ \\
			$B_1$		& 0.75	& 0.75 \\
			$B_2$		& 1	& -1 \\
			$B_3$		& 0	& 0 \\
		\end{tabular}\\[.4em]
	\end{minipage}
	\hspace*{5mm}
	\begin{minipage}[t]{0.47\textwidth}
		\setlength\extrarowheight{3pt}
		\begin{tabular}[t]{|l|l|}
			\hline
			Domain size &$x_\mathrm{min} = 0, x_\mathrm{max} = 1$ \\
			\hline
			Initial shock position &$x_\mathrm{shock} = 0.5$ \\
			\hline
			Boundary conditions & all: periodic\\ 
			\hline
			Simulation end time & $t_\mathrm{max} = 0.1$ \\
			\hline
			Adiabatic index & $\gamma = 2.0$ (artificial) \\
			\hline
		\end{tabular} %
	\end{minipage}
	\caption{Initial conditions: Brio and Wu MHD shock tube (1D)}
	\label{tab:BrioWu}
\end{table}

\begin{figure}[!h]
	\centering
	\includegraphics[scale=1]{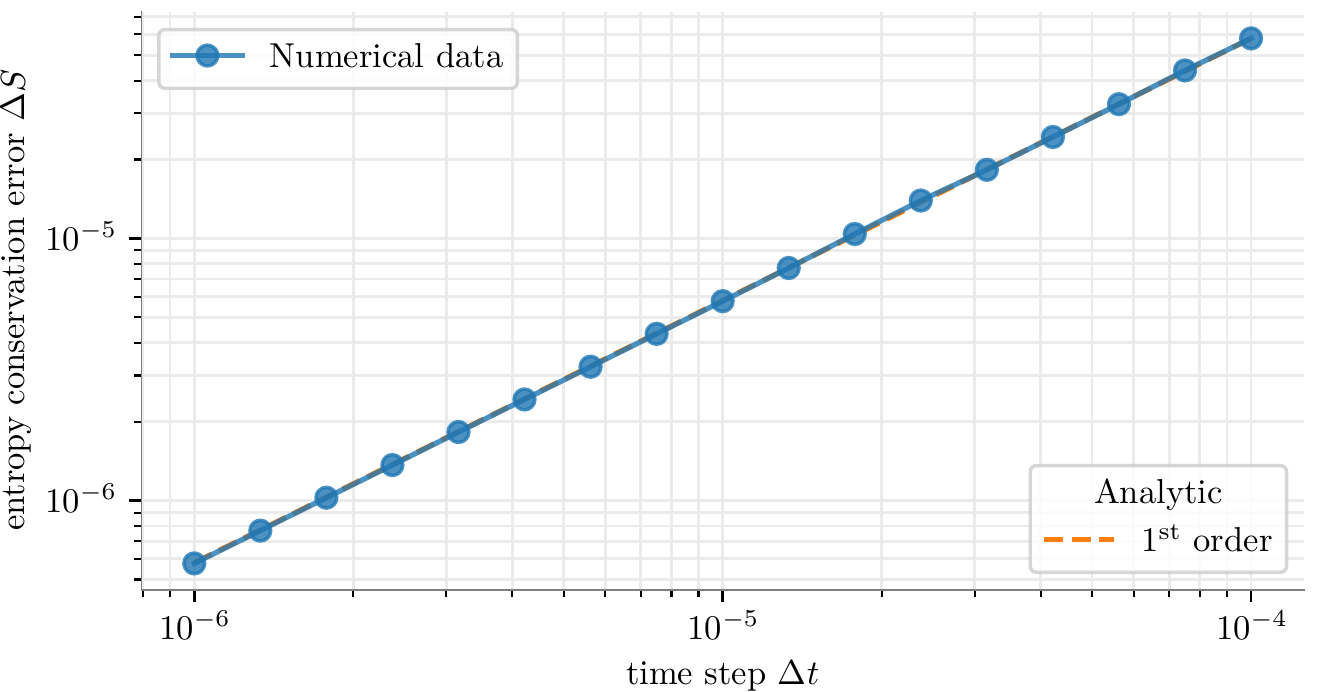}
	\caption{Entropy conservation test, explicit Euler time integration}
	\label{fig:ECtest1st}
\end{figure}

\begin{figure}[!h]
	\centering
	\includegraphics[scale=1]{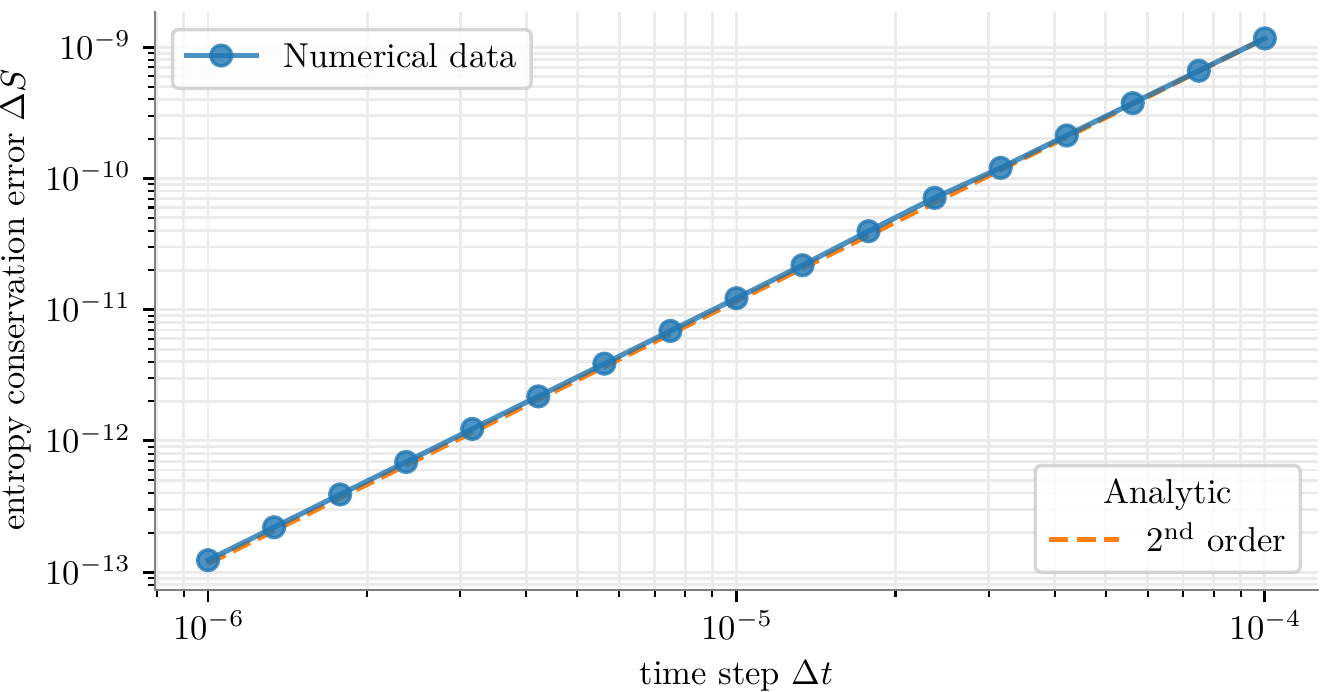}
	\caption{Entropy conservation test, SSP RK2 time integration}
	\label{fig:ECtest2nd}
\end{figure}

\begin{figure}[!h]
	\centering
	\includegraphics[scale=1]{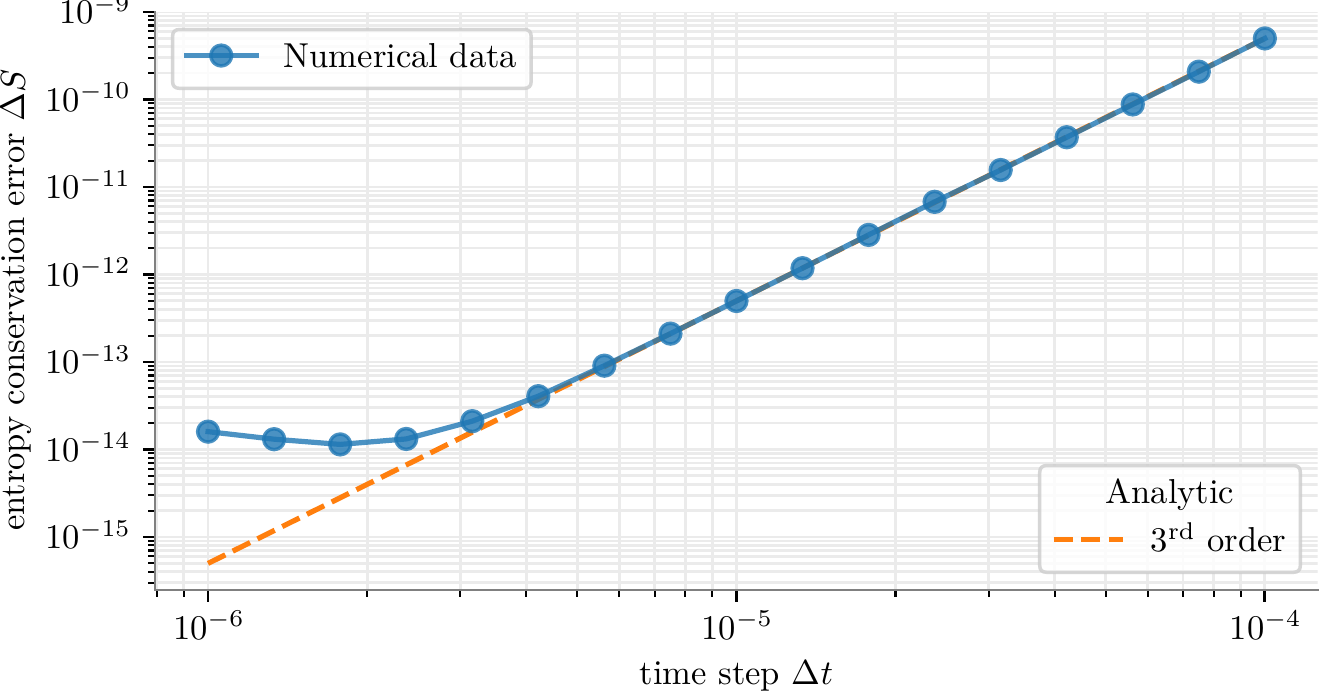}
	\caption{Entropy conservation test, SSP RK3 time integration}
	\label{fig:ECtest3rd}
\end{figure}

We ran a number of simulations, for each time integration scheme, using logarithmically equally spaced time steps and plot the measured entropy conservation error. We see that for the first (Fig.~\ref{fig:ECtest1st}) and second order (Fig.~\ref{fig:ECtest2nd}) temporal approximations, we obtain the expected order of convergence in time. For the third order time integrator (Fig.~\ref{fig:ECtest3rd}) we see that for very fine temporal resolution (i.e. very small time steps) the approximation tapers around $10^{-14}$. This is caused by the finite precision of the computer and hence a natural limit for the accuracy of the entropy conservation in our scheme.

\subsubsection{Wind Tunnel with a Step (2D)}\label{scn:WindTunnel}

The wind tunnel that contains a step was first described in \cite{Emery1968}, who used it to compare several hydrodynamical schemes. Woodward and Colella \cite{Woodward1984} later reused it to compare more advanced methods including their piece-wise parabolic method (PPM) scheme. It exercises the scheme's ability to handle strong unsteady shock interactions in multiple dimensions. Furthermore, it can be used to verify that a code solves problems with non-trivial boundaries correctly. Note, that we include this test to show that our MHD scheme performs well in the purely hydrodynamical limit (i.e.~vanishing magnetic fields). The flow is unsteady and exhibits multiple shock reflections and interactions between different discontinuities.

We use this test case to demonstrate the increased robustness gained from the somewhat unusual average states that we use to evaluate the discrete dissipation operator. If we use a naive evaluation of the discrete operator, for example the simple arithmetic mean for each component, the wind tunnel simulation crashes. This is due, in part, to unphysical mass transfer that stem from the fact that the condition \eqref{eq:almostEqual} is not satisfied in this simpler case.

This problem is computed on a two-dimensional domain with $x\in[0.0,3.0]$ and $y\in[0.0,1.0]$. Between $x=0.6$ and $x=3.0$ we insert a step with height $\Delta y=0.2$. We initialize a uniform Mach 3 flow using $\rho = 1.4$, $p = 1.0$, and $u = 3.0$, with a ratio of specific heats of $\gamma = 1.4$. The front of the step is treated as a reflecting boundary, as are the upper and lower $y$-boundaries. The left-hand $x$-boundary is an inflow boundary constantly feeding in gas with density, pressure, and velocity like given by the initial conditions. The right-hand side uses an outflow (zero-gradient) boundary. Since the exit velocity is always supersonic the exit boundary conditions have no effect on the flow.
%
%
%
\begin{figure}[h]
	\centering
	\centerline{\includegraphics[scale=1]{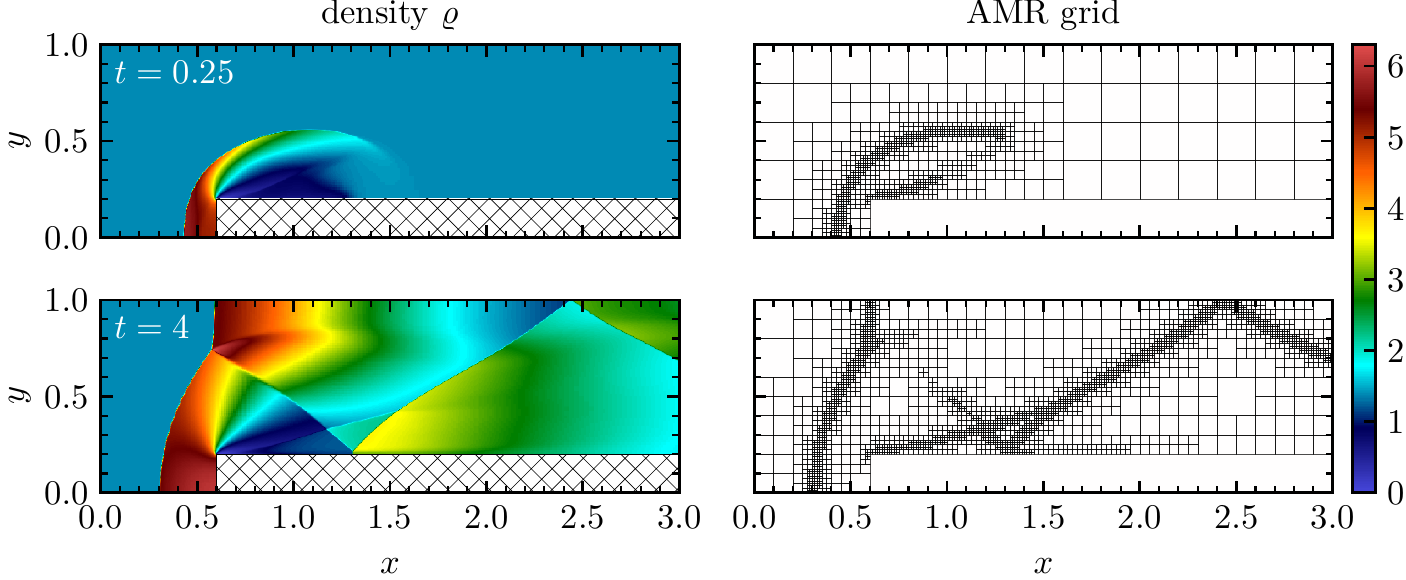}}
	\caption{Wind Tunnel with a step test. Linear plots of the dimensionless density $\rho$. Adaptive grid resolution up to $1920\times640$.}
	\label{fig:Windtunnel}
\end{figure}

Immediately, a shock forms in front of the step and curves around the corner. The general shape and position of the shocks are accurately represented. The shocks are thin, however, numerical instabilities are not observed. The corner of the step evolves into the center of a rarefaction fan and into a singular point of the flow. In contrast to previous works (e.g.~\cite{Woodward1984}), where special boundary conditions near the corner of the step had to be applied, no special boundary condition is applied in our work. The shock curves around the corner of the step and expand to the right (downstream). It grows in size until it strikes the upper reflecting boundary just after about $t=0.55$ and is reflected back to the bottom where it again reflected at about $t=1.1$. Throughout the computation (until $t\approx 12.0$), the flow is unsteady, exhibits multiple shock reflections and interactions between different types of discontinuities.

\subsection{Standard numerical test problems}\label{sec:stdProblems}

The properties of the numerical approximation are now verified and the EC/ES schemes are implemented correctly into \texttt{FLASH}. We next demonstrate the utility, robustness, and accuracy of the new solver by computing the solution to several well-known MHD test problems.

\subsubsection{Orzsag-Tang vortex (2D)}
The Orszag-Tang vortex problem \cite{Orszag1979} is a two-dimensional, spatially periodic problem well suited for studies of MHD turbulence. Thus, it has become a classical test for numerical MHD schemes. It includes dissipation of kinetic and magnetic energy, magnetic reconnection, the formation of high-density jets, dynamic alignment and the emergence and manifestation of small-scale structures \cite{Derigs2016}.
The Orszag-Tang MHD vortex problem starts from non-random, smooth initial data as shown in Table \ref{tab:OrszagTang}. As the flow evolves it gradually becomes increasingly complex, forming intermediate shocks. Thus, this problem demonstrates the transition from initially smooth data to compressible, supersonic MHD turbulence. The initial data is chosen such that the root mean square values of the velocity and the magnetic fields as well as the initial Mach number are all one. 

\begin{table}[h]
	\centering
	\begin{minipage}[t]{0.43\textwidth}
		\begin{tabular}[t]{|l|l|}
			\hline
			Density $\rho$ & $1.0$ \\
			\hline
			Pressure $p$ & $1.0/\gamma$ \\
			\hline
			Velocity $\vec{u}$ & $(-\sin(2\pi y),\,\sin(2\pi x),\,0.0)^\intercal$ \\
			\hline
			Mag.~field $\vec{B}$ & $\frac{1}{\gamma}(-\sin(2\pi y),\,\sin(4\pi x),\,0.0)^\intercal$\\
			\hline
		\end{tabular}
	\end{minipage}
	\hspace*{10mm}
	\begin{minipage}[t]{0.48\textwidth}
		\begin{tabular}[t]{|l|l|}
			\hline
			Domain size &$\{x,y\}_\mathrm{min} = \{0.0,0.0\}$ \\
			&$\{x,y\}_\mathrm{max} = \{1.0,1.0\}$ \\
			\hline
			Boundary conditions & all: periodic\\
			\hline
			Adaptive refinement on	& density, magnetic field \\
			\hline
			Simulation end time & $t_\mathrm{max} = 0.5$ \\
			\hline
			Adiabatic index & $\gamma = 5/3$\\
			\hline
		\end{tabular}
	\end{minipage}
	\caption{Initial conditions and runtime parameters: Orszag-Tang MHD vortex \cite{Derigs2016}.}
	\label{tab:OrszagTang}
\end{table}

Fig.~\ref{fig:OrszagTang} shows the density of the plasma at $t=0.5$ given the initial conditions listed in Table~\ref{tab:OrszagTang}. We superimpose the magnetic field using the technique of line integral convolution (LIC) as this method, in contrast to an ordinary streamline plot, preserves all details of the vector field geometry \cite{Cabral1993}. As the solution evolves in time, the initial vortex splits into two vortices. Sharp gradients accumulate and the vortex pattern becomes increasingly complex due to highly non-linear interactions between multiple intermediate shock waves traveling at different speeds. The result compares very well with results given in the literature, e.g.~\cite{Balbas2005,Dai1998,Londrillo2000}.

\begin{figure}[!ht]
	\centering
	\includegraphics[scale=1]{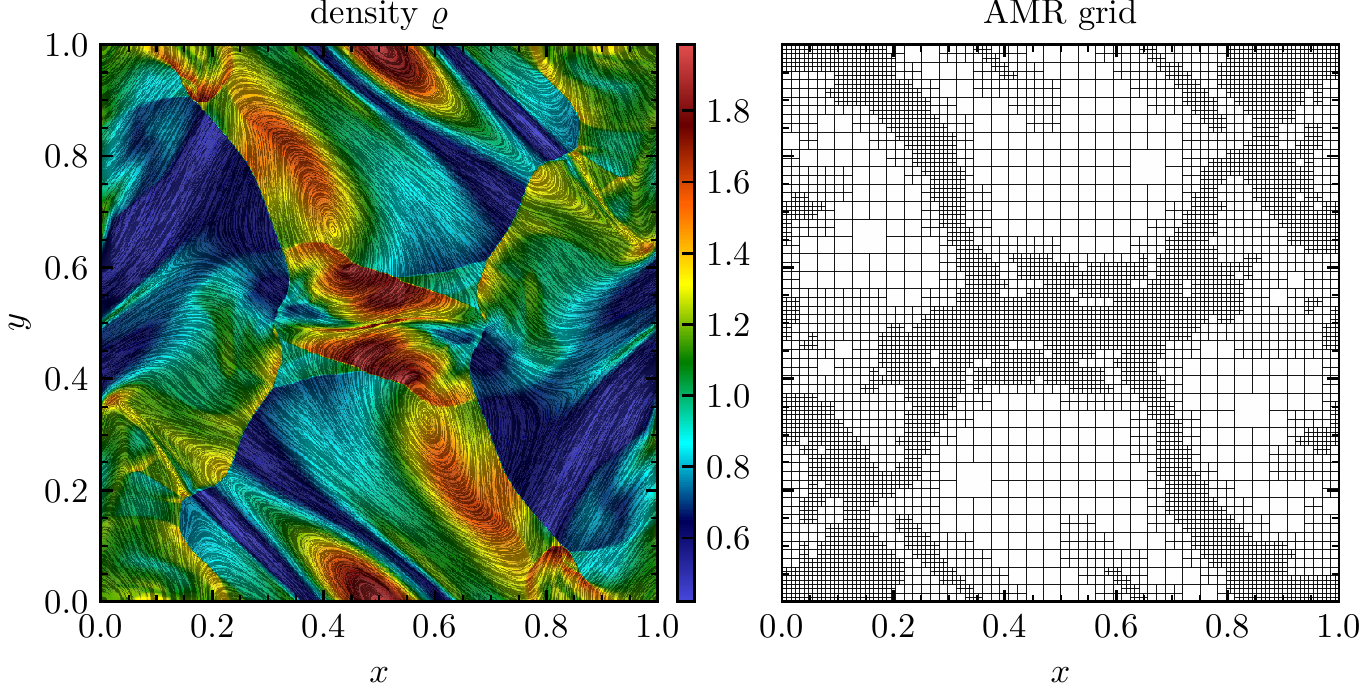}
	\caption{Orszag-Tang MHD vortex test($t=0.15$): Density plots with superimposed magnetic field morphology. The shown AMR grid shows ``blocks'' containing $8\times8$ computational cells each. Adaptive grid resolution up to $1024\times1024$. The density plot can be compared to e.g.~Fig.~10 of \cite{Londrillo2000} and Fig.~14 of \cite{Dai1998}.}
	\label{fig:OrszagTang}
\end{figure}

\subsubsection{MHD rotor test (2D)}\label{MHDRotor}
The MHD rotor problem \cite{Balsara1999} describes a rapidly spinning dense cylinder embedded in a magnetized, homogeneous medium at rest. Due to centrifugal forces, the dense cylinder is in non-equilibrium. As the rotor spins with the given initial rotating velocity, the initially uniform magnetic field will wind up the rotor. The wrapping of the rotor by the magnetic field leads to strong torsional Alfv\'en waves launched into the ambient fluid. The initial conditions are listed in Table~\ref{tab:Rotor}.

\begin{table}[h]
	\centering
	\begin{minipage}[t]{0.41\textwidth}
		\begin{tabular}[t]{l|ccc}
			&	{$r \le r_0$}				& {$r \in (r_0,r_1)$}		& {$r \ge r_1$}\\
			\midrule
			$\rho$		&	$10.0$ 						& $1.0 + 9.0 f(r)$			& $1.0$	\\
			$p$		&	$1.0$						& $1.0$  					& $1.0$\\
			$B_1$		&	$5/\sqrt{4\pi}$				& $5/\sqrt{4\pi}$			& $5/\sqrt{4\pi}$\\
			$B_2$		&	$0.0$						& $0.0$						& $0.0$\\
			$B_3$		&	$0.0$						& $0.0$						& $0.0$\\
			$u$		&	$-20.0 \Delta y$		& $-20.0 f(r) \Delta y$		& $0.0$\\
			$v$		&	$20.0 \Delta x$		& $20.0 f(r) \Delta x$		& $0.0$\\
			$w$		&	$0.0$		& $0.0$		& $0.0$\\
		\end{tabular}\\[.6em]
		with $f(r) = \frac{r_1-r}{r_1-r_0}$, \par $r=\sqrt{(x-x_\mathrm{center})^2+(y-y_\mathrm{center})^2}$, \par
		$\Delta x = (x-x_\mathrm{center})$, $\Delta y = (y-y_\mathrm{center})$
	\end{minipage}
	\hspace*{5mm}
	\begin{minipage}[t]{0.54\textwidth}
		\begin{tabular}[t]{|l|l|}
			\hline
			Domain size &$\{x,y\}_\mathrm{min} = \{0.0,0.0\}$ \\
			&$\{x,y\}_\mathrm{max} = \{1.0,1.0\}$ \\
			\hline
			Inner radius	&  $r_0 = 0.1$ \\
			\hline
			Outer radius	&  $r_1 = 0.115$ \\
			\hline
			$x$-center		& $x_\mathrm{center} = 0.5$ \\
			\hline
			$y$-center		& $y_\mathrm{center} = 0.5$ \\
			\hline
			Boundary conditions & all: zero-gradient (``outflow'') \\ 
			\hline
			Adaptive refinement on	& density, magnetic field \\
			\hline
			%
			Simulation end time & $t_\mathrm{max} = 0.15$ \\
			\hline
			Adiabatic index & $\gamma = 1.4$ \\ 
			\hline
		\end{tabular}
	\end{minipage}
	\caption{Initial conditions and runtime parameters: 2D MHD rotor test \cite{Derigs2016}.}
	\label{tab:Rotor}
\end{table}

At $t=0.15$ the rotor is, up to a certain radial distance, still in uniform rotation. Beyond this radius, the rotor has exchanged momentum with its environment and decelerated. In Fig.~\ref{fig:MHDRotor:2D} we present the magnetic field superimposed on the fluid density, $\rho$. It is clearly seen that the magnetic field basically maintains its initial shape outside of the region of influence of the Alfv\'en waves. Inside, the magnetic field is refracted by the MHD discontinuities. In the center, however, we can clearly see that the magnetic field remains unchanged and only gets rotated together with the flow (the cylinder has rotated almost $180^\circ$). The shown snapshot is taken shortly before the torsional Alfv\'en waves leave the computational domain and demonstrates that our solver is able to resolve torsional Alfv\'en waves.

\begin{figure}[!ht]
	\centering
	\includegraphics[scale=1]{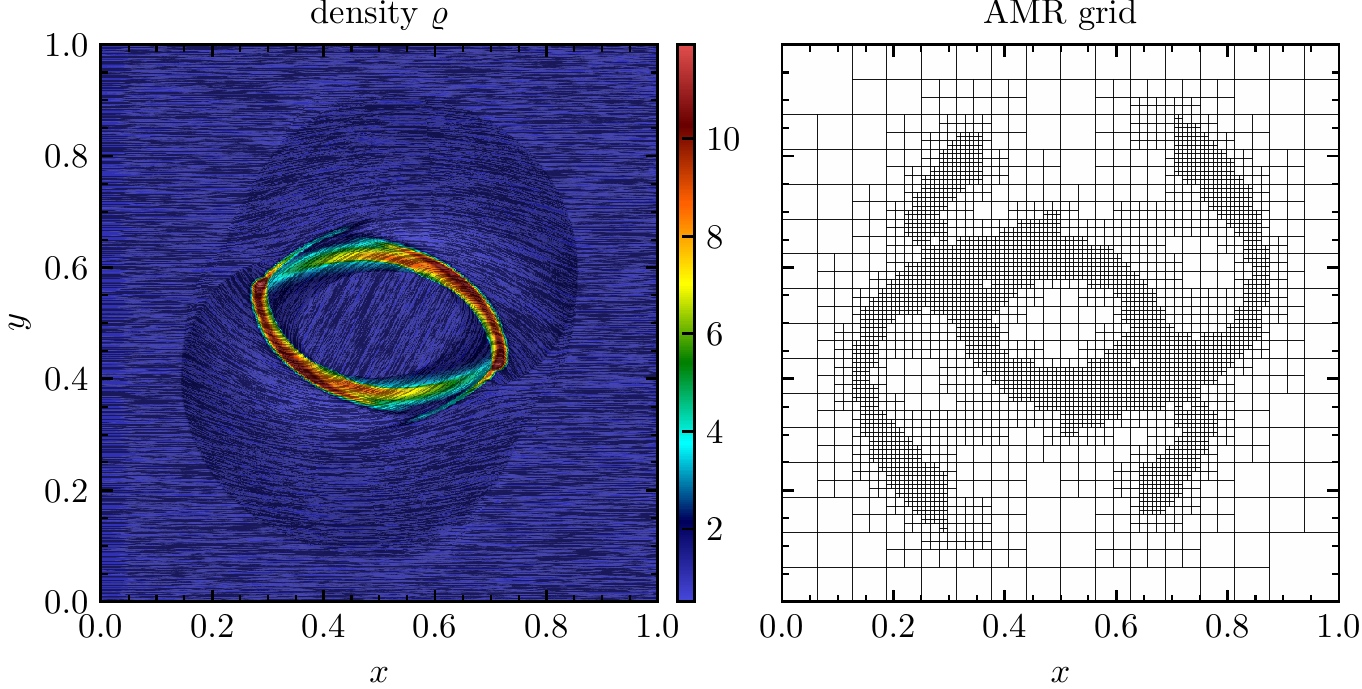}
	\caption{MHD rotor test ($t=0.15$): Density plots with superimposed magnetic field morphology. Adaptive grid resolution up to $1024\times1024$. This plot can directly be compared to Fig.~7 of \cite{Winters2016}, Fig.~14 of \cite{Londrillo2000}, and Fig.~2 of \cite{Balsara1999}.}
	\label{fig:MHDRotor:2D}
\end{figure}

\subsubsection{MHD explosion wave (2D/3D)}
The two-dimensional version of the MHD blast wave problem was studied by \cite{Balsara1999}. We use an extended three-dimensional version to demonstrate the robustness of our scheme in simulations involving regimes with low thermal pressures and high kinetic as well as magnetic energies in three dimensions.
This test problem leads to the onset of strong MHD discontinuities, relevant to astrophysical phenomena where magnetic fields can have strong dynamical effects.
It describes an initially circular pressure pulse. We choose here a relative magnitude of $10^4$ for comparison with \cite{Balsara1999}.
The initial conditions used are listed in Table~\ref{tab:MHDBlast}.

\begin{table}[h]
	\centering
	\begin{tabular}[t]{|l|l|}
		\hline
		Domain size &$\{x,y,z\}_\mathrm{min} = \{\text{-}0.5,\text{-}0.5,\text{-}0.5\}$ \\
		&$\{x,y,z\}_\mathrm{max} = \{0.5,0.5,0.5\}$ \\
		\hline
		Inner radius	&  $r_0 = 0.09$ \\
		\hline
		Outer radius	&  $r_1 = 0.1$ \\
		\hline
		Explosion center		& $\vec{x}_\mathrm{center} = (0.0,0.0,0.0)$\\
		\hline
		Boundary conditions & all: periodic \\ 
		\hline
		Adaptive refinement on	& density, pressure \\
		\hline
		Simulation end time & $t_\mathrm{max} = 0.01$ \\
		\hline
		Adiabatic index & $\gamma = 1.4$ \\ 
		\hline
	\end{tabular}\\[.4em]
	\begin{tabular}[t]{l|ccc}
		&	{$r \le r_0$}				& {$r \in (r_0,r_1)$}		& {$r \ge r_1$}\\
		\midrule
		$\rho$	&	$1.0$ 						& $1.0$						& $1.0$	\\
		$p$		&	$1000.0$					&{$0.1+999.9 f(r)$}  			& $0.1$\\
		$B_1$	&	$\frac{100}{\sqrt{4\pi}}$	& $\frac{100}{\sqrt{4\pi}}$	& $\frac{100}{\sqrt{4\pi}}$\\
		$B_2$	&	$0.0$						& $0.0$						& $0.0$\\
		$B_3$	&	$0.0$						& $0.0$						& $0.0$\\
		$u$		&	$0.0$						& $0.0$						& $0.0$\\
		$v$		&	$0.0$						& $0.0$						& $0.0$\\
		$w$		&	$0.0$						& $0.0$						& $0.0$\\

	\end{tabular}\\[.4em]
	with $r=\sqrt{(x-x_\mathrm{center})^2+(y-y_\mathrm{center})^2}$, \par and $f(r) = \frac{r_1-r}{r_1-r_0}$
	\caption{Initial conditions and runtime parameters: MHD blast wave test \cite{Derigs2016}.}
	\label{tab:MHDBlast}
\end{table}

\begin{figure}[!h]
	\centering
	\includegraphics[scale=1]{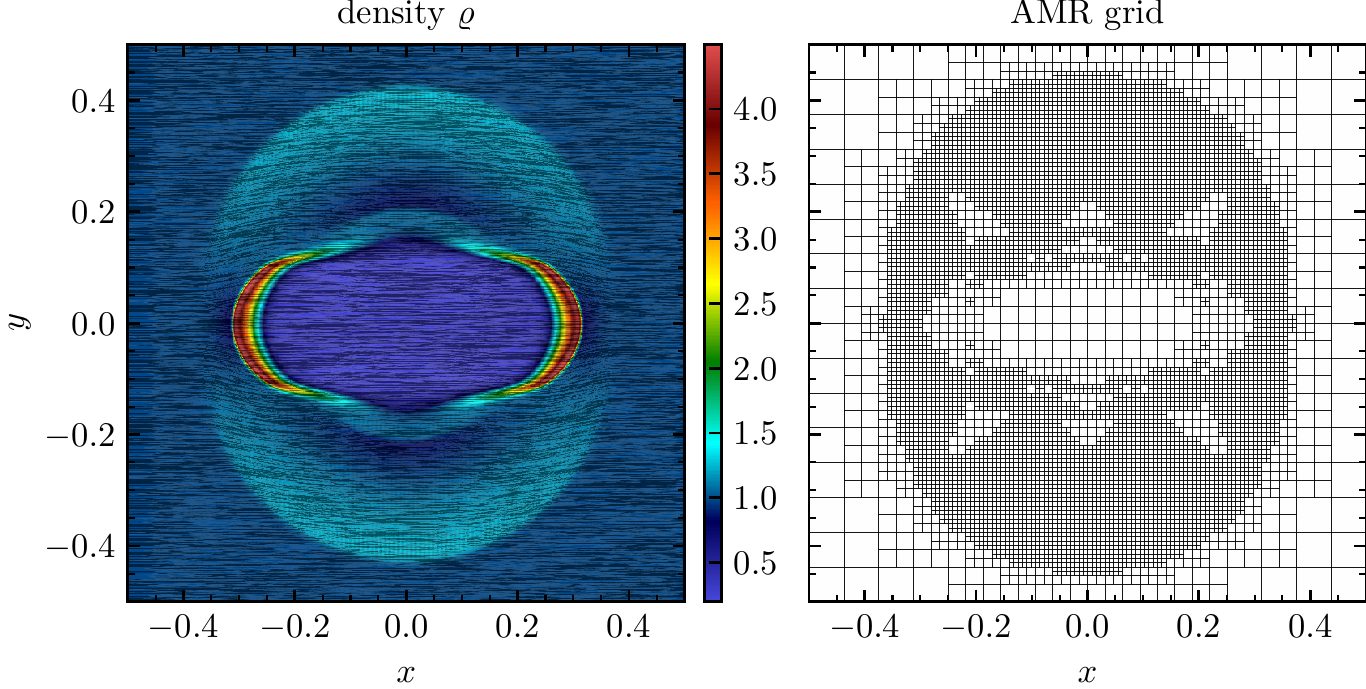}
	\caption{MHD blast wave test ($t=0.01$): Density plots with superimposed magnetic field morphology. The shown AMR grid shows ``blocks'' containing $8\times8$ computational cells each. Adaptive grid resolution up to $1024\times1024$. The plot can be compared to Fig.~13 of \cite{Londrillo2000} as well as Fig.~4 of \cite{Balsara1999}.}
	\label{fig:Blast:2D}
\end{figure}

The MHD explosion wave is highly anisotropic and the explosion bubble is strongly distorted as the displacement of the gas in the transverse $y$ and $z-$direction is inhibited due to the magnetic field (set in the $x-$direction). The out-going blast wave shows no grid alignment effects or other numerical defects as seen for various schemes, e.g.,~in \cite{Derigs2016}.

\begin{figure}[!h]
	\centering
	\includegraphics[width=.6\textwidth]{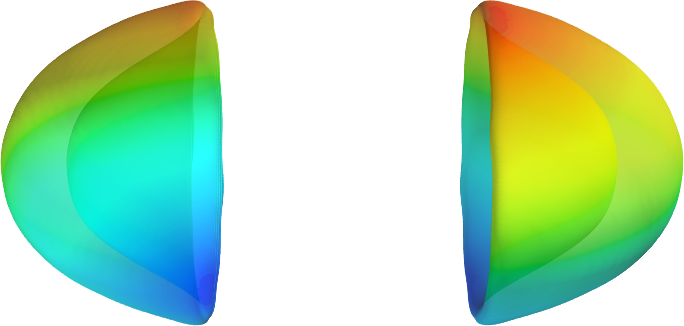}
	\caption{MHD blast wave test ($t=0.01$): Semi-transparent rendering of the surface with constant value of $\varrho=1.2$. Adaptive grid resolution up to $256\times256\times256$. No numerical artifacts are visible and the symmetry (distortion only in the $x$-direction due to the aligned magnetic field) is as expected.}
\end{figure}

\section{Concluding remarks and outlook}\label{sec:conclusions}

In this work, we collected background information and demonstrated the inextricable link between physics and applied and computational mathematics. In particular, we examined the fundamental laws of thermodynamics and built partial differential equations (PDEs) to mathematically model the evolution and propagation of wave phenomena. Of utmost importance was the concept of entropy conservation (or stability) depending on the smoothness of the solution. This is because it is the entropy (introduced by the second law of thermodynamics) that separates possible state configurations of systems from impossible states. However, entropy conservation is almost always an auxiliary conservation law of a given PDE system.

On the continuous level, the link between the primary conserved quantities of a PDE and the conservation of entropy has been well understood since the late 1970s. However, it wasn't until the landmark work of Tadmor in the 1980s that the design of numerical algorithms that could discretely remain entropy conservative was known. It was within this framework, built from a finite volume type discretization, that we described the construction of a baseline entropy conservative numerical flux function for the ideal MHD equations.

A very subtle part of the continuous as well as discrete entropy analysis is the treatment of the divergence-free constraint on the magnetic fields, $\nabla\cdot\vec{B}=0$. Entropy conservation was shown in the continuous case by Godunov. To do so Godunov augmented the ideal MHD system to include a source term proportional to the divergence-free condition (essentially adding zero in a clever way). However, on the discrete level, this addition of zero strategy no longer works. This is due to inherent numerical inaccuracies in the divergence-free constraint for a discretization. Even if the initial flow is divergence-free, as the numerical approximation of the flow evolves inaccuracies can build and drive instabilities. We showed that a particular discretization of the Janhunen source term (similar to the Godunov or Powell source term but remains conservative in the hydrodynamic variables) was enough to guarantee a consistent, entropy conservative numerical flux for ideal MHD.

The PDEs considered are of a hyperbolic nature. This is problematic as discontinuities may develop, e.g., shock waves, in the solution regardless of the initial solution smoothness. Physically, it is known that entropy must be dissipated in the presence of shocks. Therefore, we presented the details of a particular dissipation term to be added to the baseline entropy conservative flux to create an entropy stable approximation. That is, the discrete entropy in the numerical approximation will always possess the correct sign. Furthermore, we demonstrated at which mean state the dissipation term should be evaluated to avoid unphysical transfers in the conserved quantities, e.g., mass.

The final section of this work applied the entropy stable algorithm to five two- and three-dimensional numerical examples. The high Mach number flow in a wind tunnel was a pure Euler calculation, i.e., no magnetic fields. Through the battery of numerical examples considered we demonstrated that the numerical algorithm is able to model the fluid dynamics of a wide range of flows, e.g., low or high Mach number.

Two main avenues guide our future research:
\begin{enumerate}

\item[1.] There are many approaches to address the errors in the divergence-free constraint. The three most popular are projection methods \cite{Brackbill1980,Marder1987}, constrained transport \cite{Balsara1999,Evans1988}, and hyperbolic divergence cleaning \cite{Dedner2002,Tricco2016}. An extensive comparison of projection and constrained transport type methods is provided by T\'{o}th \cite{toth2000}. The most computationally efficient of these choices is hyperbolic divergence cleaning. The idea is to introduce an additional variable proportional to the divergence error. This additional variable is governed by a simple advection equation and coupled into the flux of the magnetic field equations such that errors in the divergence are moved away from where they are generated in the approximation. However, at this time, it is unclear if hyperbolic divergence cleaning can lay ``on top'' of an entropy stable algorithm and remain entropy stable.\\

\item[2.] High-order spatial discretizations for the approximation of the ideal MHD equations that retain entropy stability are also of interest. This is because high-order spatial methods are able to capture complex flow features with fewer degrees of freedom, e.g., \cite{tcfd2012}. The results presented in this work used a simple linear reconstruction with a minmod limiter to increase spatial accuracy to second order. There are higher order reconstructions available in the literature, e.g., \cite{barth1989,roe1986,sweby1984,vanleer1979} but it might be that they fail to preserve the entropy stability we have worked so hard to obtain for the low order approximation. This issue has been previously addressed for general systems of conservation laws and there are arbitrarily high-order entropy stable schemes available \cite{fisher2013_2,Fjordholm2012_2,Fjordholm2016,lefloch2000}. However, all the high-order extensions referenced assume that there was no source term. We demonstrated the importance of the source term discretization for the ideal MHD equations and its crucial role to guarantee discrete entropy conservation. Therefore, an interesting question is to determine how the source term can be discretized at high-order while maintaining entropy consistency.

\end{enumerate}

\section*{Acknowledgement}
Dominik Derigs and Stefanie Walch acknowledge the support of the Bonn-Cologne Graduate School for Physics and Astronomy (BCGS), which is funded through the Excellence Initiative.

Gregor Gassner has been supported by the European Research Council (ERC) under the European Union's Eights Framework Program Horizon 2020 with the research project \textit{Extreme}, ERC grant agreement no. 714487.

Stefanie Walch thanks the Deutsche Forschungsgemeinschaft (DFG) for funding through the SPP 1573 ``The physics of the interstellar medium'' and the European Research Council under the European Community's Framework Programme FP8 via the ERC Starting Grant RADFEEDBACK (project number 679852).

This work has been partially performed using the Cologne High Efficiency Operating Platform for Sciences (\texttt{CHEOPS}) HPC cluster at the Regionales Rechenzentrum K\"{o}ln (RRZK), University of Cologne, Germany. 
Research in theoretical astrophysics is carried out within the Collaborative Research Centre 956, sub-project C5, funded by the Deutsche Forschungsgemeinschaft (DFG).
The software used in this work was developed in part by the DOE NNSA ASC- and DOE Office of Science ASCR-supported FLASH Center for Computational Science at the University of Chicago.


\bibliographystyle{plain}
\bibliography{mybibfile}

\end{document}